\documentclass[twoside,11pt,reqno]{amsart}
\usepackage{amsmath,amssymb,amscd,mathrsfs,epic,wasysym,latexsym,tikz,mathrsfs,cite,hyperref,enumerate}
\usepackage{stmaryrd}
\usepackage{pb-diagram}
\usepackage[matrix,arrow]{xy}
\usepackage{commath}
\usepackage{tikz-cd}
\usepackage{upgreek}

\DeclareFontFamily{OT1}{pzc}{}
\DeclareFontShape{OT1}{pzc}{m}{it}{<-> s * [1.10] pzcmi7t}{}
\DeclareMathAlphabet{\mathpzc}{OT1}{pzc}{m}{it}

\makeatletter

\numberwithin{equation}{section}

\newtheorem{Lemma}[equation]{Lemma}
\newtheorem{Theorem}[equation]{Theorem}
\newtheorem{Corollary}[equation]{Corollary}
\newtheorem{Inductive Assumption}[equation]{Inductive Assumption}
\theoremstyle{definition}  

\newtheorem{Remark}[equation]{Remark}

\newtheorem{Example}[equation]{Example}

\newtheorem{Notation}[equation]{Notation}

\let\<\langle
\let\>\rangle

\newcommand\Comment[2][\relax]{\space\par\medskip\noindent%
   \fbox{\begin{minipage}{\textwidth}\textbf{Comment\ifx\relax#1\else---#1\fi}\newline%
        #2\end{minipage}}\medskip
}

\def\bi{\text{\boldmath$i$}}
\def\bj{\text{\boldmath$j$}}

\def\bk{\text{\boldmath$k$}}

\def\bs{\text{\boldmath$s$}}

\def\bt{\text{\boldmath$t$}}

\def\br{\text{\boldmath$r$}}
\def\b1{\text{\boldmath$1$}}

\def\ba{\text{\boldmath$a$}}
\def\bb{\text{\boldmath$b$}}

\def\bu{\text{\boldmath$u$}}

\def\m{\mathfrak{m}}

\def\pmod#1{\text{ }(\text{\rm mod } #1)\,}

\newcommand{\Hom}{\operatorname{Hom}}

\newcommand{\End}{\operatorname{End}}

\newcommand{\im}{\operatorname{im}}

\def\sgn{\mathtt{sgn}}

\newcommand{\cha}{\operatorname{char}}

\newcommand{\bideg}{\operatorname{bideg}}

\def\Blo{\mathcal{B}}

\def\cont{{\operatorname{cont}}}
\def\core{{\operatorname{core}}}
\newcommand{\cus}{{\operatorname{cus}}}

\newcommand{\Z}{\mathbb{Z}}
\newcommand{\N}{\mathbb{N}}
\newcommand{\K}{\mathbb{K}}
\newcommand{\F}{\mathbb{F}}

\newcommand{\0}{{\bar 0}}
\renewcommand{\1}{{\bar 1}}
\def\eps{{\varepsilon}}
\def\phi{{\varphi}}

\newcommand{\zx}{{\mathsf{x}}}

\newcommand{\zc}{{\mathsf{c}}}
\newcommand{\zv}{{\mathsf{v}}}
\newcommand{\zz}{{\mathsf{z}}}
\newcommand{\ze}{{\mathsf{e}}}
\newcommand{\zf}{{\mathsf{f}}}

\newcommand{\zs}{{\mathsf{s}}}
\newcommand{\za}{{\mathsf{a}}}
\newcommand{\zb}{{\mathsf{b}}}
\newcommand{\zu}{{\mathsf{u}}}
\newcommand{\zw}{{\mathsf{w}}}
\newcommand{\zE}{{\mathsf{E}}}
\newcommand{\zB}{{\mathsf{B}}}

\newcommand{\zH}{{\mathsf{H}}}

\newcommand{\cC} {\mathcal{C}}
\newcommand{\cY} {\mathcal{TC}}
\newcommand{\cH} {\mathrm{He}}

\newcommand{\zV}{{\mathsf{V}}}

\newcommand{\funF}{{\mathpzc F}}

\newcommand{\cT} {\mathcal{T}}
\newcommand{\ct} {\mathpzc{t}}

\def\ttl{{\mathtt l}}

\newcommand{\di}{{\operatorname{div}}}

\newcommand{\super}{\mathop{\tt s}\nolimits}
\newcommand{\swr}{\wr_{\super}}

\newcommand{\gsM}{\mathop{\rm gsMor}\nolimits}
\newcommand{\sM}{\mathop{\rm sMor}\nolimits}

\newcommand{\ga}{\gamma}
\newcommand{\Ga}{\Gamma}
\newcommand{\la}{\lambda}
\newcommand{\La}{\Lambda}
\newcommand{\al}{\alpha}
\newcommand{\be}{\beta}

\def\Si{\mathfrak{S}}
\def\Ai{\mathfrak{A}}

\newcommand{\om}{\omega}
\newcommand{\Om}{\Omega}

\newcommand{\de}{\delta}

\newcommand{\ka}{\kappa}

\def\Mtype{\mathtt{M}}
\def\Qtype{\mathtt{Q}}

\newcommand{\Irr}{{\mathrm {Irr}}}

\newcommand{\funQ}{{\mathpzc{Q}}}

\newcommand{\Mor}{{\mathrm {Mor}}}

\def\rank{\mathop{\mathrm{ rank}}\nolimits}

\newcommand{\pr}{{\mathrm {pr}}}

\newcommand{\C}{{\mathbb C}}

\newcommand{\Cent}{Z}
\newcommand{\zC}{{\mathsf{C}}}
\newcommand{\zX}{{\mathsf{X}}}
\newcommand{\zF}{{\mathsf{F}}}

\newcommand{\EC}{{\mathcal E}}

\renewcommand{\mod}{\bmod \,}

\def\K{\mathbb K}

\newcommand{\Zig}{{\mathsf{A}}}
\newcommand{\zA}{{\mathsf{A}}}

\newcommand{\ttA}{{\tt A}}

\newcommand{\tti}{{\mathtt i}}

\def\lgathz{\mathsf{g}}
\def\lgath{g}

\def\col{{\operatorname{col}}}

\def\g{{\mathfrak g}}

\def\Par{{\mathcal P}}

\def\Weights{{\mathscr{W}}}
\def\Nuclei{{\mathscr{N}}}

\def\b{\mathfrak{b}}
\def\k{\Bbbk}

\def\Comp{\Uplambda}
\def\EC{{\mathcal{C}}}

\def\height{{\operatorname{ht}}}

\def\wt{{\operatorname{wt}}}

\def\sop{{\mathrm{sop}}}
\def\re{{\mathrm{re}}}
\def\im{{\mathrm{im}\,}}
\def\onto{{\twoheadrightarrow}}
\def\into{{\hookrightarrow}}

\def\mod#1{#1\!\operatorname{-mod}}

\def\iso{\stackrel{{}_\sim}{\longrightarrow}}

\def\ch{\operatorname{ch}}
\def\lan{\langle}
\def\ran{\rangle}

\def\Seq{\operatorname{Tri}}

\def\parity{\operatorname{par}}

\newcommand{\cc} {\mathpzc{c}}
\newcommand{\cm} {\mathpzc{m}}

\def\t{{\mathsf t}}

\def\ggw{\text{\boldmath$l$}}
\def\ggwhat{\text{\boldmath$g$}}

\newcommand{\ggi}{{f}}
\newcommand{\ggis}{\mathsf{f}}

\def\bla{\text{\boldmath$\lambda$}}
\def\bmu{\text{\boldmath$\mu$}}

{\catcode`\|=\active
  \gdef\set#1{\mathinner{\lbrace\,{\mathcode`\|"8000%
  \let|\midvert #1}\,\rbrace}}
}
\def\midvert{\egroup\mid\bgroup}

\colorlet{darkgreen}{green!50!black}
\tikzset{dots/.style={very thick,loosely dotted},
         greendot/.style={fill,circle,color=darkgreen,inner sep=1.5pt,outer sep=0},
         blackdot/.style={fill,circle,color=black,inner sep=1.1pt,outer sep=0},
         graydot/.style={fill,circle,color=gray,inner sep=1.1pt,outer sep=0},
         reddot/.style={fill,circle,color=red,inner sep=1.1pt,outer sep=0},
         bluedot/.style={fill,circle,color=blue,inner sep=1.1pt,outer sep=0}
}
\def\greendot(#1,#2){\node[greendot] at(#1,#2){}}
\def\blackdot(#1,#2){\node[blackdot] at(#1,#2){}}
\def\graydot(#1,#2){\node[graydot] at(#1,#2){}}
\def\reddot(#1,#2){\node[reddot] at(#1,#2){}}
\def\bluedot(#1,#2){\node[bluedot] at(#1,#2){}}

\newenvironment{braid}{
  \begin{tikzpicture}[baseline=6mm,black,line width=.7pt, scale=0.32,
                      draw/.append style={rounded corners},
                      every node/.append style={font=\fontsize{5}{5}\selectfont}]%
  }{\end{tikzpicture}
}

\def\Grid(#1,#2){
  \draw[very thin,gray,step=2mm] (0,0)grid(#1,#2);
  \draw[very thin,darkgreen,step=10mm] (0,0)grid(#1,#2);
}

\newcommand\Tableau[2][\relax]{
  \begin{tikzpicture}[scale=0.5,draw/.append style={thick,black}]
    \ifx\relax#1\relax%
    \else 
      \foreach\box in {#1} { \filldraw[blue!30]\box+(-.5,-.5)rectangle++(.5,.5); }
    \fi
    \newcount\row\newcount\col
    \row=0
    \foreach \Row in {#2} {
       \col=1
       \foreach\k in \Row {
          \draw(\the\col,\the\row)+(-.5,-.5)rectangle++(.5,.5);
          \draw(\the\col,\the\row)node{\k};
          \global\advance\col by 1
       }
       \global\advance\row by -1
    }
  \end{tikzpicture}
}

\newcommand\YoungDiagram[2][\relax]{
  \begin{tikzpicture}[scale=0.5,draw/.append style={thick,black}]
    \ifx\relax#1\relax%
    \else 
    \foreach\box in {#1} {
      \filldraw[blue!30]\box rectangle ++(1,1);
    }
    \fi
    \newcount\row
    \row=0
    \foreach \col in {#2} {
       \draw(1,\the\row)grid ++(\col,1);
       \global\advance\row by -1
    }
  \end{tikzpicture}
}

\newenvironment{Young}{\begingroup
       \def\vr{\vrule height0.89\hoogte width\dikte depth 0.2\hoogte}
       \def\fbox##1{\vbox{\offinterlineskip
                    \hrule height\dikte
                    \hbox to \breedte{\vr\hfill##1\hfill\vr}
                    \hrule height\dikte}}
       \vbox\bgroup \offinterlineskip \tabskip=-\dikte \lineskip=-\dikte
            \halign\bgroup &\fbox{##\unskip}\unskip  \crcr }
       {\egroup\egroup\endgroup}
\def\diagram#1{\relax\ifmmode\vcenter{\,\begin{Young}#1\end{Young}\,}\else%
              $\vcenter{\,\begin{Young}#1\end{Young}\,}$\fi}

\begin{document}

\title[RoCK blocks for double covers]{{RoCK blocks of double covers of symmetric groups and generalized Schur algebras}}

\author{\sc Alexander Kleshchev}
\address{Department of Mathematics\\ University of Oregon\\ Eugene\\ OR 97403, USA}
\email{klesh@uoregon.edu}


\thanks{The author was supported by the NSF grant DMS-2346684. }

\begin{abstract}
We study blocks of the double covers of symmetric and alternating groups. The main result is a `local' description, up to Morita equivalence, of arbitrary defect RoCK blocks of these groups in terms of generalized Schur superalgebras corresponding to an explicit Brauer tree superalgebra. In view of the recent results on 
Brou\'e's Conjecture for these groups, our result provides a `local'  description of an arbitrary block of the double covers of symmetric and alternating groups up to derived equivalence. 
\end{abstract}

\maketitle

\section{Introduction}

\subsection{Double covers and generalized Schur algebras}

Let $\tilde\Si_n$ be a double cover of the symmetric group $\Si_n$, and $\tilde\Ai_n$ be the double cover of the alternating group $\Ai_n$. In this paper, we are concerned with the blocks of $\tilde\Si_n$ and $\tilde\Ai_n$. Our main result is a `local' description, up to Morita superequvalence, of arbitrary defect RoCK blocks of these groups of $\bar p$-weight $d$   in terms of generalized Schur superalgebras $T^{\Zig_\ell}(n,d)$ corresponding to an explicit Brauer tree superalgebra $\Zig_\ell$. 
The analogous results for the symmetric groups were obtained in \cite{EK2}. In view of the recent results \cite{ELV,BK} on 
Brou\'e's Conjecture for $\tilde \Si_n$ and $\tilde\Ai_n$, our result provides a `local'  description of an arbitrary block of these groups  up to derived equivalence.

Let $\F$ be an algebraically closed field of odd characteristic $p$. 
We have a canonical central element $z\in\tilde\Si_n$ and a central idempotent $e_z:=(1-z)/2\in\F\tilde\Si_n.$ This yields an ideal decomposition $\F\tilde\Si_n=\F\tilde\Si_ne_z\oplus \F\tilde\Si_n(1-e_z).$  
So the blocks of $\F\tilde\Si_n$ split into blocks of $\cT_n:=\F\tilde\Si_ne_z$ and the blocks of $\F\tilde\Si_n(1-e_z)$. 
Note that $\F\tilde\Si_n(1-e_z)\cong\F\Si_n$. A `local'  description of the blocks of $\F\Si_n$ in terms of certain generalized Schur algebras was conjecture by Turner \cite{Turner} and established in \cite{EK2}. So in this paper we concentrate on the blocks of $\cT_n$ sometimes referred to as the {\em spin blocks of the symmetric group $\Si_n$}. 

The algebra $\cT_n$ comes with a natural $\Z/2$-grading $\cT_n=(\cT_n)_\0\oplus (\cT_n)_\1$ with $(\cT_n)_\0=\F\Ai_ne_z$, and the blocks of $(\cT_n)_\0$ are the {\em spin blocks of the alternating group $\Ai_n$}. As $\cT_n$ is a superalgebra, we can also consider the {\em superblocks} of $\cT_n$, i.e. indecomposable superideals of $\cT_n$. It is well-known that the superblocks are indecomposable as usual ideals (i.e. superblocks are blocks) unless we deal with trivial defect blocks. Trivial defect blocks are simple algebras and are completely elementary, thus working with superblocks vs. blocks does not create any serious problems.

Superblocks of $\cT_n$ are labeled by the pairs $(\rho,d)$, where $\rho$ is a $\bar p$-core partition and $d$ is a non-negative integer such that $|\rho|+dp=n$. We denote by  $\Blo^{\rho,d}$ the block of $\cT_n$ corresponding to such a pair $(\rho,d)$. The integer $d$ is then referred to as the {\em $\bar p$-weight} of  $\Blo^{\rho,d}$. Denoting by $h(\rho)$ the number of non-zero parts of the partition $\rho$, the {\em parity} of  $\Blo^{\rho,d}$ is defined as 
$\parity(\Blo^{\rho,d}):=|\rho|-h(\rho)+d\pmod{2}.$ The even (resp.  odd) parity blocks are the ones whose irreducible supermodules have type $\tt M$ (resp. type $\tt Q$), see \cite[\S12.2, Theorem~22.3.1]{Kbook}.

The case $d=0$ corresponds to the defect zero situation, so we often assume that $d>0$. The even part $\Blo^{\rho,d}_\0$ is then a spin block of $\tilde\Ai_n$. The defect group of $\Blo^{\rho,d}$ (which is the same as the defect group of $\Blo^{\rho,d}_\0$) depends only on $d$ and is abelian if and only if $d<p$. 

The following was conjectured by Schaps and Kessar, cf. \cite{AriSch} and proved in \cite{ELV,BK}: 

\vspace{2.5mm}
\noindent
{\bf Theorem (Ebert-Lauda-Vera and Brundan-Kleshchev).} 
{\em 
Let $\rho,\rho'$ be $\bar p$-cores and $d,d'\in\Z_{\geq 0}$. 

\begin{enumerate}
\item[{\rm (i)}] $\Blo^{\rho,d}$ and $\Blo^{\rho',d'}$ are derived equivalent if and only if $d=d'$ and $\parity(\Blo^{\rho,d})=\parity(\Blo^{\rho',d'})$. 
\item[{\rm (ii)}] $\Blo^{\rho,d}_\0$ and $\Blo^{\rho',d'}_\0$ are derived equivalent if and only if $d=d'$ and $\parity(\Blo^{\rho,d})=\parity(\Blo^{\rho',d'})$.
\item[{\rm (iii)}] $\Blo^{\rho,d}$ and $\Blo^{\rho',d'}_\0$ are derived equivalent if and only if $d=d'$ and $\parity(\Blo^{\rho,d})\neq\parity(\Blo^{\rho',d'})$.
\end{enumerate}
}

\vspace{2mm}

In \cite{KlLi}, we have defined the notion of a {\em RoCK spin block}. To be more precise, for every $d\in\Z_{\geq 0}$, we define {\em $d$-Rouquier $\bar p$-cores}; then, if $\rho$ is a $d$-Rouquier $\bar p$-core, the spin blocks $\Blo^{\rho,d}$  and $\Blo^{\rho,d}_\0$ (of symmetric and alternating groups,  respectively) are called {\em RoCK}. Importantly, for every $d$ and every $\eps\in\Z/2$ there exists a $d$-Rouquier $\bar p$-core $\rho$ with $\parity(\rho)=\eps$. So, in view of the theorem above, 
every spin block of a symmetric or alternating group is derived equivalent to a RoCK spin block. Therefore, to get a `local' description of an arbitrary spin block of a symmetric or alternating group up to derived equivalence, it suffices to get a `local'  description of an arbitrary RoCK spin block.

Let $A_\ell$ be the (graded) Brauer tree (super)algebra defined in (\ref{EAEll}). Denote by $\cC_1$ the {\em rank $1$ Clifford superalgebra}, defined in Example~\ref{ExCliffird}. For superalgebras  $A$ and $B$, we always denote by $A\otimes B$ their {\em tensor product as superalgebras}, and by $A \swr \Si_d$ the {\em wreath superproduct} defined in (\ref{EWdA}). In \cite[Proposition 5.4.10]{KlLi}, the following Morita (super)equivalences were established for the (non-trivial) abelian defect case $d<p$:

\vspace{2.5mm}
\noindent
{\bf Theorem (Kleshchev-Livesey).} 
{\em 
Let $\Blo^{\rho,d}$ be a RoCK spin block with $d<p$. 
\begin{enumerate}
\item[{\rm (i)}] If $\parity(\Blo^{\rho,d})$ is even, then $\Blo^{\rho,d} \sim_{\sM} A_\ell \swr \Si_d$ and $\Blo^{\rho,d}_\0 \sim_{\Mor}(A_\ell \swr \Si_d)_\0$. 
\item[{\rm (ii)}] If $\parity(\Blo^{\rho,d})$ is odd, then $\Blo^{\rho,d} \sim_{\sM} (A_\ell \swr \Si_d)\otimes \cC_1$ and $\Blo^{\rho,d}_\0 \sim_{\Mor}A_\ell \swr \Si_d$. 
\end{enumerate}
}

\vspace{2mm}

As explained in \cite{KlLi}, this can be considered a local description of an {\em abelian defect} RoCK block. 

The case of an arbitrary defect $d$ is much more subtle. First of all we need a {\em regrading} $\Zig_\ell$ of the graded superalgebra $A_\ell$, defined in \S\ref{SSBrTree}. The wreath superproduct $A_\ell\swr \Si_d$ is isomorphic to a regrading of the wreath superproduct $\Zig_\ell\swr \Si_d$, see \S\ref{SSRegrRem}, therefore  $A_\ell\swr \Si_d\sim_{\sM}\Zig_\ell\swr \Si_d$. So in the abelian defect case, treated in the previous theorem, we could use either $A_\ell$ or $\Zig_\ell$. However, for the arbitrary defect case the distinction between $A_\ell$ and $\Zig_\ell$ becomes crucial. 

In \cite{KM2}, certain {\em generalized Schur algebras} $T^A(n,d)$ were studied. These generalize Turner's construction \cite{Turner} studied in \cite{EK1}. Let $T^{\Zig_\ell}(n,d)$ be the {\em generalized Schur (graded super)algebra corresponding to $A=\Zig_\ell$}, see \S\ref{SST}. One can think of it informally as an `upgrade' of the wreath superproduct $\Zig_\ell\swr \Si_d$; in fact, in general $\Zig_\ell\swr \Si_d$ is an idempotent truncation of $T^{\Zig_\ell}(d,d)$, and $T^{\Zig_\ell}(d,d)\sim_{\sM} \Zig_\ell\swr \Si_d$ if and only if $d<p$. Note also that for a fixed $d$, the graded superalgebras $T^{\Zig_\ell}(n,d)$ are graded Morita superequivalent for all $n\geq d$, so it will suffice to consider only $T^{\Zig_\ell}(d,d)$. Our first main result is (see Theorem~\ref{TAMainBody}):

\vspace{2mm}
\noindent
{\bf Theorem A.} 
{\em 
Let $\Blo^{\rho,d}$ be a RoCK spin block. 
\begin{enumerate}
\item[{\rm (i)}] If $\parity(\Blo^{\rho,d})$ is even, then $\Blo^{\rho,d} \sim_{\sM} T^{\Zig_\ell}(d,d)$ and $\Blo^{\rho,d}_\0 \sim_{\Mor}T^{\Zig_\ell}(d,d)_\0$. 
\item[{\rm (ii)}] If $\parity(\Blo^{\rho,d})$ is odd, then $\Blo^{\rho,d} \sim_{\sM} T^{\Zig_\ell}(d,d)\otimes \cC_1$ and $\Blo^{\rho,d}_\0 \sim_{\Mor}T^{\Zig_\ell}(d,d)$. 
\end{enumerate}
}
\vspace{2mm}

To prove Theorem~A, we relate the blocks of $\tilde\Si_n$ to  cyclotomic quiver Hecke superalgebras following the work of Kang-Kashiwara-Tsuchioka \cite{KKT} and then study RoCK blocks in the setting of cyclotomic quiver Hecke superalgebras. 

\subsection{Cyclotomic quiver Hecke superalgebras}
Let $\ell=(p-1)/2$ and $R_\theta$ be the {\em quiver Hecke superalgebra} of Lie type $\ttA_{\ell}^{\hspace{-.3mm}{}^{(2)}}$ corresponding to an element $\theta$ of the non-negative 
part of the root lattice of Lie type $\ttA_{\ell}^{\hspace{-.3mm}{}^{(2)}}$. (In the main body of the paper we work in a more general situation, see \S\ref{SSDefQHSA}, in particular a ground field of characteristic $0$ is allowed and $\ell$ does not have to be of the form $\ell=(p-1)/2$ for a prime $p$.) 

Let $H_\theta$ be the {\em cyclotomic quotient} of $R_\theta$ corresponding to the fundamental dominant weight $\La_0$, see \S\ref{SSHTheta}. We have $H_\theta\neq 0$ if and only if $\theta\in \{\La_0-w\La_0+d\de\mid w\in W,\, d\in\Z_{\geq 0}\}$,  where $W$ is the Weyl group and $\de$ is the null root, see \S\ref{SSW}. Any $\theta\in  \{\La_0-w\La_0+d\de\mid w\in W,\, d\in\Z_{\geq 0}\}$ can be written in the from $\theta=\rho(\theta)+d(\theta)\de$ for unique $\rho(\theta)\in \{\La_0-w\La_0\mid w\in W\}$ and  unique $d(\theta)\in\Z_{\geq 0}$. 

We call $H_\theta$ a {\em RoCK block} (of cyclotomic quiver Hecke superalgebras) if $(\theta\mid \al_0^\vee)\geq 2d(\theta)$ and $(\theta\mid \al_i^\vee)\geq d(\theta)-1$ for $i=1,\dots,\ell$, see \S\ref{SRock}. This Lie-theoretic condition is equivalent to the combinatorial condition considered in \cite{KlLi}. 
Theorem A is deduced from the following theorem (see Corollary~\ref{CMainField}):

\vspace{2mm}
\noindent
{\bf Theorem B.} 
{\em 
Let $H_\theta$ be a RoCK block and $d=d(\theta)$. 
If $n\geq d$ then the graded superalgebras $H_{\theta}$ and $T^{\Zig_\ell}(n,d)$ are graded Morita superequivalent. 
}
\vspace{2mm}

Let $H_\theta$ be a RoCK block with $\rho=\rho(\theta)$ and $d=d(\theta)$. To prove Theorem~B, we construct an explicit idempotent $e\in H_\theta$ such that the idempotent truncation $X_{\rho,d}:=e H_\theta e$ is (graded) Morita (super)equivalent to $H_\theta$. The construction of the idempotent $e$ involves crucially the so-called Gelfant-Graev idempotents, see \S\ref{SSGGW}. 

A subtle point is that we then need to consider an explicit regrading $\zX_{\rho,d}$ of $X_{\rho,d}$. Thus $\zX_{\rho,d}=X_{\rho,d}$ as algebras, but the degree and parity of elements is changed. It is easy to see that such algebras 
$X_{\rho,d}$ and $\zX_{\rho,d}$
are graded Morita superequivalent, see \S\ref{SSRegr}. 

We use Gelfand-Graev idempotents 
$\bar \ggis^\bla$ corresponding to $J$-multicompositions $\bla\in\Comp^J(n,d)$ to 
construct a projective generator 
$$
\Ga_{\rho,d,n}:=\bigoplus_{\bla\in\Comp^J(n,d)}\zX_{\rho,d}\bar \ggis^\bla.
$$
Theorem B then follows from the following (see Theorem~\ref{TEIsoT}):  

\vspace{2mm}
\noindent
{\bf Theorem C.} 
{\em 
Let $H_\theta$ be a RoCK block with $\rho=\rho(\theta)$ and $d=d(\theta)$.  If $n\geq d$ then there is an isomorphism of graded superalgebras $\End_{\zX_{\rho,d}}(\Ga_{\rho,d,n})^\sop\cong  T^{\Zig_\ell}(n,d)$. 
}
\vspace{2mm}

To prove Theorem C, we need to work `integrally', i.e.  over an appropriate DVR $\k$. Working over $\k$, we proceed in the following steps: 

(1) For a special Gelfand-Graev idempotent $\bar \ggis_{\om_d}$ we have an explicit graded superalgebra isomorphism 
$
\Xi_{\rho,d}:\Zig_\ell\swr\Si_d\iso \bar \ggis_{\om_d}\zX_{\rho,d}\bar \ggis_{\om_d}, 
$
see Theorem~\ref{TXiIso}. 

(2) For $\bla\in\Comp^J(n,d)$, the idempotent truncation $\bar \ggis^{\bla}\zX_{\rho,d}\bar \ggis_{\om_d}$ is naturally a graded  right $\bar\ggis_{\om_d}\zX_{\rho,d}\bar \ggis_{\om_d}$-supermodule. In view of the isomorphism in (1), it can be considered as a   graded  right supermodule over the wreath superproduct $\Zig_\ell\swr\Si_d$. In Theorem~\ref{TIdentifyM}, we identify $\bar \ggis^{\bla}\zX_{\rho,d}\bar \ggis_{\om_d}$ with a certain `permutation module' $M_{\bla}$ over $\Zig_\ell\swr\Si_d$. 

(3) In view of (2), we get an explicit homomorphism (see (\ref{EPhi})):
$$
\Upphi_{\rho,d,n}:\zE_{\rho,d,n}=\bigoplus_{\bla,\bmu\in\Comp^J(n,d)}\bar \ggis^\bmu\zX_{\rho,d}\bar \ggis^\bla\to S^{\Zig_\ell}(n,d):=\End_{\Zig_\ell\swr\Si_d}\bigg(\bigoplus_{\bla\in\Comp^J(n,d)} M_\bla\bigg).
$$

(4) We use the faithfulness of the 
$\zX_{\rho,d}$-module $\zX_{\rho,d}\bar \ggi_{\om_d}$ to prove that $\Upphi_{\rho,d,n}$ is injective, 
see Corollary~\ref{C8.11} and Lemma~\ref{C8.19}. 

(5) We use the identification in (2), we show that the image $\Upphi_{\rho,d,n}(\zE_{\rho,d,n})$ contains certain explicit endomorphisms of the form $\tti^\bla(\zx)$ for $\bla\in\Comp^J(n-1,d-1)$ and $\zx\in \Zig_\ell$, see Lemma~\ref{L8.21}. 

(6) We use the theory of {\em imaginary cuspidal representations} of the cyclotomic quiver Hecke algebras developed in \cite{KIS}, especially Theorem~\ref{TGath2}, to show that $\Upphi_{\rho,d,n}(\zE_{\rho,d,n})$ contains the degree zero component $S^{\Zig_\ell}(n,d)^0$, see Lemma~\ref{L8.20}. 

(7) We use the fact established in Corollary~\ref{CGen2} that, under the assumption $d\leq n$, the $\k$-subalgebra $T^{\Zig_\ell}(n,d)\subseteq S^{\Zig_\ell}(n,d)$ is generated by the degree zero component $S^{\Zig_\ell}(n,d)^0$ and the set $\{\tti^\bla(\zx)\mid \zx\in \Zig_\ell,\,\bla\in\Comp^J(n-1,d-1)\}$. Together with (5) and (6), this implies that $\Upphi_{\rho,d,n}(\zE_{\rho,d,n})$ contains $T^{\Zig_\ell}(n,d)$. 

(8) To show that the image $\Upphi_{\rho,d,n}(\zE_{\rho,d,n})$ is exactly $T^{\Zig_\ell}(n,d)$ and not larger, we bring in a symmetricity argument. First, we show that the algebra $\Upphi_{\rho,d,n}(\zE_{\rho,d,n})\cong \zE_{\rho,d,n}$ is symmetric (extending scalars to a field), see Lemma~\ref{L8.22}. Then we apply the main result of \cite{KlSym} showing that $T^{\Zig_\ell}(n,d)$ is a {\em maximal} symmetric subalgebra of $S^{\Zig_\ell}(n,d)$.

\section{Preliminaries}
\subsection{General conventions}
\label{SSGenConv}
We denote $\N:=\Z_{\geq 0}$, $\N_+:=\Z_{> 0}$, $\Z/2=\Z/2\Z=\{\0,\1\}$. 

We will work over a principal ideal domain $\k$ such that $2$ is invertible in $\k$. Additional assumptions will eventually be made on $\k$ when necessary. 


Throughout the paper we fix 
$\ell\in\N_+$ and denote 
\begin{equation}\label{EIJ}
I := \{0,1,\dots,\ell\}\quad\text{and}\quad J:=\{0,1,\dots,\ell-1\}.
\end{equation}
We eventually identify $I$ with the set of vertices of a certain Dynkin diagram, see \S\ref{SSLTN}. 

For $ d\in \N$ and a set $S$, we write elements of $S^ d$ as $\bs=(s_1,\dots,s_ d)=s_1\cdots s_ d$, with $s_1,\dots,s_ d\in S$, and refer to them as {\em words with entries in $S$}. 
We use the notation $s^d:=s\cdots s\in S^d$ for $s\in S$. 
Given words $\bs\in S^ d$ and $\bt\in S^c$, we have the concatenation word $\bs\bt:=s_1\cdots s_ dt_1\cdots t_c\in S^{d+c}$.  


\subsection{Partitions and compositions}
\label{SSPar}

A {\em composition} is an infinite sequence $\la=(\la_1,\la_2,\dots)$ of non-negative integers which are eventually zero. 
Any composition $\la$ has finite sum $|\la|:=\la_1+\la_2+\dots$, and we say that $\la$ is a composition of $|\la|$. We write $\Comp$ for the set of all compositions, and for each $d\in\N$ we write $\Comp(d)$ for the set of all compositions of $d$. When writing compositions, we may collect consecutive equal parts together with a superscript, and omit an infinite tail of $0$'s. For example, for $d\in\N_+$, we have the composition 
\begin{equation}\label{EOmd}
\om_d:=(1^d)\in \Comp(d).
\end{equation}
For $n,d\in\N$, we set
$$
\Comp(n,d):=\{\la=(\la_1,\la_2,\dots)\in\Comp(d)\mid \la_k=0
\ \text{for all $k>n$}\}.
$$
If $\la\in\Comp(n,d)$ we always write $\la=(\la_1,\dots,\la_n)$. 

Let $\Comp^J( d)$ denote the set of all \emph{$J$-multicompositions} of $d$. So the elements of $\Comp^J(d)$ are  the tuples 
$\bla=(\la^{(0)},\dots,\la^{(\ell-1)})$ of compositions satisfying $\sum_{j\in J}|\la^{(j)}|= d$. We set
$$
\Comp^J(n,d):=\{\bla\in \Comp^J(d)\mid \la^{(j)}_k=0\ \text{for all $j\in J$ and $k>n$}\}.
$$
We have a bijection
\begin{equation}\label{EBijMultiComp}
\al:\Comp^J(n,d)\iso \Comp(n\ell,d),\ \bla\mapsto (\la^{(0)}_1,\la^{(1)}_1,\dots,\la^{(\ell-1)}_1,\dots, \la^{(0)}_n,\la^{(1)}_n,\dots,\la^{(\ell-1)}_n).
\end{equation}

A \emph{partition} is a composition whose parts are weakly decreasing. We write 
$\Par(d)$ for the set of partitions of $d$. 
We set 
\begin{align*}
\Comp_+(n,d):=\Comp(n,d)\cap\Par(d).
\end{align*}
Let $\Par^\ell( d)$ denote the set of all \emph{$\ell$-multipartitions} of $d$. We identify $\Par^\ell(d)$ with the set $\Par^J(d)$ of all $J$-multipartitions of $d$, i.e. the set of the $\ell$-tuples $(\la^{(0)},\dots,\la^{(\ell-1)})$ of partitions satisfying  $\sum_{i=0}^{\ell-1}|\la^{(i)}|= d$. 

Define the sets  of {\em colored compositions}
$$
\Comp^\col(n,d):=\{(\la,\bj)\mid \la\in \Comp(n,d),\ \bj\in J^n\},\quad\Comp^\col(d):=\bigsqcup_{n\in\N}\Comp^\col(n,d).
$$
Let $(\la,\bj)\in \Comp^\col(n,d)$. For each $j\in J$, set 
\begin{equation}\label{ELaBjj}
|\la,\bj|_{j}:=\sum_{\substack{1\leq r\leq n\\ j_r=j}}\la_r.
\end{equation}

We denote
$\bk:=0\,1\,\cdots\,(\ell-1)\in J^\ell.$
Then 
$$\bk^n=0\,1\,\cdots\,(\ell-1)\ 0\,1\,\cdots\,(\ell-1)\ \cdots\ 0\,1\,\cdots\,(\ell-1)\in J^{n\ell},$$
We now have the injective map 
\begin{equation}\label{EBeta}
\upgamma^J_{n,d}:\Comp^J(n,d)\to \Comp^\col(n\ell,d),\ \bla\mapsto (\al(\bla),\bk^n).
\end{equation}
We often consider multicompositions in $\Comp^J(n,d)$ as colored compositions in $\Comp^\col(n\ell,d)$ via this injection.

A composition $\la=(\la_1,\la_2,\dots)$ is called {\em essential} if $\la_k=0$ for some $k\in\N_+$ implies that $\la_l=0$ for all $l\geq k$. We denote by $\EC(d)$ the set of all essential compositions of $d$. We also have the set $\EC^\col(d)$ of {\em essential colored composition of $d$}, i.e. those colored compositions $(\la,\bj)$ of $d$ with $\la$ essential.

\subsection{Symmetric group}
\label{SSSG}

We denote by $\Si_ d$  the  symmetric group on $\{1,\dots, d\}$ (with $\Si_0$ interpreted as the trivial group).
Let 
$s_r:=(r,r+1) \in \Si_d$ for $r=1,\dots, d-1$ be the elementary transpositions in $\Si_d$. 
We denote by $\ttl(w)$ the corresponding {\em length} of $w\in \Si_d$ and set $\sgn(w):=(-1)^{\ttl(w)}$. We denote by $w_d$ the longest element of $\Si_d$.

For any $\la=(\la_1,\dots,\la_n)\in\Comp(n, d)$, we have the {\em standard parabolic subgroup} 
$$\Si_\la:= \Si_{\la_1}\times\dots\times \Si_{\la_n}\leq \Si_d.$$
Using the bijection (\ref{EBijMultiComp}), for any $\bla\in\Comp^J(n,d)$, we  have the parabolic subgroup 
\begin{equation}\label{ESiBla}
\Si_\bla:=\Si_{\al(\bla)}.
\end{equation} 

For any set $S$, the group $\Si_d$ acts on $S^d$ on the left by place permutations:  $ g\cdot \bs=s_{ g^{-1}(1)}\cdots s_{ g^{-1}(d)}$ for $ g\in \Si_d$ and $\bs=s_1\cdots s_d\in S^d$.

\section{Graded superalgebras}\label{SSBasicRep}

\subsection{Graded superspaces and superalgebras}
By a graded superspace we understand a $\k$-module with $\k$-module decomposition $V=\bigoplus_{n\in\Z,\,\eps\in\Z/2}V_{n,\eps}$. 
We refer to the elements of $V_{n,\eps}$ as elements of {\em bidegree} $(n,\eps)$. We also refer to $n$ as {\em degree} and $\eps$ as {\em parity}, and write for $v\in V_{n,\eps}$:
$$
\bideg(v)=(n,\eps),\ \deg(v)=n,\ |v|=\eps.
$$
We use the notation $V_\eps:=\bigoplus_{n\in \Z}V_{n,\eps}$ and 
\begin{equation}\label{EA>0}
V^n:= V_{n,\0}\oplus V_{n,\1},\quad V^{>n}:=\bigoplus_{m>n,\,\eps\in\Z/2}V_{m,\eps}.
\end{equation}
A {\em graded subsuperspace} $W\subseteq V$ is a $\k$-submodule such that $W=\sum_{n,\eps}(W\cap V_{n,\eps})$. 

Let $V,W$ be graded superspaces. The tensor product $V\otimes W$ is considered as a graded superspace via $\bideg(v\otimes w)=(\deg(v)+\deg(w),|v|+|w|)$ for all homogeneous $v\in V$ and $w\in W$.
For $m\in\Z$ and $\de\in\Z/2$, a bidegree $(m,\de)$ homogeneous linear map $f:V\to W$ is a $\k$-linear map satisfying $f(V_{n,\eps})\subseteq W_{n+m,\eps+\de}$ for all $n,\eps$. We denote the $\k$-module of all bidegree $(m,\de)$ homogeneous linear maps  from $V$ to $W$ by $\Hom_\k(V,W)_{m,\de}$, and set 
$$
\Hom_\k(V,W):=\bigoplus_{m\in\Z,\,\de\in\Z/2}\Hom_\k(V,W)_{m,\de}.
$$

Let $V$ be a graded superspace and $d\in\N$.
The symmetric group $\Si_d$ acts on $V^{\otimes d}$ via
\begin{equation}\label{EWSignAction}
{}^{g}(v_1\otimes\dots\otimes v_d):=(-1)^{[g;v_1,\dots,v_d]}v_{g^{-1}(1)}\otimes\dots\otimes v_{g^{-1}(d)},
\end{equation}
where
$$
[g;v_1,\dots,v_d]:=\sum_{1\leq a<b\leq d,\, g^{-1}(a)>g^{-1}(b)}|v_a||v_b|.
$$

A graded ($\k$-)superalgebra is a graded superspace $A=\bigoplus_{n\in \Z,\,\eps\in\Z/2}A_{n,\eps}$ which is a unital algebra such that $A_{n,\eps}A_{m,\de}\subseteq A_{n+m,\eps+\de}$. 
The {\em opposite superalgebra}\, $A^\sop$ is the same as $A$ as a $\k$-module but has multiplication $a*b=(-1)^{|a||b|}ba$. 

Let $A,B$ be graded superalgebras. An {\em isomorphism} $f:A \to B$ is an algebra isomorphism which is homogeneous of bidegree $(0,\0)$. We write $A\cong B$ if $A$ and $B$ are isomorphic in this sense. The tensor product $A\otimes B$ is considered as a graded superalgebra via 
$
(a\otimes b)(a'\otimes b')=(-1)^{|b| |a'|}aa'\otimes bb'
$
for all homogeneous $a,a'\in A$ and $b,b'\in B$.

\begin{Example} 
{\rm 
Let $V$ be a graded superspace which is finite rank free as a $\k$-module. Then $\End_\k(V)$ is a graded superalgebra. 
Since $\k$ is a PID, every $V_{n,\eps}$ is a free $\k$-module of finite rank. Choosing a homogeneous basis of $V$ we see that $\End_\k(V)$ is isomorphic to a matrix algebra over $\k$ with appropriate bi-degrees assigned to the matrix units. 
We often refer to the graded superalgebra of the form $\End_\k(V)$ as a {\em graded matrix superalgebra}.
}
\end{Example}

Let $A$ be a graded superalgebra and $B\subseteq A$ be a (unital) graded subsuperalgebra. The 
{\em supercentralizer} $\Cent_A(B)$ of $B$ in $A$ is the 
graded subsuperalgebra of $A$ defined as the linear span of all (homogeneous) $a\in A$ such that $ba=(-1)^{|b||a|}ab$ for all (homogeneous) $b\in B$. 
We will need the following analogue of \cite[Proposition 4.10]{Evseev}. 

\begin{Lemma}\label{LEvseev}
Let $A$ be a graded superalgebra and $B$ a unital graded subsuperalgebra isomorphic to a graded matrix superalgebra. Then we have an isomorphism of graded superalgebras 
$
B \otimes \Cent_{A}(B) \iso A,\ b\otimes z \mapsto bz.
$
\end{Lemma}

\begin{proof}
By definition, we may assume that $B= \End_\k(V)$ for some finite rank free graded $\k$-superspace $V$. Let $v_1,\dots,v_n$ be a homogeneous $\k$-basis of $V$, and $E_{i,j}\in B$ be the element given by $E_{i,j}(v_k)=\de_{j,k}v_i$. 
Let $C:=E_{1,1}AE_{1,1}$. Note that for any $c\in C$, the element 
$\phi(c):=\sum_{i=1}^n(-1)^{|E_{i,1}||c|}E_{i,1}cE_{1,i}$ super-commutes with every $E_{j,k}$, so $\phi(c)\in \Cent_{A}(B)$. It follows that the maps $\phi:C\to \Cent_{A}(B)$ and $\Cent_{A}(B)\to C,\ z\mapsto zE_{1,1}$ are 
mutually inverse isomorphisms of graded superalgebras. Moreover, for all $i,j$, we have mutually inverse bidegree $(0,\0)$ isomorphisms of superspaces $C\to E_{i,i}AE_{j,j},\ c\mapsto E_{i,1}cE_{1,j}$ and $E_{i,i}AE_{j,j}\to C,\ a\mapsto E_{1,i}aE_{j,1}$. Note that for all homogeneous $c\in C$ and all $i,j$ we have $E_{i,j}\phi(c)=(-1)^{|E_{j,1}||c|}E_{i,1}cE_{1,j}$, so $E_{i,j}Z_B(A)=E_{i,i}AE_{j,j}$. It follows that the graded superalgebra homomorphism 
$
B \otimes \Cent_{A}(B) \iso A,\ b\otimes z \mapsto bz
$
is an isomorphism. 
\end{proof}

We always consider the group algebra $\k\Si_d$ of the symmetric group $\Si_d$ as a graded superalgebra in the trivial way, i.e. concentrated in bidegree $(0,\0)$. 
Given a graded superalgebra $A$, we consider the {\em wreath superproduct} 
\begin{equation}\label{EWdA}
W_d(A):=A\swr \Si_d, 
\end{equation}
where $A\swr \Si_d=A^{\otimes d}\otimes \k\Si_d$ as graded superspaces. To define the algebra structure,  we identify $A^{\otimes d}$ and $\k\Si_d$ with subspaces of $A^{\otimes d}\otimes \k\Si_d$ in the obvious way, and postulate that $A^{\otimes d}$, $\k\Si_d$ are subalgebras of  $A\swr \Si_d$, and that, recalling (\ref{EWSignAction}),  we have 
\begin{align*}
w\,(a_1\otimes\dots\otimes a_d)={}^w(a_1\otimes\dots\otimes a_d)\,w\qquad(w\in\Si_d,\ a_1,\dots, a_d\in A).
\end{align*}

Let $A$ be a superalgebra and $d\in\N_+$. 
For $1\leq t\leq d$, $1\leq r<s\leq d$, we have maps 
\begin{align}
\iota^{(d)}_{t}&:A\to A^{\otimes d},\ a\mapsto 1_A^{\otimes (t-1)}\otimes a\otimes 1_A^{\otimes(d-t)},\\
\label{EIotars}
\iota^{(d)}_{r,s}&:A\otimes A\to A^{\otimes d},\ a\otimes b\mapsto 1_A^{\otimes (r-1)}\otimes a\otimes 1_A^{\otimes(s-r-1)}\otimes b\otimes 1_A^{\otimes(d-s)}.
\end{align}
For $x\in A$ and $y\in A\otimes A$, we sometimes denote 
\begin{equation}\label{EInsertion}
x_{r}:= \iota^{(d)}_r(x)\in A^{\otimes d},\quad y_{r,s}:=\iota^{(d)}_{r,s}(y)\in A^{\otimes d}.
\end{equation}

\subsection{Graded supermodules} Let $A$ be a superalgebra. 
A {\em graded $A$-supermodule} $V$ is a (left) $A$-module which is also a graded superspace such that $A_{n,\eps}V_{m,\de}\subseteq V_{n+m,\eps+\de}$ for all $n,m,\eps,\de$. 
A graded subsupermodule is then an $A$-invariant graded subsuperspace of $V$. 
A graded $A$-supermodule $V$ is {\em irreducible} if it has exactly two graded subsupermodules: $0$ and $V$.

Let $V,W$ be graded $A$-supermodules. A 
{\em bidegree $(m,\de)$ homogeneous graded $A$-supermodule homomorphism from $V$ to $W$} is a bidegree $(m,\de)$ homogeneous linear map $f:V\to W$ satisfying $f(av)=(-1)^{\de|a|}af(v)$ for all (homogeneous) $a\in A,v\in V$. We denote by $\Hom_A(V,W)_{m,\de}$ the $\k$-module of all bidegree $(m,\de)$ homogeneous graded $A$-supermodule homomorphism from $V$ to $W$, and set 
$$
\Hom_A(V,W):=\bigoplus_{m\in\Z,\,\de\in\Z/2}\Hom_A(V,W)_{m,\de}.
$$
We refer to the elements of $\Hom_A(V,W)$ as 
{\em graded $A$-supermodule homomorphisms from $V$ to $W$}.

We denote by 
$\mod{A}$ the category of all 
{\em finitely generated}\, graded $A$-supermodules and all graded $A$-supermodule homomorphisms. 
An isomorphism in $\mod{A}$ will be denoted $\cong$. 
On the other hand, a bidegree $(0,\0)$ homogeneous  isomorphism in $\mod{A}$ will be denoted $\simeq$. 
We denote by $\Irr(A)$ the set of the isomorphism classes of the graded irreducible $A$-supermodules (for the isomorphism $\cong$). 

In fact, $\mod{A}$ is a graded $(\funQ,\Uppi)$-supercategory in the sense of \cite[Definition 6.4]{BE}, with $\Uppi$ the parity change functor and $\funQ$ the degree shift functor. To be more precise, if $M=\bigoplus_{n,\eps}M_{n,\eps}\in \mod{A}$, we have $(\Uppi M)_{n,\eps}=M_{n,\eps-\1}$ with the new action $a\cdot v=(-1)^{|a|}av$, and $(\funQ M)_{n,\eps}=M_{n-1,\eps}$ (with the same action).

Let $A,B$ be graded superalgebras. Given a graded $A$-supermodule $V$ and a graded $B$-supermodule $W$, we have the graded $(A\otimes B)$-supermodule $V\boxtimes W$ with the action 
$$
(a\otimes b)(v\otimes w)=(-1)^{|b||m|}(av\otimes bw)\qquad(a\in A,\, b\in B,\, v\in V,\, w\in W).
$$
A graded {\em $(A,B)$-bisupermodule}  $V$ is an $(A,B)$-bimodule which is also a graded superspace such that $A_{n,\eps}V_{m,\de}\subseteq V_{n+m,\eps+\de}$ and $V_{m,\de}B_{n,\eps}\subseteq V_{n+m,\eps+\de}$ for all $n,m,\eps,\de$. 
A homomorphism of graded $(A,B)$-bisupermodules is defined similarly to a homomorphism of graded $A$-supermodules. In particular, a bidegree $(m,\de)$ homogeneous homomorphism of graded $(A,B)$-bisupermodules $f:V\to W$ satisfies  $f(avb)=(-1)^{\de|a|}af(v)b$ for all homogeneous $a\in A,v\in V,b\in B$. 
Just like for graded supermodules, we use the notation $\cong$ to denote an isomorphism of graded bisupermodules, and  $\simeq$ to denote a bidegree $(0,\0)$ homogeneous  isomorphism of graded bisupermodules.

\begin{Example} \label{ExEndSopBim}
{\rm 
Given a graded $A$-supermodule $M$, we have the graded superalgebra $\End_A(M)^\sop$, and $M$ is naturally a graded $(A,\End_A(M)^\sop)$-bisupermodule with $f\in \End_A(M)^\sop=\End_A(M)$ acting via
$m\cdot f=(-1)^{|f||m|}f(m)$. 
}
\end{Example}

\subsection{Graded Morita superequivalence}
Let $A$ and $B$ be graded superalgebras. A {\em graded Morita superequivalence} between $A$ and $B$ is a Morita equivalence between $A$ and $B$ induced by a graded $(A,B)$-bisupermodule $M$ and a graded $(B,A)$-bisupermodule $N$, i.e. $M\otimes_B N\simeq A$ and $N\otimes_A M\simeq B$ 
(both bimodule isomorphisms homogeneous of bidegree $(0,\0)$). 
We write $A\sim_{\gsM}B$ for graded Morita superequivalence.

\begin{Example} \label{ExProgenerator} 
{\rm 
Let $A$ be a graded superalgebra and $\{e_1,\dots,e_n\}$ be (not necessarily distinct) bidegree $(0,\0)$ idempotents of $A$. 
We have a graded superalgebra $\bigoplus_{1\leq i,j\leq n}e_iAe_j$ with multiplication inherited from $A$. In fact, setting 
$$
P:=\bigoplus_{i=1}^n Ae_i,
$$
there the isomorphism of graded superalgebras
\begin{equation}\label{EEndAei}
\bigoplus_{1\leq i,j\leq n}e_iAe_j\iso \End_A(P)^\sop,\ a\mapsto\phi_a,
\end{equation}
where for (a homogeneous) $a\in e_iAe_j$, the endomorphism $\phi_a$ sends the summand $Ae_k$ of $P$ to $0$ if $k\neq i$ and the summand $Ae_i$ to the summand $Ae_j$ as follows:
$$
\phi_a(be_i)=(-1)^{|a||b|}be_ia
$$ 
(for a homogeneous $b\in A$). If $P$ is a projective generator for $A$ (equivalently, if $\sum_{i=1}^nAe_iA=A$), then $A$ is graded Morita superequivalent to $\End_A(P)^\sop$, with $M$ being $P$ considered as a graded $(A,\End_A(P)^\sop)$-superbimodule as in Example~\ref{ExEndSopBim}. 
}
\end{Example}

Let the graded $\k$-superalgebra $A$ be finitely generated as a $\k$-module, $e\in A$ be a bidegree $(0,\0)$ idempotent, and $\F=\k/\m$ be the quotient field of $\k$ by a maximal ideal $\m$. Denote $A_\F:=\F\otimes_\k A$ and $e_\F:=1_\F\otimes e\in A_\F$. 
We have the functors
\begin{equation}\label{EFunctors}
\begin{split}
Ae\otimes_{eAe}-:\mod{eAe}\to\mod{A},\\ 
eA\otimes_{A}-:\mod{A}\to \mod{eAe}.
\end{split}
\end{equation}
and
\begin{equation}\label{EFunctorsF}
\begin{split}
A_\F e_\F\otimes_{e_\F A_\F e_\F}-:\mod{e_\F A_\F e_\F}\to\mod{A_\F},\\ e_\F A_\F\otimes_{A_\F}-:\mod{A_\F}\to \mod{e_\F A_\F e_\F }.
\end{split}
\end{equation}
The following lemma can be found in \cite[Lemma 2.2.20]{KIS}, see also \cite[Lemma 8.26]{EK2}.

\begin{Lemma} \label{LMorExtScal} 
If for all fields\, $\F\in{\mathscr F}=\{\k/\m\mid \text{$\m$ is a maximal ideal of $\k$}\}
$, the functors (\ref{EFunctorsF}) are mutually quasi-inverse equivalences, then so are the functors (\ref{EFunctors}). \end{Lemma}

\subsection{Regrading}
\label{SSRegr}
An easy special case of a graded Morita superequivalence comes from a `regrading' of a graded superalgebra in the following sense. Suppose that $A$ is a graded superalgebra and we are given an orthogonal idempotent decomposition $1_A=\sum_{r=1}^n e_r$ with the idempotents $e_r$ of bidegree $(0,\0)$. Then as graded superspaces we have  $A=\bigoplus_{1\leq r,s\leq n}e_sAe_r$. Given {\em grading supershift parameters}
\begin{equation}\label{EShiftParameters}
t_1,\dots,t_n\in \Z\quad\text{and}\quad \eps_1,\dots,\eps_n\in\Z/2, 
\end{equation}
we can consider the new graded superalgebra 
\begin{equation}\label{EARegraded}
\zA:=\bigoplus_{1\leq r,s\leq n}\funQ^{t_r-t_s}\Uppi^{\eps_r-\eps_s} e_sAe_r,
\end{equation} 
which equals $A$ as an algebra but has different $\Z$- and $\Z/2$-gradings. To distinguish between the elements of $A$ and $\zA$, we denote by $\za$ the element of $\zA$ corresponding to $a\in A$. In particular, we have the idempotents $\ze_1,\dots,\ze_n\in\zA$. 
We have the graded $(\zA,A)$-bisupermodule
$$X:=\bigoplus_{r=1}^n \funQ^{-t_r}\Uppi^{-\eps_r}e_r A=
\bigoplus_{s=1}^n \funQ^{-t_s}\Uppi^{-\eps_s} \zA\ze_s
,$$
and the graded $(A,\zA)$-bisupermodule 
$$
Y:=\bigoplus_{s=1}^n \funQ^{t_s}\Uppi^{\eps_s} Ae_s=\bigoplus_{r=1}^n \funQ^{t_r}\Uppi^{\eps_r} \ze_r \zA,
$$
inducing a graded Morita superequivalence between $A$ and $\zA$. Under this graded Morita superequivalence, a graded $A$-supermodule $V$, corresponds to the graded $\zA$-supermodule 
\begin{equation}\label{EzVMor}
\zV:=X\otimes_A V\simeq \bigoplus_{r=1}^n \funQ^{-t_r}\Uppi^{-\eps_r}e_rV.
\end{equation}

\subsection{Superalgebras and supermodules}

A {\em superalgebra} is a $\Z/2$-graded algebra. In order to apply the theory developed in the previous subsections, we identify superalgebras with graded superalgebras concentrated in $\Z$-graded degree $0$, and similarly for (bi)supermodules. 
On the other hand, we consider graded superalgebras as  superalgebras by forgetting $\Z$-grading.

In particular, 
we get the notion of {\em Morita superequivalence} of superalgebras $A$ and $B$ induced by 
an $(A,B)$-bisupermodule $M$ and a $(B,A)$-bisupermodule $N$, i.e. $M\otimes_B N\simeq A$ and $N\otimes_A M\simeq B$.   We write $A\sim_{\sM}B$ for Morita superequivalence.

Suppose that $A$ is a superalgebra  which is free of finite rank as a $\k$-module. An element $\t\in \Hom_\k(A,\k)$ is called a {\em symmetrizing form} if $\t(A_{\1})=0$ and the bilinear form $(a,b)_\t:=\t(ab)
$ is a perfect symmetric pairing. 
If $A$ is equipped with a symmetrizing form, we say $A$ is a {\em symmetric superalgebra.} If $A,B$ are symmetric superalgebras, then  $A \otimes B$ is also a symmetric superalgebra with symmetrizing form 
$
a \otimes b \mapsto \t_A(a)\t_B(b),
$
where $\t_A$ and $\t_B$ are the symmetrizing forms of $A$ and $B$ respectively. 

\begin{Lemma} \label{LSopSymm} 
If $A$ is a symmetric superalgebra then so is~$A^\sop$. 
\end{Lemma}
\begin{proof}
Denoting the product in $A^\sop$ by $*$, the lemma follows from the computation $\t(a*b)=\t((-1)^{|a||b|}ba)=\t((-1)^{|a||b|}ab)=\t(b*a)$. 
\end{proof}

\begin{Lemma} \label{LSymRegr} 
Suppose that a graded superalgebra $\zA$ is obtained from a graded superalgebra $A$ by regrading, as in (\ref{EARegraded}). If $A$ is a symmetric superalgebra then so is~$\zA$. 
\end{Lemma}
\begin{proof}
Identify $A$ and $\zA$ as sets. We claim that if $\t$ is a symmetrizing form on $A$ then it is also a symmetrizing form on $\zA$. The only thing that is not clear is that $\t(\zA_\1)=0$. 
But is easy to see that $\t(e_iAe_j)=0$ unless $i=j$, and $\ze_i\zA_\1\ze_i=e_iA_\1e_i$.  
\end{proof}

\begin{Lemma} \label{LSymIdTr}
Let $A$ be a superalgebra  which is free of finite rank as a $\k$-module, and $e_1,\dots,e_n\in A$ be even idempotents. 
If $A$ is a symmetric superalgebra then so is the endomorphism superalgebra $\End_A(\bigoplus_{i=1}^n Ae_i)$. In particular, if $A$ is a symmetric superalgebra then so is the idempotent truncation $eAe$ for any even idempotent $e$. 
\end{Lemma}
\begin{proof}
The proof in \cite[Proposition IV.4.4]{SY} generalizes to superalgebras.
\end{proof}

\begin{Example} \label{ExCliffird} 
{\rm 
The rank $n$ {\em Clifford superalgebra\, $\cC_n$}  is the superalgebra
given by odd generators $\cc_1,\dots,\cc_n$ subject to the relations 
\begin{eqnarray}
\hspace{5mm}\cc_r^2=1 \quad \text{and}\quad 
\cc_r\cc_s = -\cc_s\cc_r\ \text{for}\ r\neq s.
\label{AHC2}
\end{eqnarray}
The superalgebra\, $\cC_n$ has basis $\{\cc_1^{\eps_1} \dots \cc_n^{\eps_n}\mid \eps_1,\dots,\eps_n\in\{0,1\}\}$.  
It is a symmetric superalgebra with the symmetrizing form $\t$ defined on basis elements via $\t(\cc_1^{\eps_1} \dots  \cc_n^{\eps_n})=\de_{\eps_1,0}\dots\de_{\eps_n,0}.
$ 
We have $\cC_n\otimes \cC_m\cong \cC_{n+m}.$ In particular, $\cC_n\cong \cC_1^{\otimes n}$. 
If $\k$ is an algebraically closed field of characteristic $\neq 2$ and $n$ is even, then $\cC_n$ is isomorphic to a matrix superalgebra (see for example \cite[Example 12.1.3]{Kbook}). Hence 
\begin{equation}\label{ECliffEq}
\cC_n\sim_{\sM}
\left\{
\begin{array}{ll}
\k &\hbox{if $n$ is even,}\\
\cC_1 &\hbox{if $n$ is odd.}
\end{array}
\right.
\end{equation}
}
\end{Example}

\begin{Lemma} \label{LSymTensCliff}
Let $n\in\N_+$ and $A$ be a $\k$-superalgebra which is free of finite rank as a $\k$-module. Then $A$ is a symmetric superalgebra if and only if $A\otimes\cC_n$ is a symmetric superalgebra.
\end{Lemma}
\begin{proof}
The `only-if' direction follows from the fact that $\cC_n$ is a symmetric superalgebra and tensor product of symmetric superalgebras is a symmetric superalgebra. 

To prove the converse, since $\cC_n\cong \cC_1^{\otimes n}$,  it suffices to establish that $A$ is a symmetric superalgebra if so is $A\otimes\cC_1$. 
So suppose that $A\otimes\cC_1$ is a symmetric superalgebra with symmetrizing form $\t$ and the corresponding symmetric perfect pairing $(\cdot,\cdot)_\t$. By definition, $\t((A\otimes\cC_1)_\1)=0$, so $\t(A_\1\otimes 1)=\t(A_\0\otimes \cc_1)=0$. We claim that, moreover,  $\t(A_\1\otimes \cc_1)=0$. Indeed, for $a\in A_\1$, we have $a\otimes \cc_1=(a\otimes 1)(1\otimes \cc_1)=-(1\otimes \cc_1)(a\otimes 1)$, hence $\t(a\otimes \cc_1)=\t((a\otimes 1)(1\otimes \cc_1))=\t((1\otimes \cc_1)(a\otimes 1))=-\t(a\otimes \cc_1)$ which now must be zero by our running assumption that $2$ is invertible in $\k$. It follows that the restriction of $(\cdot,\cdot)_\t$ to $A_\0\otimes 1$ is a perfect symmetric pairing, 
so we can define the symmetrizing form on $A$ via $a\mapsto \t(a\otimes 1)$. 
\end{proof}

\begin{Example} \label{ExT}
{\rm 
The {\em twisted group algebra $\cT_n$} of the symmetric group $\Si_n$ given by odd generators $\ct_1,\dots,\ct_{n-1}$ subject to the relations  
\begin{equation*}\label{ET_nRel}
\ct_r^2=1,\quad \ct_r\ct_s = -\ct_s\ct_r\text{ if }|r-s|>1,\quad (\ct_r\ct_{r+1})^3=1.
\end{equation*}
(By convention, $\cT_1$ is the totally even superalgebra $\k$.)
Choosing for each $w\in \Si_n$ a reduced decomposition $w=s_{r_1}\cdots s_{r_l}$, we define $\ct_w:=\ct_{r_1}\cdots \ct_{r_l}$ (which depends up to a sign on the choice of a reduced decomposition). Then  $\{\ct_w\mid w\in \Si_n\}$ is a basis of $\cT_n$, and $\cT_n$ is a symmetric superalgebra with symmetrizing form given by $\t(\ct_w)=\de_{w,1}$. 
}
\end{Example}

\begin{Example} \label{ExTC} 
{\rm 
Recalling the notation (\ref{EWdA}), we consider the wreath superproduct $\cY_n:=W_n(\cC_1)$ 
with basis 
$$
\{g\otimes \cc_1^{\eps_1}\otimes \dots\otimes \cc_1^{\eps_n}\mid g\in\Si_n,\,\eps_1,\dots,\eps_n\in\{0,1\}\}.
$$ 
This superalgebra is a symmetric superalgebra with the symmetrizing form $\t$ defined on basis elements via $$\t(g\otimes \cc_1^{\eps_1}\otimes \dots\otimes \cc_1^{\eps_n})=\de_{g,1}\de_{\eps_1,0}\dots\de_{\eps_n,0}.
$$ 
By \cite[\S2]{Ser} and \cite{Yam}, if $\k$ contains a square root of $-2$ then 
there is an isomorphism of superalgebras $\cY_n\iso \cT_n\otimes \cC_n$ (explaining our notation $\cY_n$).  
We refer the reader to \cite[Proposition 5.1.3, Lemma 5.1.4]{KlLi} and  \cite[Lemma 13.2.3]{Kbook} for more details on this.
}
\end{Example}

\begin{Example} \label{ExIH} 
{\rm 
Let $q\in \k$ with $q^2\neq 1$  and set $\xi:=q-q^{-1}$. 
The {\em Iwahori-Hecke algebra of the symmetric group $\Si_n$}  
is the purely even (super)algebra $\cH_n(q)$ given by 
(even) generators $T_1,\dots,T_{n-1}$ subject only to the following relations:
\begin{eqnarray}
T_r^2 = \xi T_r + 1,\ T_rT_{r+1}T_r = T_{r+1}T_rT_{i+1},\ \text{and}\  T_rT_s = T_s T_r\ \text{for}\ |r - s| > 1.  
\label{AHC3}
\end{eqnarray}
For $g\in\Si_n$ with reduced decomposition $g=s_{r_1}\dots s_{r_l}$, set $T_g:=T_{r_1}\dots T_{r_l}$.  
Then 
$\{T_g\mid g\in\Si_n\}
$ 
is  a basis of $\cH_n(q)$. 
It is well-known, see for example \cite[8.1.1]{GP}, that the algebra $\cH_n(q)$ is symmetric with the symmetrizing form $\t$ defined on basis elements via $\t(T_g)=\de_{g,1}$. Moreover, $\t(T_gT_h)=\de_{g,h^{-1}}$. 
}
\end{Example}

\begin{Example} \label{ExTC(q)} 
{\rm 
Let $q\in \k$ with $q^2\neq 1$ and set $\xi:=q-q^{-1}$. 
The {\em quantum wreath superproduct} $\cY_n(q)$, also known as {\em Olshanski's superalgebra}, is given by 
even generators $T_1,\dots,T_{n-1}$ and odd generators $\cc_1,\dots,\cc_n$ subject only to the relations (\ref{AHC2}), (\ref{AHC3}) and 
\begin{eqnarray*}
T_r\cc_r = \cc_{r+1}T_r\ \text{and}\  
T_r\cc_s = \cc_s T_r\ \text{for}\ s \neq r,r+1.
\label{AHC4}
\end{eqnarray*}
For $g\in\Si_n$ we define $T_g$ as in Example~\ref{ExIH}.  
By \cite[\S\S\,2-d,\,2-e]{BK1}, $\cY_n(q)$ has basis
$$\{T_g\cc_1^{\eps_1} \dots \cc_n^{\eps_n}\mid g\in\Si_n,\,\eps_1,\dots,\eps_n\in\{0,1\}\}.
$$ 
Now, $\cY_n(q)$ is a symmetric superalgebra with the symmetrizing form $\t$ defined on basis elements via $$\t(T_g\cc_1^{\eps_1} \dots \cc_n^{\eps_n})=\de_{g,1}\de_{\eps_1,0}\dots\de_{\eps_n,0}.
$$ 
To see this, it is sufficient to note that
$$
\t(T_g\cc_1^{\eps_1} \dots \cc_n^{\eps_n}T_h\cc_1^{\ka_1} \dots \cc_n^{\ka_n})=\de_{g,h^{-1}}\de_{\eps_1,\ka_1}\cdots \de_{\eps_n,\ka_n},
$$ 
which follows using $\t(T_gT_h)=\de_{g,h^{-1}}$, see Example~\ref{ExIH}.
}

\end{Example}

\section{Brauer tree superalgebras and generalized Schur algebras}

\subsection{Brauer tree superalgebras $\Zig_\ell$}
\label{SSBrTree}

Let $\k Q$ be the path algebra of the quiver $Q$ 
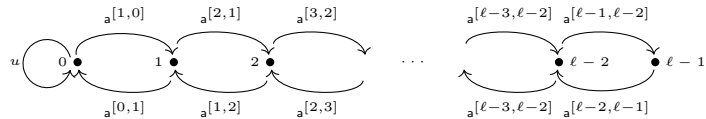
\begin{align*}
\begin{braid}\tikzset{baseline=3mm}
\coordinate (1) at (0,0);
\coordinate (2) at (4,0);
\coordinate (3) at (8,0);
\coordinate (4) at (12,0);
\coordinate (6) at (16,0);
\coordinate (L1) at (20,0);
\coordinate (L) at (24,0);
\draw [thin, black, ->] (-0.3,0.2) arc (15:345:1cm);
\draw [thin, black,->,shorten <= 0.1cm, shorten >= 0.1cm]   (1) to[distance=1.5cm,out=100, in=100] (2);
\draw [thin,black,->,shorten <= 0.25cm, shorten >= 0.1cm]   (2) to[distance=1.5cm,out=-100, in=-80] (1);
\draw [thin,black,->,shorten <= 0.25cm, shorten >= 0.1cm]   (2) to[distance=1.5cm,out=80, in=100] (3);
\draw [thin,black,->,shorten <= 0.25cm, shorten >= 0.1cm]   (3) to[distance=1.5cm,out=-100, in=-80] (2);
\draw [thin,black,->,shorten <= 0.25cm, shorten >= 0.1cm]   (3) to[distance=1.5cm,out=80, in=100] (4);
\draw [thin,black,->,shorten <= 0.25cm, shorten >= 0.1cm]   (4) to[distance=1.5cm,out=-100, in=-80] (3);
\draw [thin,black,->,shorten <= 0.25cm, shorten >= 0.1cm]   (6) to[distance=1.5cm,out=80, in=100] (L1);
\draw [thin,black,->,shorten <= 0.25cm, shorten >= 0.1cm]   (L1) to[distance=1.5cm,out=-100, in=-80] (6);
\draw [thin,black,->,shorten <= 0.25cm, shorten >= 0.1cm]   (L1) to[distance=1.5cm,out=80, in=100] (L);
\draw [thin,black,->,shorten <= 0.1cm, shorten >= 0.1cm]   (L) to[distance=1.5cm,out=-100, in=-100] (L1);
\blackdot(0,0);
\blackdot(4,0);
\blackdot(8,0);
\blackdot(20,0);
\blackdot(24,0);
\draw(0,0) node[left]{$0$};
\draw(4,0) node[left]{$1$};
\draw(8,0) node[left]{$2$};
\draw(14,0) node {$\cdots$};
\draw(20,0) node[right]{$\ell-2$};
\draw(24,0) node[right]{$\ell-1$};
 \draw(-2.6,0) node{$u$};
\draw(2,1.2) node[above]{$\za^{[1,0]}$};
\draw(6,1.2) node[above]{$\za^{[2,1]}$};
\draw(10,1.2) node[above]{$\za^{[3,2]}$};
\draw(18,1.2) node[above]{$\za^{[\ell-3,\ell-2]}$};
\draw(22,1.2) node[above]{$\za^{[\ell-1,\ell-2]}$};
\draw(2,-1.2) node[below]{$\za^{[0,1]}$};
\draw(6,-1.2) node[below]{$\za^{[1,2]}$};
\draw(10,-1.2) node[below]{$\za^{[2,3]}$};
\draw(18,-1.2) node[below]{$\za^{[\ell-3,\ell-2]}$};
\draw(22,-1.2) node[below]{$\za^{[\ell-2,\ell-1]}$};
\end{braid}
\end{align*}
generated by length $0$ paths $\{\ze^{[j]}\mid j\in J\}$, and length $1$ paths $\zu$ and 
$\za^{[k,k+1]},\za^{[k+1,k]}$ for $k=0,1,\dots,\ell-2$. 

Define the algebra $\Zig_\ell$ to be $\k Q$  
 modulo the following relations:
 
 \noindent 
 $\bullet$ all paths of length three or greater are zero;
 
 \noindent 
 $\bullet$  all paths of length two that are not cycles are zero;
 
 \noindent 
 $\bullet$ length-two cycles based at the same vertex are equal (in particular, $\zu^2=\za^{[0,1]}\za^{[1,0]}$).

\vspace{2mm}
Let 
$
\zc^{[0]}:=\zu^2$ and $\zc^{[i]}:=\za^{[i,i-1]}\za^{[i-1,i]}\quad \text{for $i=1,\dots,\ell-1$}. 
$
Then 
\begin{equation}\label{EZBasis}
\zB_\ell:=\{\ze^{[j]},\zc^{[j]}\mid j\in J\}
\cup\big\{\za^{[k,k+1]},\za^{[k+1,k]}\mid k\in \{0,1,\dots,\ell-2\}\big\}
\cup\{\zu\}
\end{equation}
is a basis of $\Zig_\ell$. It follows that
\begin{equation}\label{EzejZig}
\rank_\k \ze^{[j]}\Zig_\ell=
\left\{
\begin{array}{ll}
3 &\hbox{if $j=\ell-1$;}\\
4 &\hbox{otherwise.}
\end{array}
\right.
\end{equation}

We consider $\Zig_\ell$ as a graded superalgebra by setting
$$
\bideg(\ze^{[j]})=(0,\0),\ \bideg(\zu)=\bideg(\za^{[k+1,k]})=\bideg(\za^{[k,k+1]})=(2,\1).
$$
Note that then $\bideg(\zc^{[j]})=(4,\0)$ for all $j$. 

The graded superalgebra $\Zig_\ell$ has a (degree $-2$) symmetrizing form $\t$ such that $\t(\zc^{[j]})=1$ for all $j\in J$ and $\t(\ze^{[j]})=\t(\za^{[i,k]})=\t(\zu)=0$ for all admissible $j,i,k$. 

\subsection{Affine Brauer tree superalgebras $H_d(\Zig_\ell)$}
\label{SSAffZig}
Let $\zz$ be a variable of bidegree $(4,\0)$ and consider the graded polynomial superalgebra $\k[\zz]$. 
Denote by $\Zig_\ell[\zz]$ the free product $\k[\zz]\star \Zig_\ell$ subject to the relations 
$\zb\zz=-\zz\zb$ for all odd $\zb\in \zB_\ell$ and $\zb\zz=\zz\zb$ for all even $\zb\in \zB_\ell$. 
Then  
$\Zig_\ell[\zz]$ has basis  
$
\{\zz^n\zb\mid n\in\N,\,\zb\in\zB_\ell\}
$. 
We consider $\Zig_\ell$ naturally as a subalgebra of $\Zig_\ell[\zz]$.

Fix $d\in\N_+$. 
Using the notation (\ref{EInsertion}), 
the graded superalgebra $\Zig_\ell[\zz]^{\otimes d}$ 
has basis
\begin{align*}
\{\zz_1^{n_1}\cdots\zz_d^{n_d}\zb^{(1)}_1\cdots \zb^{(d)}_d\mid 
n_1,\dots,n_d\in\N,\ \zb^{(1)},\dots,\zb^{(d)}\in\zB_\ell\}.
\end{align*}
Recall that we always consider the group algebra $\k \Si_d$ as a graded superalgebra concentrated in bidegree $(0,\0)$. We consider the free product $\Zig_\ell[\zz]^{\otimes d}\star\k \Si_d$. For any $\bi\in J^d$, we have that  
\begin{equation}\label{EEI}
\ze^\bi:=\ze^{[i_1]}\otimes\dots\otimes \ze^{[i_d]}\in \Zig_\ell^{\otimes d}\subseteq \Zig_\ell[z]^{\otimes d}\subseteq \Zig_\ell[z]^{\otimes d}\star\k \Si_d.
\end{equation} 
Recalling the notation (\ref{EWSignAction}), (\ref{EInsertion}), we define the {\em affine Brauer tree (graded super)algebra  $H_d(\Zig_\ell)$} to be the free product $\Zig_\ell[z]^{\otimes d}\star\k \Si_d$ subject to the following relations:
\begin{align}
w\,(\za^{(1)}\otimes\dots\otimes \za^{(d)})&={}^w(\za^{(1)}\otimes\dots\otimes \za^{(d)})\,w
\label{ERAff1}
\\
(s_r \zz_t - \zz_{s_r(t)} s_r)\ze^\bi
&=
\begin{cases}
\big((\delta_{r,t}- \delta_{r+1,t})(\zc_r^{[i_r]} + \zc_{r+1}^{[i_{r+1}]})
&
\\
\qquad\qquad\qquad+\de_{i_r,0}\zu_r\zu_{r+1}\big)\ze^\bi
&
\text{if}\ i_r = i_{r+1},\\
(\delta_{r,t}+\delta_{r+1,t})\za_r^{[i_{r+1},i_r]} \za_{r+1}^{[i_r,i_{r+1}]} \ze^\bi
&
\text{if}\ |i_r-i_{r+1}|=1,\\
0
& 
\textup{otherwise,}
\end{cases}
\label{ERAff2}
\end{align}
for all $w\in \Si_d$, $\za^{(1)},\dots,\za^{(d)}\in \Zig_\ell$, 
$1\leq r<n$, $1\leq t\leq n$ and $\bi\in J^n$.

There are natural graded superalgebra homomorphisms (of bidegree  $(0,\0)$): 
\begin{align}\label{EIOTAS}
\iota_1:\Zig_\ell[\zz]^{\otimes d} \to H_d(\Zig_\ell)
\quad\text{and}
\quad \iota_2:\k \Si_d\to H_d(\Zig_\ell).
\end{align}
By the following theorem, these are embeddings. 
We will use the same symbols for elements of the domain of these maps as for their images in $H_d(\Zig_\ell)$. 

\begin{Theorem}\label{TAffBasis} \cite[Theorem 3.9]{KlLi}
The map 
$$\Zig_\ell[\zz]^{\otimes d} \otimes \k \Si_d\to H_d(\Zig_\ell),\ x \otimes y \mapsto \iota_1(x)\iota_2(y)
$$
is an isomorphism of graded superspaces. In particular, the following are bases of $H_d(\Zig_\ell)$:
\begin{align*}
&\{w\,\zz_1^{ n_1}\cdots \zz_d^{ n_d}\zb^{(1)}_1\cdots \zb^{(d)}_d\mid 
 n_1,\dots, n_d\in\N,\ \zb^{(1)},\dots,\zb^{(d)}\in \zB_\ell,\,w\in \Si_d\},
 \\
&\{\zz_1^{ n_1}\cdots \zz_d^{ n_d} \zb^{(1)}_1\cdots \zb^{(d)}_d\,w\mid 
 n_1,\dots, n_d\in\N,\ \zb^{(1)},\dots,\zb^{(d)}\in \zB_\ell,\,w\in \Si_d\}.
\end{align*}
\end{Theorem}

By the theorem and relations, $\Zig_\ell^{\otimes d}\otimes\k \Si_d$ is a graded subsuperalgebra of $H_d(\Zig_\ell)$ isomorphic to the wreath superproduct $W_d(\Zig_\ell)$ (recall the notation (\ref{EWdA})). Thus we have the natural embedding of graded superalgebra
\begin{equation}\label{EIotaZig}
\iota^{\Zig_\ell}:W_d(\Zig_\ell)\,\into\,H_d(\Zig_\ell).
\end{equation}
The quotient of $H_d(\Zig_\ell)$ by the graded superideal generated by $\zz_1,\dots,\zz_d$ is also isomorphic to $W_d(\Zig_\ell)$. 

\subsection{Some modules over $W_d(\Zig_\ell)$}  
From (\ref{EEI}), we have the idempotents 
$
\ze^\bi\in \Zig_\ell^{\otimes d}\subseteq W_d(\Zig_\ell)$ with
$$
\sum_{\bi\in J^d} \ze^\bi=1_{W_d(\Zig_\ell)}.
$$

Let $(\la,\bi)\in\Comp^\col(n,d)$. We define the idempotent 
\begin{equation}\label{EZeI}
\ze^{\la,\bi}:=\ze^{i_1^{\la_1}\cdots i_n^{\la_n}}=(\ze^{[i_1]})^{\otimes \la_1}\otimes\dots\otimes (\ze^{[i_n]})^{\otimes \la_n}\in \Zig_\ell^{\otimes d}\subseteq W_d(\Zig_\ell)
\end{equation}
and the {\em parabolic subalgebra}
$$
W_{\la,\bi}(\Zig_\ell)=\ze^{\la,\bi}\otimes \k \Si_\la\subseteq \ze^{\la,\bi} W_d(\Zig_\ell) \ze^{\la,\bi}.
$$ 
We have the {\em trivial} right graded $W_{\la,\bi}(\Zig_\ell)$-supermodule $\k_{\la,\bi}=\k$ concentrated in bidegree $(0,\0)$ with the action $1_\k\cdot(\ze^{\la,\bi}\otimes g)=1_\k$ for all $g\in\Si_\la$. We consider the induced graded right $W_d(\Zig_\ell)$-supermodule
\begin{equation}\label{EMLaBi}
M_{\la,\bi}:=\k_{\la,\bi}\otimes_{W_{\la,\bi}(\Zig_\ell)}\ze^{\la,\bi}W_d(\Zig_\ell)
\end{equation}
with generator $m_{\la,\bi}:=1_\k\otimes \ze^{\la,\bi}$ satisfying $m_{\la,\bi}\ze^{\la,\bi}=m_{\la,\bi}$ and $m_{\la,\bi} w=m_{\la,\bi}$ for all $w\in\Si_\la$.

Recall the notation (\ref{ELaBjj}). 

\begin{Lemma} \label{LDimM} 
Let $(\la,\bi)\in\Comp^\col(n,d)$. 
Then the module $M_{\la,\bi}$ is $\k$-free, with
$$
\rank_\k M_{\la,\bi}=
\binom{d}{\la_1\,\cdots\,\la_n}\,4^{d-|\la,\bi|_{\ell-1}}\,3^{|\la,\bi|_{\ell-1}}.
$$
\end{Lemma}
\begin{proof}
This follows from (\ref{EzejZig}) and the fact that $\k\Si_d$ is a free left $\k\Si_\la$-module of rank $[\Si_d:\Si_\la]=\binom{d}{\la_1\,\cdots\,\la_n}$, cf. \cite[Lemma~3.12]{EK2}.
\end{proof}

\begin{Remark} \label{RBLa}
{\rm 
Recall from \S\ref{SSPar} that we consider multicompositions in $\Comp^J(n,d)$ as colored compositions in $\Comp^\col(n\ell,d)$ via the injection (\ref{EBeta}). In particular, for any $\bla\in\Comp^J(n,d)$, recalling the notation (\ref{ESiBla}),  we have the idempotent 
$$\ze^\bla=(\ze^{[0]})^{\otimes \la^{(0)}_1}\otimes\cdots\otimes(\ze^{[\ell-1]})^{\otimes\la^{(\ell-1)}_1}\otimes\cdots\otimes (\ze^{[0]})^{\otimes\la^{(0)}_n}\otimes\cdots\otimes(\ze^{[\ell-1]})^{\otimes\la^{(\ell-1)}_n}\in W_d(\Zig_\ell),$$ 
the parabolic subalgebra $W_\bla(\Zig_\ell)=\ze^\bla\otimes\k\Si_\bla\subseteq \ze^\bla W_d(\Zig_\ell)\ze^\bla$, the module $M_\bla$ with generator $m_\bla$ satisfying $m_\bla\ze^\bla=m_\bla$ and $m_\bla w=m_\bla$ for all $w\in\Si_\bla$.  
}
\end{Remark}

\section{Generalized Schur algebras}
Let $n\in \N_+$ and $d\in\N$. Throughout the section we assume that the field of fractions of our principal ideal domain $\k$ has characteristic $0$. 

\subsection{Classical Schur algebra}
In this subsection we will work over an arbitrary field $\F$. 
Recall from \cite{Green} the {\em classical Schur algebra $S(n,d)_\F$}, with the standard idempotents $\{\xi_\la\mid\la\in\Comp(n,d)\}$ as in \cite[\S3.2]{Green}. The irreducible $S(n,d)_\F$-modules are labeled by their highest weights which are the elements of $\Comp_+(n,d)$:
$$
\Irr(S(n,d)_\F)=\{L_{n,d}(\la)_\F\mid \la\in \Comp_+(n,d)\},
$$
see \cite[(3.5a)]{Green}. For $\la\in\Comp_+(n,d)$ and $\mu\in\Comp(n,d)$, we have the weight multiplicities
\begin{equation*}\label{EKostka}
K_{\la,\mu,\F}:=
\left\{
\begin{array}{ll}
\dim \xi_\mu L_{n,d}(\la)_\F &\hbox{if $\mu\in\Comp(n,d)$,}\\
0 &\hbox{otherwise.}
\end{array}
\right.
\end{equation*}
($K_{\la,\mu,\C}$ are known as the Kostka numbers, but if $\cha \F\neq 0$ the numbers $K_{\la,\mu,\F}$ are in general not known.) 

The irreducible modules of an algebra of the form $S(n,d_0)_\F\otimes\dots\otimes S(n,d_{\ell-1})_\F$ are $L_{n,d}(\bla)_\F:=L_{n,d_0}(\la^{(0)})_\F\boxtimes\dots\boxtimes L_{n,d_{\ell-1}}(\la^{(\ell-1)})_\F$, labeled by the multicompositions $\bla=(\la^{(0)},\dots,\la^{(\ell-1)})\in\Comp^J_+(n,d)$ with $|\la^{(j)}|=d_j$ for all $j\in J$. For the direct sum of algebras 
\begin{equation}\label{ESJ}
S^J(n,d)_\F:=
\bigoplus_{d_0+\dots+d_{\ell-1}=d}S(n,d_0)_\F\otimes\dots\otimes S(n,d_{\ell-1})_\F,
\end{equation}
we then have
\begin{equation}\label{ESJIrr}
\Irr(S^J(n,d)_\F):=\{L_{n,d}(\bla)_\F\mid\bla\in \Comp^J_+(n,d)\}.
\end{equation}

For for $\bla\in\Comp^J_+(n,d)$ and $\bmu\in\Comp^J(n,d)$, we set
\begin{equation}\label{EKBlaBmu}
K_{\bla,\bmu,\F}:=\prod_{j\in J}K_{\la^{(j)},\mu^{(j)},\F}.
\end{equation}
Then for the element $\xi_{\mu^{(0)}}\otimes\dots\otimes \xi_{\mu^{(\ell-1)}}\in S^J(n,d)_\F$, we have 
\begin{equation}\label{E300924}
\dim\Big((\xi_{\mu^{(0)}}\otimes\dots\otimes \xi_{\mu^{(\ell-1)}})
L_{n,d}(\bla)_\F\Big)=K_{\bla,\bmu,\F}.
\end{equation}

\subsection{Generalized Schur algebra $S^{\Zig_\ell}(n,d)$}
Let $M_n(\Zig_\ell)$ be the graded $\k$-superalgebra of $n\times n$ matrices with entries in $\Zig_\ell$. For $\za\in\Zig_\ell$, set $\xi_{r,s}^\za:=\za E_{r,s}\in M_n(\Zig_\ell)$, where $E_{r,s}$ is the $(r,s)$ matrix unit. Recalling (\ref{EZBasis}), we have a basis $\{\xi_{r,s}^\zb\mid 1\leq r,s\leq n,\,\zb\in\zB_\ell\}$ of $M_n(\Zig_\ell)$. The symmetric group $\Si_d$ acts on the graded superalgebra $M_n(\Zig_\ell)^{\otimes d}$ according to (\ref{EWSignAction}). We define the  graded $\k$-superalgebra of invariants
$$
S^{\Zig_\ell}(n,d):=\big(M_n(\Zig_\ell)^{\otimes d}\big)^{\Si_d}.
$$

Let $\zH$ be a set of non-zero homogeneous elements of $\Zig_\ell$ (for example, $\zH=\zB_\ell$). We denote by $\Seq^\zH(n,d)$ the set of triples $(\ba,\br,\bs)$ with $\ba=\za_1\cdots\za_d\in\zH^d$ and $\br=r_1\cdots r_d,\bs=s_1\cdots s_d\in\{1,\dots,n\}^d$ such that for any $1\leq k\neq l\leq d$ we have $(\za_k,r_k,s_k)=(\za_l,r_l,s_l)$ only if $\za_k$ is even. Then $\Si_d$ acts on $\Seq^\zH(n,d)$ by (simultaneous) place permutations on $d$-tuples $\ba,\br,\bs$, and we write  
$(\ba,\br,\bs)\sim(\bb,\bt,\bu)$ if $(\ba,\br,\bs),(\bb,\bt,\bu)\in \Seq^\zH(n,d)$ are in the same orbit. For any $(\ba,\br,\bs)\in \Seq^\zH(n,d)$, we now have the element 
$$
\xi^\ba_{\br,\bs}:=\sum_{(\bb,\bt,\bu)\sim(\ba,\br,\bs)}(-1)^{\lan\ba,\br,\bs\ran+\lan\bb,\bt,\bu\ran}\xi^{b_1}_{t_1,u_1}\otimes\dots\otimes \xi^{b_d}_{t_d,u_d}\in S^{\Zig_\ell}(n,d),
$$
where 
$
\lan\ba,\br,\bs\ran
$
is defined as in \cite[\S5.2]{KM2}. Then 
$
\{\xi^\bb_{\br,\bs}\mid (\bb,\br,\bs)\in\Seq^{\zB_\ell}(n,d)/\Si_d\}
$ 
is a basis of $S^{\Zig_\ell}(n,d)$, see \cite[Lemma 3.3]{KM2}. 

Let $\bla=(\la^{(0)},\dots,\la^{(\ell-1)})\in\Comp^J(n,d)$. 
Consider the $d$-tuple
\begin{equation}\label{EECurlyR}
\ze\{\bla\}:=(\ze^{[0]})^{\la^{(0)}_1}\cdots(\ze^{[\ell-1]})^{\la^{(\ell-1)}_1}\,\cdots\, (\ze^{[0]})^{\la^{(0)}_n}\cdots(\ze^{[\ell-1]})^{\la^{(\ell-1)}_n}\in\zB_\ell^d,
\end{equation}
(not to be confused with the element $\ze^\bla\in W_d(\Zig_\ell)$ from Remark~\ref{RBLa}) 
and the $d$-tuple
$$
\br\{\bla\}:=1^{\la^{(0)}_1+\dots+\la^{(\ell-1)}_1}\cdots \,n^{\la^{(0)}_n+\dots+\la^{(\ell-1)}_n}\in\{1,\dots,n\}^d.
$$
As in \cite[(5.5)]{KM2}, we have the idempotent
$$
\xi_\bla:=\xi^{\ze\{\bla\}}_{\br\{\bla\},\br\{\bla\}}\in S^{\Zig_\ell}(n,d)
$$
and the orthogonal idempotent decomposition
\begin{equation*}\label{EXiLaS}
1_{S^{\Zig_\ell}(n,d)}=\sum_{\bla\in\Comp^J(n,d)}\xi_\bla.
\end{equation*}

\subsection{Tensor space}
Let $V_n=\Zig_\ell^{\oplus n}$ considered as a right graded $\Zig_\ell$-supermodule in the natural way. Considering $V_n$ as column vectors, we have the isomorphism $M_n(\Zig_\ell) \iso \End_{\Zig_\ell}(V_n)$ sending a matrix $\xi\in M_n(\Zig_\ell)$ to the left multiplication by $\xi$. This extends to 
the isomorphism $M_n(\Zig_\ell)^{\otimes d} \iso \End_{\Zig_\ell^{\otimes d}}(V_n^{\otimes d})$. Recalling (\ref{EWSignAction}), $\Si_d$ acts on $V_n^{\otimes d}$ on the right with $\k$-linear maps via $vg := {}^{g^{-1}}v$ for
$v \in V_n^{\otimes d},\, g \in\Si_d$. Thus we have a right graded  supermodule structures on $V_n^{\otimes d}$ over both
$\k \Si_d$ and $\Zig_\ell^{\otimes d}$. In this way, $V_n^{\otimes d}$  becomes a graded right $W_d(\Zig_\ell)$-supermodule, cf. \cite[\S5.2]{EK1}. Moreover:

\begin{Lemma} \label{LTensWdS} {\rm \cite[Lemma 5.7]{EK1}} 
The embedding $S^{\Zig_\ell}(n,d)\,\into\, M_n(\Zig_\ell)^{\otimes d}\iso \End_{\Zig_\ell^{\otimes d}}(V_n^{\otimes d})$ defines an isomorphism of graded superalgebras $S^{\Zig_\ell}(n,d)\cong \End_{W_d(\Zig_\ell)}(V_n^{\otimes d})$. 
\end{Lemma}

From now on we identify $S^{\Zig_\ell}(n,d)$ with the endomorphism algebra $\End_{W_d(\Zig_\ell)}(V_n^{\otimes d})$ via the isomorphism of Lemma~\ref{LTensWdS}.  

For $1\leq r\leq n$ and $j\in J$, we set $v_{r,j}:=(0,\dots,0,\ze^{[j]},0,\cdots,0)\in V_n$ with $\ze^{[j]}$ in the $r$th position. For $\bla\in \Comp^J(n,d)$, we set
$$
v_\bla:=v_{1,0}^{\otimes \la^{(0)}_1}\otimes\dots\otimes v_{1,\ell-1}^{\otimes \la^{(\ell-1)}_1}\otimes\dots\otimes v_{n,0}^{\otimes \la^{(0)}_n}\otimes\dots\otimes v_{n,\ell-1}^{\otimes \la^{(\ell-1)}_n}\in V_n^{\otimes d}.
$$

Recall the notation of Remark~\ref{RBLa}.

\begin{Lemma} \label{LVLaMLa} {\rm \cite[Lemma 5.25]{EK1}}
Let $\bla\in\Comp^J(n,d)$. 
Then $v_\bla\in \xi_\bla V_n^{\otimes d}$ and 
there is an isomorphism of graded right $W_d(\Zig_\ell)$-supermodules $\xi_\bla V_n^{\otimes d}\iso M_\bla$ mapping  $v_\bla$ to $m_\bla$.
\end{Lemma}

Identifying $V_n^{\otimes d}=\bigoplus_{\bla\in \Comp^J(n,d)}\xi_\bla V_n^{\otimes d}$ with $\bigoplus_{\bla\in \Comp^J(n,d)}M_\bla$ via the lemma, we obtain an isomorphism 
\begin{equation}\label{EIsoSchurPermEnd}
S^{\Zig_\ell}(n,d)\iso\End_{W_d(\Zig_\ell)}\Big(\bigoplus_{\bla\in \Comp^J(n,d)}M_\bla\Big)
\end{equation}
of graded $\k$-superalgebras, under which the idempotent $\xi_\bla$ corresponds to the projection onto the summand $M_\bla$. We identify $S^{\Zig_\ell}(n,d)$ with the endomorphism algebra in the right hand side via this isomorphism, cf. \cite[Corollary 5.26]{EK1}. Then, upon restriction, we have for any 
$\bla,\bmu\in \Comp^J(n,d)$ the identification
\begin{equation}\label{EIdentifyXilaSXimu}
\xi_\bmu S^{\Zig_\ell}(n,d) \xi_\bla=\End_{W_d(\Zig_\ell)}(M_\bla,M_\bmu).
\end{equation}

\subsection{Some special elements of $S^{\Zig_\ell}(n,d)$}
\label{SSSpecialElements}

Let $\bla=(\la^{(0)},\dots,\la^{(\ell-1)})\in\Comp^J(n-1,d-1)$. For $j\in J$, we define 
$$\bla_{\{j\}}=(\la_{\{j\}}^{(0)},\dots,\la_{\{j\}}^{(\ell-1)})\in\Comp^J(n,d)$$ where
$$
\la_{\{j\}}^{(i)}:=
\left\{
\begin{array}{ll}
(0,\la^{(i)}_1,\dots,\la^{(i)}_{n-1}) &\hbox{if $i\neq j$,}\\
(1,\la^{(j)}_1,\dots,\la^{(j)}_{n-1})  &\hbox{if $i=j$.}
\end{array}
\right.
$$

Using the notation (\ref{EInsertion}), for $\zx\in\Zig_\ell$, we have the element $\zx_1=\zx\otimes 1_{\Zig_\ell}^{\otimes (d-1)}\in \Zig_\ell^{\otimes d}\subseteq W_d(\Zig_\ell)$. Now, suppose $\zx\in\ze^{[j]}\Zig_\ell\ze^{[k]}$ and consider the element $m_{\bla_{\{j\}}}\zx_1\in M_{\bla_{\{j\}}}$. 
Recall the notation of Remark~\ref{RBLa} and (\ref{ESiBla}). 
Since $\ze^{\bla_{\{k\}}}=\ze^{[k]}\otimes \ze^\bla$ and $\Si_{\bla_{\{k\}}}=\Si_1\times \Si_\bla$, we have $m_{\bla_{\{j\}}}\zx_1(\ze^{\bla_{\{k\}}}\otimes g)=m_{\bla_{\{j\}}}\zx_1$ for all $g\in \Si_{\bla_{\{k\}}}$. So by the adjointness of induction and restriction, there exists a unique $W_d(\Zig_\ell)$-module homomorphism $M_{\bla_{\{k\}}}\to M_{\bla_{\{j\}}},\ m_{\bla_{\{k\}}}\mapsto m_{\bla_{\{j\}}}\zx_1$. Therefore there exists a unique endomorphism $\tti^\bla(\zx)\in S^{\Zig_\ell}(n, d)=\End_{W_d(\Zig_\ell)}\big(\bigoplus_{\bmu\in \Comp^J(n,d)}M_\bmu\big)$ with
$$
\tti^\bla(\zx):m_\bmu\mapsto
\left\{
\begin{array}{ll}
m_{\bla_{\{j\}}}\zx_1 &\hbox{if $\bmu=\bla_{\{k\}}$,}\\
0 &\hbox{otherwise.}
\end{array}
\right.
$$
This yields an (injective) homomorphism of graded superalgebras
\begin{equation}\label{Etti}
\tti^\bla:\Zig_\ell\to S^{\Zig_\ell}(n,d).
\end{equation}

Recall the notation (\ref{EECurlyR}). For $\bla\in\Comp^J(n-1,d-1)$, denote $h_r(\bla):=\sum_{j\in J}\la^{(j)}_r$ for $r=1,\dots,n-1$ and set 
\begin{equation*}\label{EHr}
h(\bla):=(h_1(\bla),\dots,h_{n-1}(\bla))\in \Comp(n-1,d-1).
\end{equation*}

\begin{Lemma} \label{LSomettila} 
Let $\bla\in\Comp^J(n-1,d-1)$ and set $h_r:=h_r(\bla)$ for $r=1,\dots,n-1$. Then
$$
\tti^\bla(\zx)=\xi^{\zx\ze\{\bla\}}_{1\,2^{h_1}\cdots \,n^{h_{n-1}},1\,2^{h_1}\cdots \,n^{h_{n-1}}}.
$$
\end{Lemma}
\begin{proof}
Let $\bmu\in\Comp^J(n,d)$. 
It suffices to observe that for $\zx \in \ze^{[j]}\Zig_\ell\ze^{[k]}$, the endomorphisms $\tti^\bla(\zx)$ and $\xi^{\zx\ze\{\bla\}}_{1\,2^{h_1}\cdots \,n^{h_{n-1}},1\,2^{h_1}\cdots \,n^{h_{n-1}}}
$ act the same way on the basis element  $v_\bmu$ of $ V_n^{\otimes d}$. But both endomorphisms send $v_\bmu$ to $0$ unless $\bmu=\bla_{\{k\}}$, and it is easy to see that 
$$
\xi^{\zx\ze\{\bla\}}_{1\,2^{h_1}\cdots \,n^{h_{n-1}},1\,2^{h_1}\cdots \,n^{h_{n-1}}} (v_{\bla_{\{k\}}})
=v_{\bla_{\{j\}}}\zx_1,
$$
as required (cf. the proof of \cite[Theorem 7.7]{EK1}). 
\end{proof}

\begin{Lemma} \label{LSumILa} 
Let $\chi=(h_1,\dots,h_{n-1})\in\Comp(n-1,d-1)$. 
Then for any $\zx\in\Zig_\ell$, we have 
$$
\sum_{\substack{\bla\in\Comp^J(n-1,d-1),\\ h(\bla)=\chi}}\tti^\bla(\zx)=\xi^{\zx 1_{\Zig_\ell}^{d-1}}_{1\,2^{h_1}\cdots\,n^{h_{n-1}},1\,2^{h_1}\cdots\,n^{h_{n-1}}}.
$$
\end{Lemma}
\begin{proof}
Note that 
$$
\sum_{\substack{\bla\in\Comp^J(n-1,d-1),\\ h(\bla)=\chi}}\xi^{\zx\ze\{\bla\}}_{1\,2^{h_1}\cdots \,n^{h_{n-1}},1\,2^{h_1}\cdots \,n^{h_{n-1}}}=\xi^{\zx 1_{\Zig_\ell}^{d-1}}_{1\,2^{h_1}\cdots\,n^{h_{n-1}},1\,2^{h_1}\cdots\,n^{h_{n-1}}}.
$$
So the result follows from Lemma~\ref{LSomettila}.
\end{proof}

\subsection{Generalized Schur algebra $T^{\Zig_\ell}(n,d)$.} 
\label{SST}
Following \cite[(3.9)]{KM2}, we define for every $(\bb,\br,\bs)\in\Seq^{\zB_\ell}(n,d)$ the element
$$
\eta^\bb_{\br,\bs}:=[\bb,\br,\bs]_\zc^!\,\,\xi^\bb_{\br,\bs}\in S^{\Zig_\ell}(n,d)_\k,
$$
where 
$$
[\bb,\br,\bs]^!_\zc:=\prod_{\substack{j\in J\\ 1\leq r,s\leq n}}\sharp\big\{k\mid 1\leq k\leq d\ \text{and}\ (\zb_k,r_k,s_k)=(\zc^{[j]},r,s)\big\}!
$$
By \cite[Proposition~3.12]{KM2}, the (full rank) $\k$-submodule $T^{\Zig_\ell}(n,d)$ of $S^{\Zig_\ell}(n,d)$ spanned by all $\eta^\bb_{\br,\bs}$'s is a graded $\k$-subsuperalgebra of $S^{\Zig_\ell}(n,d)$. 
Note that by definition we have $\eta^\bb_{\br,\bs}=\xi^\bb_{\br,\bs}$ unless at least two $\zc^{[j]}$'s appear in $\bb$, in which case the degree of $\xi^\bb_{\br,\bs}$ is at least $4$. So for small degrees we have:
\begin{equation}\label{ESmallDegST}
S^{\Zig_\ell}(n,d)^m=T^{\Zig_\ell}(n,d)^m\qquad (m=0,1,2,3).
\end{equation} 
In particular, the degree zero idempotents $\xi_\bla$ belong to $T^{\Zig_\ell}(n,d)$, and we have an orthogonal idempotent decomposition 
\begin{equation}\label{EXiLa}
1_{T^{\Zig_\ell}(n,d)}=\sum_{\bla\in\Comp^J(n,d)}\xi_\bla.
\end{equation}

By \cite[Corollary 6.7]{KM2}, the symmetrizing form $\t$ on $\Zig_\ell$ defined in \S\ref{SSBrTree}, extends to give a (degree $-2d$) symmetrizing form on $T^{\Zig_\ell}(n,d)$. 
Moreover, $T^{\Zig_\ell}(n,d)$ is a {\em maximal symmetric subalgebra} of $S^{\Zig_\ell}(n,d)$ in the following sense:

\begin{Theorem} \label{TMaxSymm} {\rm \cite[Corollary 3.11]{KlSym}}
If $n\geq d$ then the subalgebra $T^{\Zig_\ell}(n,d)\subseteq S^{\Zig_\ell}(n,d)$ is maximally symmetric, i.e. if $C$ is an intermediate $\k$-subalgebra with $T^{\Zig_\ell}(n,d)\subseteq C\subseteq S^{\Zig_\ell}(n,d)$ such that the $(\k/\m)$-algebras $(\k/\m)\otimes_\k C$ are symmetric for all maximal ideals $\m$ of $\k$, then $C=T^{\Zig_\ell}(n,d)$. 
\end{Theorem}

\subsection{Generation of $T^{\Zig_\ell}(n,d)$}
For $\zx\in\Zig_\ell$ we denote
$$
\tti_{r,s}(\zx):=\sum_{c=0}^{d-1} 1_{M_n(\Zig_\ell)}^{\otimes c}\otimes \xi_{r,s}^\zx\otimes 1_{M_n(\Zig_\ell)}^{\otimes d-c-1}\in (M_n(\Zig_\ell)^{\otimes d})^{\Si_d}=S^{\Zig_\ell}(n,d).
$$

\begin{Lemma} \label{LGen1} {\rm \cite[Theorem 4.13]{KM2}}
The $\k$-subalgebra $T^{\Zig_\ell}(n,d)\subseteq S^{\Zig_\ell}(n,d)$ is generated by the degree zero component  $S^{\Zig_\ell}(n,d)^0$ and the set $\{\tti_{r,s}(\zx)\mid \zx\in\Zig_\ell,\, 1\leq r,s\leq n\}$. 
\end{Lemma}

\begin{Corollary} \label{CGen1} 
The $\k$-subalgebra $T^{\Zig_\ell}(n,d)\subseteq S^{\Zig_\ell}(n,d)$ is generated by the degree zero component $S^{\Zig_\ell}(n,d)^0$ and the set $\{\tti_{1,1}(\zx)\mid \zx\in\Zig_\ell\}$. 
\end{Corollary}
\begin{proof}
If $d=1$, it is easy to see that the matrices in $M_n(\Zig_\ell^0)$ together with the matrix units of the form $\xi_{1,1}^\zx$ generate  $M_n(\Zig_\ell)$. 

Let $d>1$ and $T$ be the subalgebra of $S^{\Zig_\ell}(n,d)$ generated by $S^{\Zig_\ell}(n,d)^0\cup\{\tti_{1,1}(\zx)\mid \zx\in\Zig_\ell\}$. 
For $1< r,s\leq n$  
and $\zx\in\Zig_\ell$, using the fact that $\tti_{r,1}(1_{\Zig_\ell})\in  S^{\Zig_\ell}(n,d)^0$, we get
\begin{align*}
\tti_{r,1}(\zx)&=\tti_{r,1}(1_{\Zig_\ell})\tti_{1,1}(\zx)-\tti_{1,1}(\zx)\tti_{r,1}(1_{\Zig_\ell})\in T,
\\
\tti_{1,s}(\zx)&=
\tti_{1,1}(\zx)\tti_{1,s}(1_{\Zig_\ell})
-
\tti_{1,s}(1_{\Zig_\ell})\tti_{1,1}(\zx)\in T.
\end{align*}
Now, if $r\neq s$, we also get
$$
\tti_{r,s}(\zx)=\tti_{r,1}(1_{\Zig_\ell})\tti_{1,s}(\zx)-
\tti_{1,s}(\zx)\tti_{r,1}(1_{\Zig_\ell})\in T.
$$
If $r>1$, we get 
$$
\tti_{r,r}(\zx)-\tti_{1,1}(\zx)=\tti_{r,1}(1_{\Zig_\ell})\tti_{1,r}(\zx)-
\tti_{1,r}(\zx)\tti_{r,1}(1_{\Zig_\ell})\in T,
$$
whence $\tti_{r,r}(\zx)\in T$. The proof is now completed using Lemma~\ref{LGen1}. 
\end{proof}

Recall the homomorphism (\ref{Etti}). 

\begin{Corollary} \label{CGen2} 
Let $d\leq n$. Then the $\k$-subalgebra $T^{\Zig_\ell}(n,d)\subseteq S^{\Zig_\ell}(n,d)$ is generated by the degree zero component $S^{\Zig_\ell}(n,d)^0$ and the set $\{\tti^\bla(\zx)\mid \zx\in \Zig_\ell,\,\bla\in\Comp^J(n-1,d-1)\}$. 
\end{Corollary}
\begin{proof}
Let $T$ be the subalgebra of $S^{\Zig_\ell}(n,d)$ generated by $S^{\Zig_\ell}(n,d)^0\cup \{\tti^\bla(\zx)\mid \zx\in \Zig_\ell,\,\bla\in\Comp^J(n-1,d-1)\}$.  
Note that $T\subseteq T^{\Zig_\ell}(n,d)$ by (\ref{ESmallDegST}) since $\deg(\tti^\bla(\zx))=\deg(\zx)\leq 2$. To prove that $T=T^{\Zig_\ell}(n,d)$, in view of Corollary~\ref{CGen1}, it suffices to prove that any $\tti_{1,1}(\zx)$, with $\zx\in\Zig_\ell^{>0}$, lies in $T$.  

Let $\mu\in\La(n,d-1)$. Since $d-1<n$ by assumption, we have $\mu_k=0$ for some $1\leq k\leq n$. Define the following $d$-tuples
\begin{align*}
\bi(\mu):=1^{\mu_1+1}2^{\mu_2}\cdots n^{\mu_n},
\quad
\bj(\mu):=1\,2^{\mu_1}\cdots k^{\mu_{k-1}}(k+1)^{\mu_{k+1}}\cdots n^{\mu_n}.
\end{align*}
Denote ${\mathtt 1}:=1_{\Zig_\ell}$. Note that 
$$
\tti_{1,1}(\zx)=\sum_{\mu\in\La(n,d-1)}\xi^{\zx {\mathtt 1}^{d-1}}_{\bi(\mu),\bi(\mu)},$$
so it suffices to prove that 
$\xi^{\zx {\mathtt 1}^{d-1}}_{\bi(\mu),\bi(\mu)}\in T$ for any $\zx\in\Zig_\ell^{>0}$ and $\mu\in\La(n,d-1)$.  
But for such $\zx$ and $\mu$, we have 
\begin{align*}
\xi^{\zx {\mathtt 1}^{d-1}}_{\bi(\mu),\bi(\mu)}=
\xi^{{\mathtt 1}^d}_{\bi(\mu),\bj(\mu)}\xi^{\zx{\mathtt 1}^{d-1}}_{\bj(\mu),\bj(\mu)}
\xi^{{\mathtt 1}^d}_{\bj(\mu),\bi(\mu)}.
\end{align*}
The second factor in the right hand side belongs to $T$ since by Lemma~~\ref{LSumILa} we have 
$$
\xi^{\zx{\mathtt 1}^{d-1}}_{\bj(\mu),\bj(\mu)}=
\sum_{\substack{\bla\in\Comp^J(n-1,d-1),\\ h(\bla)=(\mu_1,\dots,\mu_{k-1},\mu_{k+1},\dots,\mu_n)}}\tti^\bla(\zx).
$$
It remains to note that the first and third factors have degree $0$,  so also belong to $T$.
\end{proof}

\subsection{Irreducible modules over $T^{\Zig_\ell}(n,d)_\F$}
In this subsection, $\F$ will denote a quotient field of $\k$ or the  field of fractions of $\k$. In particular, we will be able to change scalars from $\k$ to $\F$. 
To distinguish which ground ring we are working over, we use the index $\k$ or $\F$. So we have the algebras $S^{\Zig_\ell}(n,d)_\k$ and $S^{\Zig_\ell}(n,d)_\F$ with $S^{\Zig_\ell}(n,d)_\F\cong\F\otimes_\k S^{\Zig_\ell}(n,d)_\k$. 

To define the graded superalgebra $T^{\Zig_\ell}(n,d)_\F$ over $\F$, we extend scalars from $\k$ to $\F$:
\begin{equation}\label{ETF}
T^{\Zig_\ell}(n,d)_\F:=\F\otimes_\k T^{\Zig_\ell}(n,d)_\k.
\end{equation}
Note that $T^{\Zig_\ell}(n,d)_\F$ has the same dimension as $S^{\Zig_\ell}(n,d)_\F$ but the two algebras are not in general isomorphic if $0<\cha \F<d$. In fact, in this case one can use Theorem~\ref{TMaxSymm} to see that the algebras $T^{\Zig_\ell}(n,d)_\F$ and $S^{\Zig_\ell}(n,d)_\F$ are not Morita equivalent. 
We write $\eta^\bb_{\br,\bs}=1_\F\otimes\eta^\bb_{\br,\bs}\in\F\otimes_\k T^{\Zig_\ell}(n,d)_\k=T^{\Zig_\ell}(n,d)_\F$, $\xi_{\bla}=1_\F\otimes\xi_{\bla}\in T^{\Zig_\ell}(n,d)_\F$, etc.

The algebra 
$T^{\Zig_\ell}(n,d)_\F$ 
is 
non-negatively graded, and the irreducible graded $T^{\Zig_\ell}(n,d)_\F$-supermodules are the inflations of the irreducible supermodules over the degree zero component $T^{\Zig_\ell}(n,d)^0_\F$. 
Using (\ref{ESmallDegST}), we have 
$$
T^{\Zig_\ell}(n,d)^0_\F=S^{\Zig_\ell}(n,d)^0_\F\cong S^{\Zig_\ell^0}(n,d)_\F,
$$
see \cite{KM2}. 
Moreover, $\Zig_\ell^0\cong \k^{J}$, and by \cite[Lemma 4.8]{KM2}, we deduce that  
\begin{equation}\label{EDegZero}
T^{\Zig_\ell}(n,d)^0_\F\cong S^J(n,d)_\F,
\end{equation}
using the notation (\ref{ESJ}). So, recalling (\ref{ESJIrr}), we have 
\begin{equation}\label{EIrrT}
\Irr(T^{\Zig_\ell}(n,d)_\F)=\{L^{\Zig_\ell}_{n,d}(\bla)_\F\mid\bla\in \Comp^J_+(n,d)\},
\end{equation}
where $L^{\Zig_\ell}_{n,d}(\bla)_\F$ denotes the inflation of $L_{n,d}(\bla)_\F$ from $T^{\Zig_\ell}(n,d)^0_\F= S^J(n,d)_\F$ to $T^{\Zig_\ell}(n,d)_\F$. 

For $\bla\in\Comp^J_+(n,d)$ and $\bmu\in\Comp^J(n,d)$, 
recalling (\ref{E300924}), note that 
\begin{equation}\label{EWtMulBMuBLa}
\dim\xi_\bmu L_{n,d}^{\Zig_\ell}(\bla)_\F
=
\dim\Big((\xi_{\mu^{(0)}}\otimes\dots\otimes \xi_{\mu^{(\ell-1)}})L_{n,d}(\bla)_\F\Big)
=
K_{\bla,\bmu,\F}.
\end{equation}

If $n\geq d$, we have $\Comp^J_+(n,d)=\Par^J(d)$. In particular, we deduce:

\begin{Lemma} \label{LIrrT}
Let $n\geq d$. Then $|\Irr(T^{\Zig_\ell}(n,d)_\F)|=|\Par^\ell( d)|$.
\end{Lemma}

\begin{Lemma} \label{LTTrunc}
Let $n\geq d$. Then $T^{\Zig_\ell}(n,d)_\F$ is graded Morita superequivalent to $W_{d,\F}^{\Zig_\ell}$ if and only if $\cha\F=0$ or $\cha\F>d$.
\end{Lemma}
\begin{proof}
By \cite[Lemma 5.15]{EK1} and \cite[Remark 5.17]{KM2}, $W_{d,\F}^{\Zig_\ell}$ is an idempotent truncation of $T^{\Zig_\ell}(n,d)_\F$. So in view of Lemma~\ref{LIrrT}, it suffices to note that $|\Irr(W_{d,\F}^{\Zig_\ell})|=|\Par^J(d)|$ if and only if $\cha\F=0$ or $\cha\F>d$, which is well-known. For example, the `if'-part of this statement follows from \cite[Proposition 4.6]{FKM}. On the other hand, if $0<\cha\F\leq d$ one can use \cite[(6.4b)]{Green} to deduce $|\Irr(W_{d,\F}^{\Zig_\ell})|<|\Par^J(d)|$.  
\end{proof}

\section{Lie Theory}

\subsection{Lie type $\ttA_{\ell}^{\hspace{-.3mm}{}^{(2)}}$}
\label{SSLTN}
Let $\g$ be the Kac-Moody Lie algebra (over $\C$) of type $\ttA_{2\ell}^{(2)}$,  
see \cite[Chapter 4]{Kac}. The Dynkin diagram of $\g$ has vertices labeled by $I$:
\vspace{2mm} 
$$
{\begin{picture}(330, 15)%
\put(6,5){\circle{4}}%
\put(101,2.45){$<$}%
\put(12,2.45){$<$}%
\put(236,2.42){$<$}%
\put(25, 5){\circle{4}}%
\put(44, 5){\circle{4}}%
\put(8, 4){\line(1, 0){15.5}}%
\put(8, 6){\line(1, 0){15.5}}%
\put(27, 5){\line(1, 0){15}}%
\put(46, 5){\line(1, 0){1}}%
\put(49, 5){\line(1, 0){1}}%
\put(52, 5){\line(1, 0){1}}%
\put(55, 5){\line(1, 0){1}}%
\put(58, 5){\line(1, 0){1}}%
\put(61, 5){\line(1, 0){1}}%
\put(64, 5){\line(1, 0){1}}%
\put(67, 5){\line(1, 0){1}}%
\put(70, 5){\line(1, 0){1}}%
\put(73, 5){\line(1, 0){1}}%
\put(76, 5){\circle{4}}%
\put(78, 5){\line(1, 0){15}}%
\put(95, 5){\circle{4}}%
\put(114,5){\circle{4}}%
\put(97, 4){\line(1, 0){15.5}}%
\put(97, 6){\line(1, 0){15.5}}%
\put(6, 11){\makebox(0, 0)[b]{$_{0}$}}%
\put(25, 11){\makebox(0, 0)[b]{$_{1}$}}%
\put(44, 11){\makebox(0, 0)[b]{$_{{2}}$}}%
\put(75, 11){\makebox(0, 0)[b]{$_{{\ell-2}}$}}%
\put(96, 11){\makebox(0, 0)[b]{$_{{\ell-1}}$}}%
\put(114, 11){\makebox(0, 0)[b]{$_{{\ell}}$}}%
\put(230,5){\circle{4}}%
\put(249,5){\circle{4}}%
\put(231.3,3.2){\line(1,0){16.6}}%
\put(232,4.4){\line(1,0){15.2}}%
\put(232,5.6){\line(1,0){15.2}}%
\put(231.3,6.8){\line(1,0){16.6}}%
\put(230, 11){\makebox(0, 0)[b]{$_{{0}}$}}%
\put(249, 11){\makebox(0, 0)[b]{$_{{1}}$}}%
\put(290,2){\makebox(0,0)[b]{\quad if $\ell = 1$.}}%
\put(175,2){\makebox(0,0)[b]{if $\ell \geq 2$, \qquad and}}%
\end{picture}}
$$
We denote by $P$ the {\em weight lattice} of $\g$. 
We have the set $\{\alpha_i \:|\:i \in I\}\subset P$ of {\em simple roots} and the set $\{\La_i \:|\:i \in I\}\subset P$ of {\em fundamental dominant weights}. 
We denote by $Q$ the sublattice of $P$ generated by the simple roots and set 
$$Q_+:=\Big\{\sum_{i\in I}m_i\al_i\mid m_i\in\N\  \text{for all $i\in I$}\Big\}\subset Q.
$$
For $\theta=\sum_{i\in I}m_i\al_i\in Q_+$, its {\em height} is  
$
\height(\theta):=\sum_{i\in I}m_i.
$

We have the {\em Weyl group} $W$ and 
  a {\em normalized $W$-invariant form $(.|.)$} on $P$ whose Gram matrix with respect to the linearly independent set $\al_0,\al_1,\dots,\al_\ell$ is: 
$$
\left(
\begin{matrix}
2 & -2 & 0 & \cdots & 0 & 0 & 0 \\
-2 & 4 & -2 & \cdots & 0 & 0 & 0 \\
0 & -2 & 4 & \cdots & 0 & 0 & 0 \\
 & & & \ddots & & & \\
0 & 0 & 0 & \dots & 4 & -2& 0 \\
0 & 0 & 0 & \dots & -2 & 4& -4 \\
0 & 0 & 0 & \dots & 0 & -4& 8 \\
\end{matrix}
\right)
\quad
\text{if $\ell\geq 2$, and}
\quad
\left(
\begin{matrix}
2 & -4 \\
-4 & 8
\end{matrix}
\right)
\quad \text{if $\ell=1$.}
$$
In particular, we have 
$(\al_0|\al_0)=2$, $(\al_\ell|\al_\ell)=8$, and $(\al_i|\al_i)=4$ for all other $i\in I$.

For $\bi=i_1\cdots i_n\in I^n$, we denote 
\begin{equation}\label{EWt}
\wt(\bi):=\al_{i_1}+\dots+\al_{i_n}\in Q_+.
\end{equation} 
For $\theta\in Q_+$ of height $n$, we set
$$I^\theta:=\{\bi\in I^n\mid \wt(\bi)=\theta\}.$$ 
We define $I^{\theta}_{\di}$ to be the set of all expressions of the form
$i_1^{(m_1)} \cdots i_r^{(m_r)}$ with
$m_1,\dots,m_r\in \N_+$, $i_1,\dots,i_r\in I$
 and $m_1 \al_{i_1} + \cdots + m_r \al_{i_r} = \theta$. We refer to such expressions as {\em divided power words}. 
 We identify $I^\theta$ with the subset of $I^\theta_\di$ which consists of all expressions as above with all $m_k=1$. We use the same notation for concatenation of divided power words as for concatenation of words. 
For $\bi=i_1^{(m_1)} \cdots i_r^{(m_r)}\in I^\theta_\di$, we denote
\begin{equation}\label{EBi!}
\bi!:=m_1!\cdots m_r!
\end{equation}
and
\begin{equation}\label{EBiBar}
\bar\bi:=i_1^{m_1} \cdots i_r^{m_r}\in I^\theta.
\end{equation}

\subsection{The irreducible module $V(\La_0)$}
\label{SSW}
We denote by $V(\La_0)$ the irreducible integrable $\g$-module  with highest weight $\La_0$. Its weights are of the form $\La_0-\theta$ for $\theta\in Q_+$. For such $\theta$ we denote 
the $(\La_0-\theta)$-weight space of $V(\La_0)$ by $V(\La_0)_{\La_0-\theta}$. We denote  
\begin{equation}\label{EWeights}
\Weights:=\{\theta\in Q_+\mid V(\La_0)_{\La_0-\theta}\neq 0\}. 
\end{equation}

By \cite[(12.6.1),(12.6.2),(12.13.5)]{Kac}, we have
$$
\Weights=\{\La_0-w\La_0+d\de\mid w\in W,\ d\in\N\}. 
$$
We define the set of {\em nuclei}:
\begin{equation}\label{ENuclei}
\Nuclei:=\{\La_0-w\La_0\mid w\in W\}\subseteq \Weights.
\end{equation}
Since the union in \cite[(12.6.1)]{Kac} is disjoint, any $\theta\in \Weights$ can be written in the from $\rho+d\de$ for {\em unique} $\rho\in \Nuclei$ and {\em unique} $d\in\N$. 
We refer to $\rho$ as the {\em nucleus} of $\theta$ and denote it $\rho(\theta)$, and to $d$ as the {\em mass} of $\theta$ and denote it $d(\theta)$. 

Setting $p:=2\ell+1$, denote the set of {\em $p$-strict partitions} by $\Par_p$, see \cite[\S2.3a]{KlLi}. There is a well-known surjective map 
\begin{equation}\label{ECont}
\cont:\Par_{p}\to \Weights,
\end{equation}
taking $p$-strict partitions to their residue contents, see for example \cite[Lemma 3.1.39]{KlLi}. 
For a $p$-strict partition $\la$ we denote its 
{\em $\bar p$-core} by $\core_{\bar p}(\la)$ and its {\em $\bar p$-weight} by $\wt_{\bar p}(\la)$, see \cite[\S2.3a]{KlLi}. Then $\rho(\cont(\la))=\cont(\core_{\bar p}(\la))$ and $d(\cont(\la))=\wt_{\bar p}(\la)$, i.e. under the map $\cont$, nucleus corresponds to $\bar p$-core and mass corresponds to $\bar p$-weight.

\subsection{The root system}
\label{SSRootSystem}
Let 
$
\delta = \sum_{i=0}^{\ell - 1} 2 \alpha_i + \alpha_\ell
$
be the {\em null-root}. 
As in \cite[\S5]{Kac}, the set $\Phi$ of {\em roots} of $\g$ is a disjoint union of the set $\Phi^\im=\{n\de\mid n\in \Z\}$ of {\em imaginary roots} and the set $\Phi^\re$ of {\em real roots}. 
Let $\Phi'$ be the root system of type $C_\ell$ whose Dynkin diagram is obtained by dropping the simple root $\al_0$ from our type $A_{2\ell}^{(2)}$ Dynkin diagram. Then $\Phi'=\Phi'_{\text{s}}\sqcup \Phi'_{\text{l}}$ where $\Phi'_{\text{s}}=\{\al\in\Phi'\mid (\al|\al)=4\}$ and $\Phi'_{\text{l}}=\{\al\in\Phi'\mid (\al|\al)=8\}$. 
By \cite[\S6]{Kac}, we have $\Phi^\re=\Phi^\re_{\text{s}}\sqcup \Phi^\re_{\text{m}}\sqcup \Phi^\re_{\text{l}}$ for 
\begin{align}
\Phi^\re_{\text{s}}&= \{(\al+(2n-1)\de)/2 \mid \al\in\Phi'_l,\, n\in\Z\},
\label{EPhiS}
\\
\Phi^\re_{\text{m}}&=\{\al+n\de \mid \al\in\Phi'_s,\, n\in\Z\},
\label{EPhiM}
\\ 
\Phi^\re_{\text{l}}&=\{\al+2n\de \mid \al\in\Phi'_l,\, n\in\Z\}.
\label{EPhiL}
\end{align}
The set of {\em positive roots} is 
$\Phi_+=\Phi_+^\im\sqcup\Phi_+^\re,$  
where $\Phi_+^\im=\{n\de\mid n\in\N_+\}$ and $\Phi_+^\re$ consists of the roots in $\Phi'_+$ together with the roots in (\ref{EPhiS})-(\ref{EPhiL}) with $n>0$, cf. \cite[Proposition 6.3]{Kac}. 

Let   $\tilde\al=2\al_1+\dots+2\al_{l-1}+\al_\ell\in\Phi'$ be the highest root in $\Phi'$. Denote by $\Phi'_{\sharp}$ the set of all roots in 
$\Phi'$ which are non-negative linear combinations of $\al_1,\dots,\al_{\ell-1},-\tilde\al$ (this is a non-standard choice of a system of positive roots in $\Phi'$). 

Following \cite{BKT}, a {\em convex preorder} on $\Phi_+$ is a preorder $\preceq$ such that the following three conditions hold for all $\be,\ga\in\Phi_+$:
(1) $\be\preceq\ga$ or $\ga\preceq \be$; 
(2) {if $\be\preceq \ga$ and $\be+\ga\in\Phi_+$, then $\be\preceq\be+\ga\preceq\ga$; 
(3) $\be\preceq\ga$ and $\ga\preceq\be$ if and only if $\be$ and $\ga$ are proportional.

{\em Throughout this paper}, we work with a convex preorder $\preceq$ of \cite[Example 3.3.4]{KlLi}. 
As pointed out in \cite[Example 3.3.4]{KlLi}, such a preorder has the property that
\begin{align*}
\{\be\in \Phi_+^\re\mid \be\succ\de\}
&=\{\be\in\Phi_+\mid\text{$\be$ is of the form $r\al+s\de$ with $r\in\{1/2,1\}$, $\al\in\Phi'_{\sharp}$}\},
\\
\{\be\in \Phi_+^\re\mid \be\prec\de\}
&=\{\be\in\Phi_+\mid\text{$\be$ is of the form $-r\al+s\de$ with $r\in\{1/2,1\}$, $\al\in\Phi'_{\sharp}$}\}.
\end{align*}

\section{Quiver Hecke superalgebras}

\subsection{Quiver Hecke superalgebra $R_\theta$}
\label{SSDefQHSA}
We work with the quiver Hecke superalgebra of type $A_{2\ell}^{(2)}$, as defined in \cite{KKT}. 

For $i\in I$, we define its parity $|i|\in\Z/2$ as follows 
\begin{equation}\label{EIParity}
|i|:=
\left\{
\begin{array}{ll}
\1 &\hbox{if $i=0$,}\\
\0 &\hbox{otherwise.}
\end{array}
\right.
\end{equation}

For $i,j,k\in I$, we define polynomials $Q_{i,j}(u,v),B_{i,j,k}(u,v)\in\k[u,v]$ as follows:
\begin{eqnarray*}
Q_{i,j}(u,v)&:=&
\left\{
\begin{array}{ll}
0 &\hbox{if $i=j$,}\\
1 &\hbox{if $|i-j|>1$,}\\
(j-i)(u-v) &\hbox{if $|i-j|=1$ and $1\leq i,j<\ell$,}
\\

(j-i)(u^2-v) &\hbox{if $\{i,j\}=\{0,1\}$ or $\{\ell-1,\ell\}$, and $\ell>1$,}\\

(j-i)(u^4-v) &\hbox{if $\{i,j\}=\{0,1\}$ and $\ell=1$,}
\end{array}
\right.
\\
B_{i,j,k}(u,v)&:=&\left\{
\begin{array}{ll}
-1 &\hbox{if $i=k=j+1$,}\\
1 &\hbox{if $i=k= j-1\not\in \{0,\ell-1\}$,}\\
(u+v) &\hbox{if $i=k= j-1=\ell-1>0$,}\\
(v-u) &\hbox{if $i=k=j-1=0$ and $\ell>1$,}\\
(u^2+v^2)(v-u) &\hbox{if $i=k= j-1=0$ and $\ell=1$,}\\
0 &\hbox{otherwise,}
\end{array}
\right.
\end{eqnarray*}

Let $\theta\in Q_+$ be of height $n$. The {\em quiver Hecke superalgebra} $R_\theta$ is the unital graded $\k$-superalgebra generated by the elements 
$
\{1_\bi\mid\bi\in I^\theta\}\cup \{y_1,\dots,y_n\}\cup \{\psi_1,\dots,\psi_{n-1}\}
$
and the following defining relations (for all admissible $r,s,\bi$, etc.)
\begin{equation}
\label{R1}
1_\bi 1_\bj=\de_{\bi,\bj}1_\bi,
\end{equation}
\begin{equation}
\sum_{\bi\in I^\theta}1_\bi=1,
\label{R2}
\end{equation}
\begin{equation}
\label{R2.5}
y_r1_\bi=1_\bi y_r,
\end{equation}
\begin{equation}
\label{R2.75}
\psi_r1_\bi=1_{s_r\cdot \bi}\psi_r,
\end{equation}
\begin{equation}
y_ry_s1_\bi=
\left\{
\begin{array}{ll}
-y_sy_r1_{\bi} &\hbox{if $r\neq s$ and $|i_r|=|i_s|=\1$,}\\
 y_sy_r1_{\bi}&\hbox{otherwise,}
\end{array}
\right.
\label{R3}
\end{equation}
\begin{equation}
\psi_r y_s1_{\bi}=(-1)^{|i_r||i_{r+1}||i_s|}y_s \psi_r 1_{\bi}\qquad(s\neq r,r+1),
\label{R4}
\end{equation}
\begin{equation}
\begin{split}
(\psi_ry_{r+1}-(-1)^{|i_r||i_{r+1}|}y_r \psi_r) 1_{\bi}&=
(y_{r+1} \psi_r-(-1)^{|i_r||i_{r+1}|} \psi_ry_r) 1_{\bi}
\\&=
\left\{
\begin{array}{ll}
1_{\bi} &\hbox{if $i_r=i_{r+1}$,}\\
0 &\hbox{otherwise,}
\end{array}
\right.
\end{split}
\label{R5}
\end{equation}
\begin{equation}\label{R6}
\psi_r^21_{\bi}=Q_{i_r,r_{r+1}}(y_r,y_{r+1}),
\end{equation}
\begin{equation}\label{R65}
\psi_r\psi_s 1_{\bi}=(-1)^{|i_r||i_{r+1}||i_s||i_{s+1}|}\psi_s\psi_r 1_{\bi}\qquad(|r-s|>1),
\end{equation}
\begin{equation}
(\psi_{r+1}\psi_r\psi_{r+1}-\psi_{r}\psi_{r+1} \psi_r) 1_{\bi}=B_{i_r,i_{r+1},i_{r+2}}(y_r,y_{r+2})1_{\bi}
\label{R7}
\end{equation}
The structure of a graded superalgebra on $R_\theta$ is defined by setting 
\begin{align*}
\bideg(1_{\bi})&:=(0,\0),
\\ \bideg(y_s1_{\bi})&:=((\al_{i_s}|\al_{i_s}),|i_s|),
\\ \bideg(\psi_r 1_{\bi})&:=-((\al_{i_r}|\al_{i_{r+1}}),|i_r||i_{r+1}|).
\end{align*}
We denote the identity in $R_\theta$ by $1_\theta$. 

\subsection{Basis, divided power idempotents, parabolic subalgebras}
For every $w\in \Si_n$, we choose a reduced decomposition $w=s_{r_1}\dots s_{r_l}$ and define $\psi_{w}:=\psi_{r_1}\cdots\psi_{r_l}$. In general $\psi_w$ depends on the choice of a reduced decomposition. By \cite[Corollary 3.15]{KKT}, for $\theta\in Q_+$ with $n=\height(\theta)$, the following sets are $\k$-bases of  $R_\theta$:
\begin{align*}
&\{\psi_w y_1^{k_1}\dots y_n^{k_n}1_\bi \mid w\in \Si_n,\ k_1,\dots,k_n\in\N, \ \bi\in I^\theta\},
\\
&\{ y_1^{k_1}\dots y_n^{k_n}\psi_w1_\bi \mid w\in \Si_n,\ k_1,\dots,k_n\in\N, \ \bi\in I^\theta\}
\end{align*}

Let $\theta\in Q_+$. Recall the set $I^{\theta}_{\di}$ of  divided power words from \S\ref{SSLTN}. Following \cite[\S3.2a]{KlLi}, to every $\bi\in I^{\theta}_{\di}$, we associate the {\em divided power idempotent} $1_{\bi}\in R_\theta$.

\begin{Lemma} \label{L!} \cite[Lemma 3.2.4]{KlLi}
Let $\k=\F$ and $V$ be a finite dimensional graded $R_\theta$-supermodule. For any $\bi\in I^{\theta}_{\di}$ we have $\dim 1_{\bar \bi} V=\bi!\dim 1_\bi V$.  
\end{Lemma}

Let $\theta_1,\dots,\theta_n\in Q_+$ and $\theta=\theta_1+\dots+\theta_n$. Denote 
$\underline{\theta}:=(\theta_1,\dots,\theta_n)\in Q_+^n$, 
$$R_{\underline{\theta}}
=R_{\theta_1,\dots,\theta_n}:=R_{\theta_1}\otimes\dots\otimes R_{\theta_n}$$ and 
\begin{equation}\label{ETheta}
1_{\underline{\theta}}
=1_{\theta_1,\dots,\theta_n}:=\sum_{\bi^1\in I^{\theta_1},\dots,\bi^n\in I^{\theta_n}}1_{\bi^1\cdots\bi^n}\in R_{\theta}.
\end{equation}
By \cite[\S4.1]{KKO1}, we have the natural embedding 
\begin{equation}\label{EIota}
\iota_{\underline{\theta}}:R_{\underline{\theta}}\to 1_{\underline{\theta}}R_{\theta}1_{\underline{\theta}},
\end{equation}
and we identify $R_{\underline{\theta}}$ with a subalgebra of $1_{\underline{\theta}}R_{\theta}1_{\underline{\theta}}$ via this embedding. We refer to this subalgebra as a {\em parabolic subalgebra}. Note that $\iota_{\underline{\theta}}(1_{\theta_1}\otimes\dots\otimes 1_{\theta_n})=1_{\underline{\theta}}$. 

The special case $R_{\theta,\eta}\subseteq 1_{\theta,\eta}R_{\theta+\eta}1_{\theta,\eta}$ will be especially important.

\subsection{Imaginary cuspidal algebra $\hat C_d$}
Let $\preceq$ be a convex preorder on $\Phi_+$ chosen as in  \S\ref{SSRootSystem} and $d\in\N_+$. Recalling the notation $\wt(\bi)$ from (\ref{EWt}), a word $\bi\in I^{d\de}$ is called {\em cuspidal} if $\bi=\bj\bk$ for words $\bj,\bk$ implies that $\wt(\bj)$ is a sum of positive roots $\preceq \de$ and $\wt(\bk)$ is a sum of positive roots $\succeq\de$. We denote by 
$I^{d\de}_\cus$ the set of all cuspidal words in $I^{d\de}$ and let $(1_\bi \mid \bi\in I^{d\de}\setminus I^{d\de}_\cus)$ be the two-sided ideal of $R_{d\de}$ generated by all $1_\bi$ with $\bi\in I^{d\de}\setminus I^{d\de}_\cus$. 
The {\em (rank $d$) imaginary cuspidal algebra} is the quotient 
\begin{equation}\label{ECuspidalAlgebra}
\hat C_d:=R_{d\de}/(1_\bi \mid \bi\in I^{d\de}\setminus I^{d\de}_\cus)
\end{equation}
($\hat C_0$ is interpreted as $\k$). 
For an element $r\in R_{d\de}$ we often denote its image $r+(1_\bi \mid \bi\in I^{d\de}\setminus I^{d\de}_\cus)\in \hat C_d$ again by $r$. So we have elements $\psi_w,\, 1_\bi,\, y_k$, etc. in $C_d$. 

For $n\in\N_+$ and $\la=(\la_1,\dots,\la_n)\in\Comp(n,d)$, we denote  
$
\la\de:=(\la_1\de,\dots,\la_n\de)\in Q_+^n.
$ 
So, recalling (\ref{ETheta}), we have the idempotent $1_{\la\de}\in R_{d\de}$ and the parabolic subalgebra 
$R_{\la\de}\subseteq 1_{\la\de}R_{d\de}1_{\la\de}$ identified with $R_{\la_1\de}\otimes\dots\otimes R_{\la_n\de}$ via the embedding $\iota_{\la\de}$ of (\ref{EIota}). So we can consider $\hat C_{\la_1}\otimes\dots\otimes \hat C_{\la_n}$ as a quotient of $R_{\la\de}$. 
Define the {\em cuspidal parabolic subalgebra} $\hat C_{\la}\subseteq 1_{\la\de} \hat C_d 1_{\la\de}$ to be the image of $R_{\la\de}$ under the natural projection $1_{\la\de}R_{d\de}1_{\la\de}\onto 1_{\la\de}\hat C_d1_{\la\de}$.

\begin{Lemma}\label{L030216} 
{\rm \cite[Lemma 3.3.21]{KlLi}} 
The natural map
$$
R_{\la_1\de}\otimes\dots\otimes R_{\la_n\de}\stackrel{\iota_{\la\de}}{\longrightarrow} 1_{\la\de}R_{d\de}1_{\la\de}\onto 1_{\la\de}\hat C_d1_{\la\de}
$$
factors through the surjection $R_{\la_1\de}\otimes\dots\otimes R_{\la_n\de}\,\onto\, \hat C_{\la_1}\otimes\dots\otimes  \hat  C_{\la_n}$ and induces an isomorphism
$ \hat C_{\la_1}\otimes\dots\otimes  \hat C_{\la_n}\iso  \hat C_{\la}.$ 
\end{Lemma}

In view of the lemma, we identify 
\begin{equation}\label{ECuspPar}
 \hat C_{\la_1}\otimes\dots\otimes  \hat C_{\la_n}= \hat C_{\la}. 
\end{equation}
 
\section{Gelfand-Graev truncation}

\subsection{Gelfand-Graev truncated imaginary cuspidal algebra $C_d$}\label{SSGGW}
For $j\in J$ and $m\in \N$, following \cite{KIS}, we consider cuspidal (divided power) words
\begin{equation}\label{EGGW}
\ggw^{m,j}:=\ell^{(m)}(\ell-1)^{(2m)}\,\cdots\, ( j+1)^{(2m)} j^{(m)}\cdots\, 1^{(m)} 0^{(2m)}1^{(m)}\cdots\,  j^{(m)}\,\in\, I^{m\de}_\di. 
\end{equation}
Recalling (\ref{EBi!}), we have 
\begin{equation}\label{EGG!}
\ggw^{m,j}!=((2m)!)^{\ell-j}(m!)^{2j+1}.
\end{equation}

More generally, given $d\in \N$, $n\in\N_+$ and  
a colored composition $(\mu,\bj)\in\Comp^\col(n,d)$, we define the  corresponding {\em Gelfand-Graev (cuspidal divided power) word}
\begin{equation}\label{EGHatG}
\ggw^{\mu,\bj}:=\ggw^{\mu_1,j_1}\cdots \ggw^{\mu_n,j_n}\in I^{d\de}_\di
\end{equation}
and the {\em Gelfand-Graev idempotent} 
\begin{equation}\label{EGGIdempotent}
\ggi^{\mu,\bj}:=1_{\ggw^{\mu,\bj}}\in \hat C_d.
\end{equation}
Recalling (\ref{EOmd}), in the special case where $\mu=\om_d$ and $\bj\in I^d$, we denote 
\begin{equation}\label{EIdBj}
\ggi^{\bj}:=\ggi^{\om_d,\bj}=1_{\ggw^{1,j_1}\cdots \ggw^{1,j_d}}. 
\end{equation}
For example, in the case where $\bj=j^d$ we have the idempotent $\ggi^{j^d}=1_{(\ggw^{1,j})^d}$.

The idempotents $\{\ggi^{\mu,\bj}\mid (\mu,\bj)\in\EC^\col(d)\}$ in $\hat C_d$ are orthogonal. 
Define the idempotent 
\begin{equation}\label{EGad}
\ggi_d:=\sum_{(\mu,\bj)\in\EC^\col(d)}\ggi^{\mu,\bj}\in \hat C_d,
\end{equation}
and the {\em Gelfand-Graev truncated cuspidal algebra}
$$
C_d:=\ggi_d\hat C_d\ggi_d.
$$

\subsection{The idempotents $\ggi_\la$ and the parabolic subalgebra $C_\la\subseteq \ggi_\la C_d \ggi_\la$}
Let $\la=(\la_1,\dots,\la_n)\in\Comp(n,d)$. 
Recalling (\ref{ECuspPar}), we have the cuspidal parabolic subalgebra 
$
\hat C_{\la_1}\otimes\cdots\otimes \hat C_{\la_n}=\hat C_{\la}\subseteq 1_{\la\de}\hat C_d1_{\la\de}.
$ 
Define the idempotent 
$
\ggi_\la
:=\ggi_{\la_1}\otimes \dots\otimes \ggi_{\la_n}\in\hat C_{\la}.
$
Truncating with this idempotent we get the parabolic subalgebra 
\begin{eqnarray}\label{ECPar}
C_{\la_1}\otimes \cdots\otimes C_{\la_n}
&\cong& C_{\la}\,\,:=\,\,\ggi_{\la}\hat C_{\la}\ggi_{\la}\,\,\subseteq\,\, \ggi_{\la}C_d\ggi_\la.
\end{eqnarray}
As with $\hat C_{\la}$, we identify  $C_{\la}$ with $C_{\la_1}\otimes \cdots\otimes C_{\la_n}$.

In the special case $\la=\om_d$ the idempotent $f_{\om_d}$ is especially important. In particular:

\begin{Theorem} \label{TLargeChar} If $\k$ is a field with $\cha\k=0$ or  $\cha\k>d$ then $C_d\ggi_{\om_d}$ is a projective generator for $C_d$. 
\end{Theorem}
\begin{proof}
By \cite[Theorems 4.2.55,\,4.5.9]{KlLi}, we have that $\ggi_{\om_d}C_d\ggi_{\om_d}$ is graded Morita superequivalent to $C_d$. This implies the theorem. 
\end{proof}

\subsection{Regraded truncated imaginary cuspidal algebra $\zC_d$}
\label{SSRegradingCd}
Let $(\mu,\bj)\in\Comp^\col(m,d)$. 
Define 
\begin{equation}\label{EShiftsCuspidal}
t_{\mu,\bj}:=d(2\ell+1)+\sum_{s=1}^m\mu_s^2(2j_s-4\ell)
\quad\text{and}\quad
\eps_{\mu,\bj}:=\sum_{s=1}^m\mu_sj_s\pmod{2}.
\end{equation}

Recalling the general setting of \S\ref{SSRegr}, the parameters (\ref{EShiftsCuspidal}) will be taken as grading supershift parameters which correspond to the orthogonal decomposition (\ref{EGad}) in $C_d$. This yields the new graded superalgebra 
\begin{equation}\label{EReGradingC}
\zC_d=\bigoplus_{(\la,\bi),(\mu,\bj)\in\EC^\col(d)}
\funQ^{t_{\la,\bi}-t_{\mu,\bj}}\Uppi^{\eps_{\la,\bi}-\eps_{\mu,\bj}} \ggi^{\mu,\bj}C_d\ggi^{\la,\bi}.
\end{equation}

\begin{Notation}\label{Not}
We denote by $\zc\in\zC_d$ the element corresponding to an  element $c\in C_d$; for example, 
we have the idempotents $\ggis^{\la,\bi},\ggis_\nu,\ggis^\bj\in\zC_d$ corresponding to the idempotents $\ggi^{\la,\bi},\ggi_\nu,\ggi^\bj\in C_d$. 
\end{Notation}

For any $\la\in\Comp(n,d)$, the parabolic subalgebra $C_\la=C_{\la_1}\otimes \dots\otimes C_{\la_n}\subseteq \ggi_\la C_d\ggi_\la$ 
corresponds to the parabolic subalgebra 
\begin{equation}\label{ECParz}
\zC_{\la_1}\otimes \cdots\otimes \zC_{\la_n}
\cong \zC_{\la}\subseteq \ggis_{\la}\zC_d\ggis_\la.
\end{equation}
As with $C_{\la}$, we identify  $\zC_{\la}$ with $\zC_{\la_1}\otimes \cdots\otimes \zC_{\la_n}$.

Recall the notation (\ref{ELaBjj}).

\begin{Theorem} \label{TNonNeg} {\rm \cite[Theorem 5.2.23]{KIS}} 
The algebra $\zC_d$ is non-negatively graded. Moreover, $(\ggis^{\mu,\bj}\zC_d\ggis^{\la,\bi})^0=0$ unless $|\la,\bi|_k=|\mu,\bj|_k$ for all $k\in J$. 
\end{Theorem}

\subsection{The idempotent truncation $\ggis_{\om_d}\zC_d\ggis_{\om_d}$}
In this subsection, we first work in the algebra $C_d$ and then switch to $\zC_d$. 
In the special case $d=1$, we have $f_1=\sum_{j\in J}f^j$. 
Following \cite[\S4.2c]{KlLi}, we have the following special elements of $C_1$:
\begin{align}
\label{EUFormula}
\dot u&:=\ggi^0y_{2\ell+1}\ggi^0
\\
\label{EZFormula}
\dot z&:=\textstyle\sum_{j\in J}\ggi^jy_{2\ell+1-j-1}y_{2\ell+1-j}\ggi^j, 
\\
\label{EAII-1Formula}
\dot a^{[j,j-1]}&:=\ggi^j\psi_{2\ell+1-1}\cdots\psi_{2\ell+1-2j}\ggi^{j-1}\qquad (j=1,\dots,\ell-1),
\\
\label{EAI-1IFormula}
\dot a^{[j-1,j]}&:=(-1)^j\ggi^{j-1}\psi_{2\ell+1-2j-1}\cdots\psi_{2\ell+1-1}\ggi^{j}\qquad (j=1,\dots,\ell-1).
\end{align}

We have the parabolic subalgebra 
$
C_1^{\otimes d}=C_{\om_d} \subseteq \ggi_{\om_d}C_d\ggi_{\om_d}.
$ 
Given an element $x\in C_1$ and $1\leq r\leq d$, we denote 
\begin{equation*}\label{EAR}
x_r:=\ggi_1^{\otimes (r-1)}\otimes x\otimes \ggi_1^{\otimes(d-r)}\in C_1^{\otimes d}\subseteq \ggi_{\om_d}C_d\ggi_{\om_d}.
\end{equation*}
This notation agrees with (\ref{EInsertion}). 
The elements (\ref{EUFormula})-(\ref{EAI-1IFormula}) now yield the elements 
\begin{equation}\label{EElementsCd}
\dot u_r,\dot z_r, 
\dot a^{[i,j]}_r\in \ggi_{\om_d}C_d\ggi_{\om_d}\qquad(1\leq r\leq d).
\end{equation} 
Further, following \cite[(4.3.28)]{KlLi}, we set  
$$
\upsigma:=\ggi_{\om_2}(\psi_p\psi_{p+1}\cdots \psi_{2p-1})(\psi_{p-1}\psi_p\cdots \psi_{2p-2})\cdots(\psi_{1}\psi_{2}\cdots \psi_p)\ggi_{\om_2}\in \ggi_{\om_2}C_2\ggi_{\om_2}
$$
Now, following \cite[(4.4.11)]{KlLi}, we define $\dot s\in \ggi_{\om_2}C_2\ggi_{\om_2}$ via: 
\begin{equation}\label{ETau}
\dot s\ggi^{ij}=(\upsigma+(-1)^{i}\de_{i,j})\ggi^{ij}\qquad(i,j\in J).
\end{equation}
For $1\leq r<d$, we define 
\begin{eqnarray*}\label{EIotaTauR}
\dot s_{r}&:=&\ggi_1^{\otimes (r-1)}\otimes \dot s\otimes \ggi_1^{\otimes (d-r-1)}\in C_1^{\otimes (r-1)}\otimes \ggi_{\om_2}C_2\ggi_{\om_2}\otimes C_1^{\otimes (d-r-1)}\subseteq  \ggi_{\om_d}C_d\ggi_{\om_d}. 
\end{eqnarray*}
For $w\in \Si_d$ with an arbitrarily chosen reduced decomposition $w=s_{t_1}\cdots s_{t_l}$, we define
\begin{equation}\label{ETauSigma}
\dot w=\dot s_{t_1}\cdots\dot s_{t_l}\in \ggi_{\om_d}C_d\ggi_{\om_d}. 
\end{equation}

Recalling Notation~\ref{Not}, we now have the elements 
$
\dot \zu_r,\dot \zz_r,
\dot \za^{[i,j]}_r,
\dot \zw
$  
of the regraded algebra $\ggis_{\om_d}\zC_d\ggis_{\om_d}$ corresponding to the elements $
\dot u_r,\dot z_r,
\dot a^{[i,j]}_r,
\dot w\in \ggi_{\om_d}C_d\ggi_{\om_d}
$, respectively.

\begin{Theorem} \label{THCIsoZ} {\rm \cite[Theorem 5.1.16]{KIS}} 
There is an isomorphism of graded superalgebras
\begin{align*}
\zF_d:H_d(\Zig_\ell)&\iso  \ggis_{\om_d}\zC_d\ggis_{\om_d},\\ 
\ze^\bj&\mapsto \ggis^\bj,\\ 
\zu_r&\mapsto \dot \zu_r\sum_{\bj\in J^d}(-1)^{d+j_1+\dots+j_d}\ggis^\bj,\\ 
 \zz_r&\mapsto \dot \zz_r\sum_{\bj\in J^d}(-1)^{j_r+1}\ggis^\bj,\\ 
\za^{[i,j]}_r&\mapsto \dot \za^{[i,j]}_r\sum_{\bj\in J^d}(-1)^{j_1+\dots+j_{r-1}+r-1}\ggis^\bj,\\ 
s_r&\mapsto \dot \zs_r\sum_{\bj\in J^d}(-1)^{(j_r+1)(j_{r+1}+1)}\ggis^\bj.
\end{align*}
\end{Theorem}

\subsection{The idempotent truncation $\ggis^{\la,\bi}\zC_d\ggis_{\om_d}$}
Composing the isomorphism $\zF_d$ from  
Theorem~\ref{THCIsoZ} and the embedding $\iota^{\Zig_\ell}$ from (\ref{EIotaZig}), we have a graded superalgebra homomorphism
\begin{equation}
\label{EThetad}
\Theta_d:=\zF_d\circ \iota^{\Zig_\ell}: W_d(\Zig_\ell)\to \ggis_{\om_d}\zC_{d}\ggis_{\om_d}.
\end{equation}
Note that  
\begin{equation}\label{EThetaId}
\Theta_d(\ze^\bj)=\ggis^\bj \qquad(\bj\in J^d).
\end{equation}

Let $(\la,\bi)\in \Comp^\col(n,d)$. Then $\ggis^{\la,\bi}\zC_{d}\ggis_{\om_d}$ is naturally a right graded $\ggis_{\om_d}\zC_{d}\ggis_{\om_d}$-supermodule. 
So we have the structure of right graded $W_d(\Zig_\ell)$-supermodule on $\ggis^{\la,\bi}\zC_{d}\ggis_{\om_d}$ with
\begin{equation}\label{EThetaAction}
\zv\zw=\zv\Theta_d(\zw)\qquad (\zv\in \ggis^{\la,\bi}\zC_{d}\ggis_{\om_d},\ \zw\in W_d(\Zig_\ell)).
\end{equation}

For $j\in J$, we recall from \cite[\S5.2b]{KIS} the element 
$\lgath_{d,j}
$. Following our usual conventions, the same element in the regraded algebra $\zC_d$ is denoted $\lgathz_{d,j}$. 
More generally, for $(\la,\bi)\in\Comp^\col(n,d)$, we have the element 
\begin{equation}\label{EGath}
\lgathz_{\la,\bi}=\lgathz_{\la_1,i_1}\otimes\dots \lgathz_{\la_n,i_n}\in \zC_\la\subseteq \zC_d.
\end{equation}
Some key properties of these elements established in \cite{KIS} are collected in the following theorems. Recalling (\ref{EIdBj}), we use the notation
\begin{equation}\label{ETrickyIdNot}
\ggis(\la,\bi):=\ggis^{\,i_1^{\la_1}\cdots\, i_n^{\la_n}}=\ggis^{\,\om_d,\,i_1^{\la_1}\cdots\, i_n^{\la_n}}.
\end{equation}

\begin{Theorem} \label{TGath} Let $(\la,\bi)\in\Comp^\col(n,d)$. Then:
\begin{enumerate}
\item[{\rm (i)}] 
{\rm \cite[Lemma 5.2.12]{KIS}} 
We have $\lgathz_{\la,\bi}\in \ggis^{\la,\bi}\zC_d\ggis(\la,\bi)\subseteq \ggis^{\la,\bi}\zC_d\ggis_{\om_d}$. 
\item[{\rm (ii)}] {\rm \cite[Lemmas 5.2.13, 5.2.33, 5.2.34]{KIS}} 
The element $\lgathz_{\la,\bi}$ is non-zero and $\bideg(\lgathz_{\la,\bi})=(0,\0)$. Moreover, the degree $0$ component $\ggis^{\la,\bi}\zC_\la^0\ggis(\la,\bi)$ of the idempotent truncation $\ggis^{\la,\bi}\zC_\la\ggis(\la,\bi)$ is a free $\k$-module of rank $1$ spanned by 
$\lgathz_{\la,\bi}$. 
\item[{\rm (iii)}] {\rm \cite[Lemma 9.2.5]{KIS}}
For any $w\in\Si_\la$ considered as an element of $W_d(\Zig_\ell)$, we have $\lgathz_{\la,\bi}w=\lgathz_{\la,\bi}$ for the action (\ref{EThetaAction}). 

\item[{\rm (iv)}] {\rm \cite[Lemma 5.2.18]{KIS}}
We have $
\ggis^{\la,\bi}\zC_d\ggis_{\om_d}=\lgathz_{\la,\bi}\ggis_{\om_d}\zC_d\ggis_{\om_d}.
$
\end{enumerate}
\end{Theorem}

\begin{Theorem} \label{TGath2} {\rm \cite[Corollary 9.2.8]{KIS}} 
Let $(\la,\bi),(\mu,\bj)\in\Comp^\col(n,d)$. If  
$\zv\in\ggis^{\mu,\bj}\zC_d^0 \ggis(\la,\bi)$ satisfies $\zv w=\zv$ for all $w\in\Si_\la$ then $\zv=\zc\lgathz_{\la,\bi}$ for some $\zc\in \ggis^{\mu,\bj}\zC_d^0 \ggis^{\la,\bi}$. 
\end{Theorem}

Recall from (\ref{EMLaBi}) the graded right $W_d(\Zig_\ell)$-supermodule $M_{\la,\bi}
$ with generator $m_{\la,\bi}
$.

\begin{Lemma} \label{LUptheta} 
There exists a bidegree $(0,\0)$ homomorphism of right graded $W_d(\Zig_\ell)$-supermodules
$$
\upzeta_{\la,\bi}:M_{\la,\bi}\to \ggis^{\la,\bi}\zC_{d}\ggis_{\om_d},\ m_{\la,\bi}\mapsto \lgathz_{\la,\bi}.
$$
\end{Lemma}
\begin{proof}
Recalling the element $\ze^{\la,\bi}\in W_d(\Zig_\ell)$ from (\ref{EZeI}) we note that $\Theta_{d}(\ze^{\la,\bi})=\ggis(\la,\bi)$ by (\ref{EThetaId}). Since   
$\lgathz_{\la,\bi}\in \ggis^{\la,\bi}\zC_d\ggis(\la,\bi)$ by Theorem~\ref{TGath}(i), we deduce that 
$$\lgathz_{\la,\bi}\ze^{\la,\bi}=\lgathz_{\la,\bi}\Theta_{d}(\ze^{\la,\bi})=\lgathz_{\la,\bi}\ggis(\la,\bi)=\lgathz_{\la,\bi}.
$$
Moreover, by Theorem~\ref{TGath}(iii), for any $w\in\Si_\la$, we have $\lgathz_{\la,\bi}w=\lgathz_{\la,\bi}$. Since 
$\bideg(\lgathz_{\la,\bi})=(0,\0)$ by Theorem~\ref{TGath}(ii), there is a bidegree $(0,\0)$ 
homogeneous homomorphism of right graded $W_{\la,\bi}(\Zig_\ell)$-supermodules $\k_{\la,\bi}\to \ggis^{\la,\bi}\zC_{d}\ggis_{\om_d},\ 1_\k\mapsto \lgathz_{\la,\bi}$. By Frobenius Reciprocity, this homomorphism induces the desired homomorphism $\upzeta_{\la,\bi}$.
\end{proof}

\section{Cyclotomic quiver Hecke superalgebra}
\subsection{Cyclotomic quiver Hecke superalgebra $H_\theta$}
\label{SSHTheta}

Let $\theta\in Q_+$. The {\em cyclotomic quiver Hecke superalgebra $H_\theta$} is the graded $\k$-superalgebra defined as $R_\theta$ modulo the relations
\begin{equation}
y_1=0\quad \text{and}\quad  1_\bi=0\ \text{if}\ i_1\neq 0 \ \ (\text{for all}\ \bi=i_1\cdots i_n\in I^\theta).
\label{RCyc}
\end{equation}
We have the natural projection map
\begin{equation}\label{EPi}
\pi_\theta \colon R_\theta\onto H_\theta.
\end{equation} 
For an element $x\in R_\theta$, we often denote $\pi_\theta(x)\in H_\theta$ again by $x$. Thus we have elements of the form $\psi_w, y_r, 1_\bi$, etc. in $H_\theta$. 

We use the additional index $\k$ to write $H_\theta=H_{\theta,\k}$ when it is important to emphasize the ground ring we are working over. 
If $\F$ is a field of fractions of $\k$ or a quotient field of $\k$ then 
$H_{\theta,\F}\cong \F\otimes_\k H_{\theta,\k}$, and we always identify the two $\F$-algebras. 

For $\theta,\eta\in Q_+$, we have the {\em cyclotomic parabolic subalgebra} 
\begin{equation}\label{ECycPar}
H_{\theta,\eta}:=\pi_{\theta+\eta}(R_{\theta,\eta})\subseteq 1_{\theta,\eta}H_\theta 1_{\theta,\eta}.
\end{equation}
The natural embedding 
$R_{\theta}\to R_{\theta,\eta}, \ x\mapsto 
x\otimes 1_\eta
$  followed by 
The map $\pi_{\theta+\eta}$ factors through the quotient $H_\theta$ to give the natural {\em unital} algebra homomorphism
\begin{equation}\label{EZetaHom}
\zeta_{\theta,\eta}: H_\theta\to H_{\theta,\eta}.
\end{equation}

Recall the set of weights $\Weights$ from (\ref{EWeights}).

\begin{Theorem} \label{TCyc} 
We have:
\begin{enumerate}
\item[{\rm (i)}] As a $\k$-module, $H_\theta$ is free of finite rank.
\item[{\rm (ii)}] $H_\theta\neq 0$ if and only if $\theta\in\Weights$.
\end{enumerate}
\end{Theorem}
\begin{proof}
(i) By \cite[Theorem 8.7]{KKO1}, $1_{\theta,\al_i}H_{\theta+\al_i}$ is finitely generated projective as a left $H_{\theta}$-module, cf. the proof of \cite[Theorem 8.9]{KKO1}. Now the lemma follows by induction as in \cite[Remark 4.20(ii)]{KK}. (Our running  assumption that $2$ is invertible in $\k$ might not be necessary but it is the running assumption in \cite{KKO1}, see \cite[(4.1)]{KKO1}). 

(ii) When $\k$ is a field (of characteristic $\neq 2$) it follows from the Kang-Kashiwara-Oh categorification theorem \cite[Theorem 10.2]{KKO1} that $H_\theta\neq 0$ if and only if the weight space $V(\La_0)_{\La_0-\theta}\neq 0$. 
The case where $\k$ is a PID (in which $2$ is invertible)  follows from the field case by extension of scalars using (i). 
\end{proof}

\begin{Lemma} \label{LIrrHTheta} {\rm \cite[Lemma 3.1.10]{KlLi}}
If\, $\k$ is a field and $d=d(\theta)$ then $|\Irr(H_\theta)|=|\Par^\ell( d)|$.
\end{Lemma}

\begin{Lemma} \label{LIdO}
Let $\F$ is a field of fractions of $\k$ or a quotient field of $\k$, and $e,f\in H_{\theta,\k}$ be idempotents. Denote $e_\F:=1_\F\otimes e\in H_{\theta,\F}$ and $f_\F:=1_\F\otimes f\in H_{\theta,\F}$. Then $eH_{\theta,\k}f$ is free of finite rank as a $\k$-module and 
$e_{\F}H_{\theta,\F}f_{\F}\cong \F\otimes_\k eH_{\theta,\k}f$. In particular, 
$e$ is non-zero in $H_{\theta,\k}$ if and only if $e_{\F}$ is non-zero in $H_{\theta,\F}$. 
\end{Lemma}
\begin{proof}
By Theorem~\ref{TCyc}(i), $H_{\theta,\k}$ is a free $\k$-module of finite rank. 
Since $eH_{\theta,\k}f$ is a direct summand of $H_{\theta,\k}$ and $\k$ is a PID, $eH_{\theta,\k}f$ is also a free  $\k$-module of finite rank.  Moreover, $e_{\F}H_{\theta,\F}f_{\F}\cong \F\otimes_\k eH_{\theta,\k}f$. 
Now, $e\neq 0$ if and only if $eH_{\theta,\k}e\neq 0$, and $e_\F\neq 0$ if and only if $e_\F H_{\theta,\k}e_\F\neq 0$. 
As  $eH_{\theta,\k}e$ is free, we have $eH_{\theta,\k}e\neq 0$ if and only if $e_\F H_{\theta,\k}e_\F\cong \F\otimes_\k eH_{\theta,\k}e\neq 0$. 
\end{proof}

\subsection{The algebra $H_{\rho}$}
Recall the notation of  \S\ref{SSW}. Let $\rho\in \Nuclei$, i.e. $\rho=\La_0-w\La_0$ for some $w\in W$. 
Let 
$w=r_{i_t}\cdots r_{i_1}$ be a reduced decomposition and for 
$k=1,\dots,t$, set 
$
a_k:=( r_{i_{k-1}}\cdots r_{i_1}\La_0\mid \al_{i_k}^\vee). 
$ 
It is well-known that $a_1,\dots,a_t\in\N$,  
see for example \cite[Lemma 2.4.11]{KlLi}. Define 
\begin{equation}\label{EIRho}
\bi_\rho:=i_1^{(a_1)}\cdots\,  i_t^{(a_t)}\in I^\rho_\di \qquad\text{and}\qquad
\bar\bi_\rho:=i_1^{a_1}\cdots\,  i_t^{a_t}\in I^\rho
\end{equation}
(note that $\bi_\rho$ and $\bar\bi_\rho$ in general depend on the choices made). 

\begin{Lemma} \label{LMatrix} 
Let $\rho\in \Nuclei$. Then $1_{\bi_\rho}H_{\rho}1_{\bi_\rho}$ is free $\k$-module of rank $1$, and the natural map $H_{\rho}\to \End_\k(H_{\rho}1_{\bi_\rho})$ is an isomorphism of graded $\k$-superalgebras. 
In particular, $H_{\rho}$ is a graded matrix superalgebra over $\k$. 
\end{Lemma}
\begin{proof} 
In the case where $\k$ is a field, this is \cite[Lemma 3.2.5]{KlLi}. The general case is reduced to the case of the field using Lemma~\ref{LIdO}, cf. the proof of \cite[Lemma 5.8]{EK2}.
\end{proof}

\subsection{The centralizer algebra $\Cent_{\rho,d}$}
Suppose $H_\theta\neq 0$, i.e. $\theta\in\Weights$, see Theorem~\ref{TCyc}. Recall the homomorphism (\ref{EZetaHom}) and the notation of  \S\ref{SSW}.

\begin{Lemma} \label{LZeta} 
Let $\theta\in \Weights$, $\rho=\rho(\theta)$ and $d=d(\theta)$. 
Then the homomorphism 
$
\zeta_{\rho,d\de}:H_{\rho}\to 1_{\rho,d\de}H_\theta 1_{\rho,d\de}
$ 
is injective. 
\end{Lemma}
\begin{proof}
By Theorem~\ref{TCyc}(i), it suffices to prove that the map is injective when $\k$ is a field. The field case follows from \cite[Corollary 3.2.7]{KlLi}.
\end{proof}

Let $\theta\in \Weights$, $\rho=\rho(\theta)$ and $d=d(\theta)$.  By Lemma~\ref{LZeta}, we can identify 
\begin{equation}\label{ECoreSubalg}
H_{\rho}=
\zeta_{\rho,d\de}(H_{\rho})\subseteq 
1_{\rho,d\de}H_\theta 1_{\rho,d\de}.
\end{equation}
We consider the supercentralizer of $H_{\rho}$ in $1_{\rho,d\de}H_\theta 1_{\rho,d\de}$:
\begin{equation}\label{ECentralizer}
\Cent_{\rho,d} := \Cent_{1_{\rho,d\de}H_\theta 1_{\rho,d\de}} (H_{\rho}).
\end{equation}

\begin{Lemma} \label{LJuly4} 
Let $\theta\in \Weights$, $\rho=\rho(\theta)$ and $d=d(\theta)$.  We have an isomorphism of   graded superalgebras  
$$
H_{\rho}\otimes \Cent_{\rho,d}\iso 1_{\rho,d\de}H_\theta 1_{\rho,d\de}, \ h\otimes z\mapsto hz.
$$
\end{Lemma}
\begin{proof}
This follows from Lemmas~\ref{LMatrix} and \ref{LEvseev}.
\end{proof}


\begin{Lemma} \label{LIdSymNew} 
Let $\theta\in \Weights$, $\rho=\rho(\theta)$ and $d=d(\theta)$.  
For the idempotent 
$\eps:=\iota_{\rho,d\de}(1_{\bi_\rho}\otimes 1_{d\de})$, we have an isomorphism of graded algebras $\eps H_\theta\eps\cong \Cent_{\rho,d}.$
In particular, 
\begin{enumerate}
\item[{\rm (i)}] As a $\k$-module, $\Cent_{\rho,d}$ is free of finite rank. Moreover, if $\F$ is a field of fractions of $\k$ or a quotient field of $\k$ then $\Cent_{\rho,d,\F}\cong \F\otimes_\k \Cent_{\rho,d}$. 
\item[{\rm (ii)}] $\Cent_{\rho,d}$ is a symmetric superalgebra 
whenever $H_\theta$ is a symmetric superalgebra. 
\end{enumerate}
\end{Lemma}
\begin{proof}
We have $\eps H_\theta\eps\cong \Cent_{\rho,d}$ by 
Lemmas~\ref{LJuly4} and \ref{LMatrix}. 
Now (i) follows from Lemma~\ref{LIdO} and (ii) follows from  Lemma~\ref{LSymIdTr}. 
\end{proof}

\begin{Lemma} \label{LBiBiRho}
Let $\k$ be a field, $\theta\in \Weights$, $\rho=\rho(\theta)$, $d=d(\theta)$, $\bi\in I^\rho_\di$ and $\bj\in I^{d\de}_\di$. If 
$V\in\mod{H_\theta}$ then $\dim 1_{\bi\bj}V=(\dim 1_{\bi}H_\rho1_{\bi_\rho})(\dim 1_{\bi_\rho\bj}V)$.
\end{Lemma}
\begin{proof}
Note that $1_{\bi\bj}=1_{\bi\bj}\zeta_{\rho,d\de}(1_\bi)$, so it suffices to prove that 
$$\dim \zeta_{\rho,d\de}(1_\bi)V=(\dim 1_{\bi}H_\rho1_{\bi_\rho})(\dim \zeta_{\rho,d\de}(1_{\bi_\rho})V).
$$
Considering $1_{\rho,d\de}V$ as an $H_\rho$-module via the embedding 
$
\zeta_{\rho,d\de}:H_{\rho}\to 1_{\rho,d\de}H_\theta 1_{\rho,d\de}
$ 
of Lemma~\ref{LZeta}, it suffices to prove for any finite dimensional $H_\rho$-module $W$ that 
\begin{equation}\label{E190924_2}
\dim 1_{\bi}W=(\dim 1_{\bi}H_\rho1_{\bi_\rho})(\dim 1_{\bi_\rho}W).
\end{equation}
But by Lemma~\ref{LMatrix}, $W$ is a direct sum of modules of the from $H_\rho1_{\bi_\rho}$. As $\dim 1_{\bi_\rho}H_\rho1_{\bi_\rho}=1$, we actually have $W\simeq (H_\rho1_{\bi_\rho})^{\oplus \dim1_{\bi_\rho}W}$, which implies (\ref{E190924_2}).
\end{proof}

\subsection{Kang-Kashiwara-Tsuchioka isomorphism}
\label{SSKKT}

Throughout this subsection we assume that $\k$ is an algebraically closed field of characteristic different from $2$. 

Recall from Example~\ref{ExT} the twisted group superalgebra $\cT_n$ of the symmetric group $\Si_n$. 
The {\em spin Jucys-Murphy elements} $\cm_1,\dots,\cm_n\in\cT_n$
are defined inductively via:
\begin{equation*}\label{ERecY}
\cm_1=0,\quad \cm_{r+1}=-\ct_r\cm_r\ct_r+\ct_r.
\end{equation*}
Note that the elements $\cm_r$ are odd, so the elements $\cm_r^2$ are even. By \cite[Theorem 3.2]{BKdurham}, the elements $\cm_1^2,\dots,\cm_n^2$ commute. 

Suppose first that $\cha \k>0$ and write $\cha \k=2\ell+1$. We consider the elements $i\in I=\{0,1,\dots,\ell\}$ as elements of $\k$, i.e. we identify $i=i\cdot 1_\k$. 

Let $V$ be a finite dimensional $\cT_n$-supermodule. For any $\bi=i_1\cdots i_n\in I^n$, we consider the simultaneous generalized eigenspace 
\begin{equation}\label{ETWtSp}
\hspace{3mm}
V_\bi:=\{v\in V\mid (\cm_r^2-i_r(i_r+1)/2)^Nv=0\ \text{for $N\gg0$ and $r=1,\dots,n$}\}.\index{$V_\bi$}
\end{equation}
By \cite[Lemma 3.3]{BKdurham}, we then have the {\em weight space decomposition}\,
$
V=\bigoplus_{\bi\in I^n}V_\bi.
$ 

Considering the weight space decomposition of the regular $\cT_n$-module, we deduce that there is a system $\{e(\bi) \mid \bi \in I^n\}$ of mutually orthogonal even idempotents in $\cT_n$ summing to the identity, uniquely determined by the property that $e(\bi)V = V_\bi$ for each $V\in\mod{\cT_n}$ (note that some of the $e(\bi)$ might be zero).

For  $\theta\in Q_+$ with $\height(\theta)=n$ we set
$
e_\theta:=\sum_{\bi\in I^\theta}e(\bi)\in\cT_n.
$
In this way we get a family of orthogonal central idempotents $\{e_\theta\mid \theta\in Q_+,\, \height(\theta)=n\}$ (some of which could be zero), cf. \cite{BKdurham}, \cite[\S5.3a]{KlLi}, and we have the corresponding {\em superblocks} $\cT_\theta:=e_\theta\cT_n$ with 
$$
\cT_n=\bigoplus_{\theta\in Q_+, \,  \height(\theta)=n}\cT_\theta.
$$
The (non-zero) superblocks $\cT_\theta$ are symmetric superalgebras upon restriction of the symmetrizing form $\t$ from $\cT_n$ to $\cT_\theta$, cf. Example~\ref{ExT}. 

As explained in \cite[\S5.3a]{KlLi}, the superblocks $\Blo^{\rho,d}$ appearing in the introduction are related to the superblocks $\cT_\theta$ as follows:
\begin{equation}\label{EBloRel}
\Blo^{\rho,d}=\cT_{\cont(\rho)+d\de}
\end{equation}
where $\cont(\rho)$ is as in (\ref{ECont}), see \cite[(2.3.7)]{KlLi}. Moreover, if $\theta=\cont(\rho)+d\de$ then $\cont(\rho)=\rho(\theta)$ and $d=d(\theta)$, see \S\ref{SSW} and \cite[Lemma 3.1.39]{KlLi}. 

Recall from Example~\ref{ExTC} the wreath superproduct 
$
\cY_n\cong\cT_n\otimes \cC_n.
$ 
We identify $\cY_n$ with $\cT_n\otimes \cC_n$ via this isomorphism described explicitly in \cite[Proposition 5.1.3]{KlLi} and  \cite[Lemma 13.2.3]{Kbook}. 
Recalling the idempotents $e_\theta\in\cT_n$, we have the idempotents $\{e_\theta\otimes 1\mid \theta\in Q_+,\, \height(\theta)=n\}$ in $\cY_n$, cf. \cite{BKdurham}, \cite[\S5.3b]{KlLi}, and we have the corresponding {\em superblocks} 
$\cY_\theta:=(e_\theta\otimes 1)\cY_n\cong \cT_\theta\otimes \cC_n$ (some of which could be zero) with 
$$
\cY_n=\bigoplus_{\theta\in Q_+, \,  \height(\theta)=n}\cY_\theta.
$$
The (non-zero) superblocks $\cY_\theta$ are symmetric superalgebras upon restriction of the symmetrizing form $\t$ from $\cY_n$ to $\cY_\theta$, cf. Example~\ref{ExTC}. 

Let 
$$\hat I:=\{0,1,1',2,2',\dots,\ell,\ell'\}\index{i@$\hat I$}$$ 
We have the embedding $I\subset \hat I$ and  the surjection 
$$
\pr:\hat I\to I,\ 0\mapsto 0,\ i\mapsto i,\ i'\mapsto i
\qquad(\text{for}\ i=1,\dots,\ell).
$$
So we also have $I^n\subseteq \hat I^n$ and the surjection 
$
\pr:\hat I^n\to I^n,\ \hat i_1\cdots \hat i_n\mapsto \pr(\hat i_1)\cdots \pr(\hat i_n).
$ 

The following theorem is proved in {\rm \cite[Theorem 5.4]{KKT}}, cf. \cite[Lemma 5.3.26,\,Proposition 5.3.27,\,Theorem 5.3.28]{KlLi} for more details.

\begin{Theorem} \label{TKKT1} Let $\k$ be an algebraically closed field of odd characteristic $2\ell+1$ and $\theta=\sum_{i\in I}m_i\al_i\in Q_+$ with $n=\height(\theta)$. Then there are even orthogonal idempotents $\{\mathpzc{e}(\hat\bi)\mid \hat\bi\in\hat I^n\}$ in $\cY_n$ summing to the identity 
with the following properties:
\begin{enumerate}
\item[{\rm (i)}] For any $\bi\in I^n$, in $\cY_n=\cT_n\otimes\cC_n$ we have
$$e(\bi)\otimes 1=\sum_{\hat\bi\in\hat I^n\,\text{with}\,\,\pr(\hat \bi)=\bi}\mathpzc{e}(\hat\bi).$$ Moreover, for any $\hat \bi\in \hat I^n$ and any finite dimensional $\cY_n$-supermodule $V$, 
$$(e(\pr(\hat\bi))\otimes 1)V\simeq (\mathpzc{e}(\hat\bi)V)^{\bigoplus 2^{n-m_0}}.$$ 
\item[{\rm (ii)}] Using the embedding $I^\theta\subseteq \hat I^n$, we have the idempotent 
$\mathpzc{e}:=\sum_{\bi\in I^\theta}\mathpzc{e}(\bi)$ in $\cY_\theta$ such that $\mathpzc{e}\cY_\theta \mathpzc{e}$ is Morita superequivalent to $\cY_\theta$. Moreover, there is an isomorphism of superalgebras 
$H_\theta\otimes\cC_{m_0}\cong \mathpzc{e}\cY_\theta \mathpzc{e}$ under which the idempotent $1_\bi\otimes 1\in  H_\theta\otimes\cC_{m_0}$ corresponds to the idempotent $\mathpzc{e}(\bi)$. 
\end{enumerate}
\end{Theorem}
\begin{proof}
We explain how to apply the results in \cite{KlLi} to deduce the desired results. By  \cite[(5.3.16), Lemma 5.3.17, Theorem 5.3.28]{KlLi}, there is an isomorphism $\cY_n\cong RC_n^{\La_0}$ of superalgebras, where $RC_n^{\La_0}$ is the quiver Hecke-Clifford superalgebra defined in \cite{KKT}, see \cite[\S5.3c]{KlLi}. The quiver Hecke-Clifford superalgebra comes with the orthogonal idempotents $\{\mathpzc{e}(\hat\bi)\mid \hat\bi\in\hat I^n\}$ summing to the identity. Moreover, as shown in the proof of \cite[Theorem 5.3.28]{KlLi}, under this isomorphism, the idempotent $e(\bi)\otimes 1$ corresponds to the idempotent $\sum_{\hat\bi\in\hat I^n\,\text{with}\,\,\pr(\hat \bi)=\bi}\mathpzc{e}(\hat\bi)$ (for any $\bi\in I^n$).
To complete the proof of (i), let $\hat\bi\in \hat I^n$, $\bi=\pr(\hat\bi)$, and $V$ is a finite dimensional $RC_n^{\La_0}$-module. 

Moreover, by the last relation in (i) in \cite[Definition 3.5]{KKT}, for the Clifford generators $\cc_1,\dots,\cc_n$ of $RC_n^{\La_0}$, we have 
$$\cc_r\mathpzc{e}(\hat\bi)=\mathpzc{e}(c_r(\hat\bi))\cc_r\qquad(1\leq r\leq n,\,\hat\bi\in\hat I^n),$$
where $c_r(\hat\bi)=\hat\bj$ is defined as $\hat j_s=\hat i_s$ for $s\neq r$, $\hat j_r=0$ if $\hat i_r=0$, $\hat j_r=i'$ if $\hat i_r=i\neq 0$, and $\hat j_r=i$ if $\hat i_r=i'\neq 0$. Since the elements $\cc_r$ are invertible, for any finite dimensional $RC_n^{\La_0}$-supermodule $V$, it follows that $\mathpzc{e}(\hat\bi)V\simeq \mathpzc{e}(c_r(\hat\bi))V$, hence $\mathpzc{e}(\hat\bi)V\simeq \mathpzc{e}(\hat\bj)V$ if $\pr(\hat\bi)=\pr(\hat\bj)$. 
Now, if $\bi\in I^\theta$ then the word  $\bi$ has exactly $m_1+\dots+m_\ell=n-m_0$ non-zero entries so there are exactly $2^{n-m_0}$ idempotents in the sum 
$\sum_{\hat\bj\in\hat I^n\,\text{with}\,\,\pr(\hat \bj)=\bi}\mathpzc{e}(\hat\bj)$, which implies that for any $\hat\bi\in \hat I^n$ with $\pr(\hat\bi)=\bi$ we have 
$$\Big(\sum_{\hat\bj\in\hat I^n\,\text{with}\,\,\pr(\hat \bj)=\bi}\mathpzc{e}(\hat\bj)\Big)V\simeq (\mathpzc{e}(\hat\bi)V)^{\bigoplus 2^{n-m_0}}.$$ 

Now (ii) follows by application of \cite[Lemma 5.3.26,\,Proposition 5.3.27]{KlLi}. 
\end{proof}

We will also need

\begin{Theorem}\label{T5.3.31} {\rm \cite[Theorem 5.3.31]{KlLi}}
Let $\k$ be an algebraically closed field of odd characteristic $2\ell+1$ and $\Blo^{\rho,d}$ be a superblock of $\k\tilde\Si_n$. 
Write $\theta:=\cont(\rho)+d\de$ and $\cont(\rho) = \sum_{i\in I}r_i\al_i$. Then 
$
\Blo^{\rho,d}\sim_{\sM} H_\theta\otimes \cC_{n-r_0}.
$
\end{Theorem}

We now also allow the case $\cha \k=0$. 
Recall from Example~\ref{ExTC(q)} the Olshanski's superalgebra 
$\cY_n(q)$. The superalgebra $\cY_n(q)$ comes with a family of orthogonal central idempotents $\{e_\theta\mid \theta\in Q_+,\, \height(\theta)=n\}$, cf. \cite{BKdurham}, \cite[\S5.3b]{KlLi}, and we have the corresponding {\em superblocks} $\cY_\theta(q):=e_\theta\cY_n(q)$ with 
$$
\cY_n(q)=\bigoplus_{\theta\in Q_+, \,  \height(\theta)=n}\cY_\theta(q)
$$
(using the same symbol for the idempotents $e_\theta\in\cT_n$ and $e_\theta\in \cY_n(q)$ should lead to no confusion). 
The (non-zero) superblocks $\cY_\theta(q)$ are symmetric superalgebras upon restriction of the symmetrizing form $\t$ from $\cY_n(q)$ to $\cY_\theta(q)$, cf. Example~\ref{ExTC(q)}.

The following theorem is proved in {\rm \cite[Corollary 4.8]{KKT}}, cf. \cite[Lemma 5.3.26,\,Proposition 5.3.27,\,Theorem 5.3.29]{KlLi}.

\begin{Theorem} \label{TKKT2} Let $\k$ be an algebraically closed field of characteristic different from $2$, the parameter $q\in\k^\times$ be a primitive $(2\ell+1)$st root of unity, and $\theta=\sum_{i\in I}m_i\al_i\in Q_+$. Then there is an even  idempotent $\mathpzc{e}\in  \cY_\theta(q)$ and an isomorphism of superalgebras 
$H_\theta\otimes\cC_{m_0}\cong \mathpzc{e}\cY_\theta(q) \mathpzc{e}$. Moreover, $\mathpzc{e}\cY_\theta(q) \mathpzc{e}$ is Morita superequivalent to $\cY_\theta(q)$. 
\end{Theorem}

The assumptions made on the ground field in the following lemma are probably unnecessary, cf. \cite[Remark 4.2.58]{KlLi},  but we need them to use Theorems \ref{TKKT1},\,\ref{TKKT2}. Note that the case $\cha \k=2\ell+1$, which is most important for this paper, is not excluded.  

\begin{Lemma} \label{LSymmConditions} 
Let $\k$ be an algebraically closed field of characteristic different from $2$. Exclude the case where $0<\cha \k<2\ell+1$ and $\cha \k$ divides $2\ell+1$. Then, for any $\theta\in Q_+$, the superalgebra $H_\theta$ is a symmetric superalgebra. 
\end{Lemma}
\begin{proof}
The superalgebras $\cY_\theta$ and  $\cY_\theta(q)$, and hence their idempotent truncations $e\cY_\theta e$ and $e\cY_\theta(q) e$ are symmetric superalgebras, see Examples~\ref{ExTC},\,\ref{ExTC(q)} and Lemma~\ref{LSymIdTr}. It follows from the assumptions on $\k$ and $q$ and Theorems \ref{TKKT1},\,\ref{TKKT2} that the superalgebra $H_\theta\otimes\cC_{m_0}$ is also a symmetric superalgebra, where $m_0$ comes from  $\theta=\sum_{i\in I}m_i\al_i$. 
Now $H_\theta$ is a symmetric superalgebra by Lemma~\ref{LSymTensCliff}. 
\end{proof}

\subsection{Correspondence of irreducible modules and formal characters} The material of this subsection will only be used in \S\ref{SSGGFragm}. 
In this subsection we assume that $\k$ is an algebraically closed field of odd characteristic $2\ell+1$. 

Let $\theta=\sum_{i\in I}m_i\al_i\in Q_+$ and $D$ be an irreducible $\cT_\theta$-supermodule. By \cite[Theorem 22.3.1]{Kbook}, we have that $D$ is of type $\Mtype$ if $n-m_0$ is even and of type $\Qtype$ if $n-m_0$ is odd. Let $U_n$ be the Clifford supermodule, i.e. the only irreducible $\cC_n$-supermodule up to isomorphism, see \cite[Example 12.2.4]{Kbook}. We have that $U_n$ is of type $\Mtype$ if $n$ is even and of type $\Qtype$ if $n$ is odd. Moreover, $\dim U_n=2^{\lceil n/2\rceil}$. By \cite[Lemma 12.2.13]{Kbook}, we have a bijection
$$
\Irr(\cT_\theta)\iso \Irr(\cY_\theta),\ D\mapsto D\circledast U_n,
$$
with $D\circledast U_n$ of type $\Mtype$ if and only if $D$ and $U_n$ have the same type, i.e. if and only if $n-m_0$ and $n$ have the same parity, i.e. if and only if $m_0$ is even. Moreover, 
$$
\dim (e(\bi)\otimes 1_{\cC_n})(D\circledast U_n)=
\left\{
\begin{array}{ll}
2^{\lceil n/2\rceil-1}\dim e(\bi)D &\hbox{if $n$ and $n-m_0$ are odd,}\\
2^{\lceil n/2\rceil}\dim e(\bi)D &\hbox{otherwise.}
\end{array}
\right.
$$
Applying the Morita superequivalence and the isomorphism $\cY_\theta\sim_{\sM}
\mathpzc{e}\cY_\theta \mathpzc{e}\cong H_\theta\otimes\cC_{m_0}$ of Theorem~\ref{TKKT1}, we now obtain a bijection 
$$
\Irr(\cT_\theta)\iso \Irr(H_\theta\otimes\cC_{m_0}),\ D\mapsto \mathpzc{e}(D\circledast U_n),
$$
with $\mathpzc{e}(D\circledast U_n)$ of type $\Mtype$ if and only if $m_0$ is even, and 
$$\dim (1_\bi\otimes 1_{\cC_{m_0}})\mathpzc{e}(D\circledast U_n)
=
2^{m_0-n}\dim (e(\bi)\otimes 1)(D\circledast U_n).
$$
Finally, taking into account that all irreducible $H_\theta$-supermodules are of type $\Mtype$ by \cite[Proposition 6.15]{HW}, \cite[\S4.2]{KKO1}, \cite[Theorem 8.6]{KKO2},
we have a bijection 
$$
\Irr(H_\theta)\iso\Irr(H_\theta\otimes\cC_{m_0}),\ L\mapsto L\boxtimes U_{m_0}
$$
with 
$$
\dim (1_\bi\otimes 1_{\cC_{m_0}})(L\boxtimes U_{m_0})=2^{\lceil m_0/2\rceil}\dim 1_\bi L.
$$
Composing this bijection with the inverse of previous one, we get a bijection
$
\beta:\Irr(H_\theta)\iso \Irr(\cT_\theta)
$
with 
$
\dim e(\bi)\beta(L)=2^{x(\theta)}\dim 1_\bi L$
where
$$
x(\theta)=\left\{
\begin{array}{ll}
\lceil m_0/2\rceil+n-m_0-\lceil n/2\rceil+1 &\hbox{if $n$ and $n-m_0$ are odd,}\\
\lceil m_0/2\rceil+n-m_0-\lceil n/2\rceil &\hbox{otherwise,}
\end{array}
\right.
$$
which is easily seen to be $\lceil(n-m_0)/2\rceil$.

We have proved

\begin{Theorem} \label{TWtHWtT}
Let $\theta\in Q_+$. 
There is a bijection $
\beta:\Irr(H_\theta)\iso \Irr(\cT_\theta)
$ such that for any $L\in \Irr(H_\theta)$ we have $\dim e(\bi)\beta(L)=2^{\lceil(n-m_0)/2\rceil}\dim 1_\bi L$. 
\end{Theorem}

\section{RoCK blocks}
\label{SRock}
Let $\theta\in \Weights$, $\rho=\rho(\theta)$ and $d=d(\theta)$. 
Note that $(\theta\mid \al_i^\vee)=(\rho\mid \al_i^\vee)$ for all $i\in I$. 
We say that $\theta$ is {\em RoCK}\, if $(\theta\mid \al_0^\vee)\geq 2d$ and $(\theta\mid \al_i^\vee)\geq d-1$ for $i=1,\dots,\ell$. 
This is equivalent to the cyclotomic quiver Hecke algebra $H_\theta$ being a {\em RoCK block}, as defined in \cite[\S4.1]{KlLi}. 

Throughout the section, we assume that $\theta$ is RoCK, $\rho=\rho(\theta)$ and $d=d(\theta)$.

\subsection{RoCK blocks and the parabolics $H_{\rho,d\de}$}

Recall that in general we have a parabolic subalgebra $H_{\theta_1,\theta_2}\subseteq 1_{\theta_1,\theta_2}H_{\theta_1+\theta_2}1_{\theta_1,\theta_2}$. For RoCK blocks and $\theta_1=\rho,\, \theta_2=d\de$ we actually have the equality:

\begin{Lemma} \label{LKlLi4.1.18}
We have $1_{\rho,d\de}H_\theta 1_{\rho,d\de}=H_{\rho,d\de}$. 
\end{Lemma}
\begin{proof}
When $\k$ is a field, this is a special case of \cite[Lemma 4.1.18]{KlLi}. For a general $\k$, as in the proof of \cite[Lemma 4.1.18]{KlLi}, it suffices to show that $1_{\rho,d\de}\psi_w1_{\rho,d\de}=0$ in $H_{\theta,\k}$ for certain some special $w$'s. Extending scalars to the field of fractions $\K$ of $\k$, and using the field case, we see that $1_{\rho,d\de}\psi_w1_{\rho,d\de}=0$ in $H_{\theta,\K}$, so $1_{\rho,d\de}\psi_w1_{\rho,d\de}=0$ in $H_{\theta,\k}$, cf. Theorem~\ref{TCyc}. 
\end{proof}

Recall from (\ref{ECoreSubalg}) (and Lemma~\ref{LZeta}) that we have identified 
$
H_{\rho}$ with the subalgebra 
$\zeta_{\rho,d\de}(H_{\rho})\subseteq 
1_{\rho,d\de}H_\theta 1_{\rho,d\de}.
$ 
By Lemma~\ref{LKlLi4.1.18}, the latter algebra is $H_{\rho,d\de}$, so the supercentralizer $\Cent_{\rho,d}$ from (\ref{ECentralizer}) becomes 
$
\Cent_{\rho,d}=
\Cent_{H_{\rho,d\de}} (H_{\rho}),
$ 
and Lemma~\ref{LJuly4} 
becomes: 

\begin{Lemma} \label{LJuly4RoCK} 
We have an isomorphism of graded superalgebras
$$
H_{\rho}\otimes \Cent_{\rho,d}\iso H_{\rho,d\de}, \ h\otimes z\mapsto hz.
$$
\end{Lemma}

\subsection{RoCK blocks and imaginary cuspidal algebras}
Recalling (\ref{ECycPar}), (\ref{EIota}) and (\ref{EPi}), we consider the algebra homomorphism  
\begin{equation}\label{EOmega}
\Omega_{\rho,d}\colon R_{d\de}\to H_{\rho,d\de},\ x\mapsto \pi_\theta\big(\iota_{\rho,d\de}(1_{\rho}\otimes x)\big).
\end{equation}

Recall the cuspidal quotient $\hat C_d$ of $R_{d\de}$ from (\ref{ECuspidalAlgebra}).

\begin{Lemma}\label{LOmegaFact} 
The homomorphism $\Omega_{\rho,d}$ factors through the natural surjection $R_{d\de}\onto \hat C_d$. 
\end{Lemma}
\begin{proof}
In view of Lemma~\ref{LIdO}, this follows from the case where $\k$ is a field, proved in \cite[Lemma 4.1.16]{KlLi}.
\end{proof}

In view of the lemma,  we will consider $\Om_{\rho,d}$ as a homomorphism 
\begin{equation*}\label{EOm}
\Om_{\rho,d}: \hat C_d\to H_{\rho,d\de}.
\end{equation*}

\begin{Lemma} \label{LOmHatCOntoCent}
We have $\Cent_{\rho,d}=\Om_{\rho,d}(\hat C_d)$. 
\end{Lemma}
\begin{proof}
For $\k$ a field this is proved in \cite[Lemma 4.1.22]{KlLi}. In the general case the proof is the same, cf. \cite[Lemma 5.19]{EK2}.
\end{proof}

In view of the lemma, we now consider $\Om_{\rho,d}$ as a surjective homomoprhism $\Om_{\rho,d}: \hat C_d\to \Cent_{\rho,d}$. 
Given any element $c\in \hat C_d$, we denote 
\begin{equation}\label{EBar}
\bar c:=\Om_{\rho,d}(c)\in\Cent_{\rho,d}
\end{equation}
In particular, we have idempotents $\bar\ggi^{\la,\bi},\bar\ggi_d,\bar\ggi_{\om_d}$, etc. in $\Cent_{\rho,d}$.

\subsection{Gelfand-Graev truncation of $\Cent_{\rho,d}$}
Recall from \S\ref{SSGGW}, the Gelfand-Graev truncation $C_d=\ggi_{d}\hat C_d\ggi_{d}$. We now consider the corresponding truncation of the quotient $\Cent_{\rho,d}=\Om_{\rho,d}(\hat C_d)$.
Recalling the notation (\ref{EBar}), set 
\begin{equation}\label{EGGCent}
X_{\rho,d}:=\bar\ggi_d\Cent_{\rho,d}\bar \ggi_d.
\end{equation}
The surjection $\Om_{\rho,d}:\hat C_d\,\onto\, \Cent_{\rho,d}$ restricts to the surjection (of graded superalgebras)
\begin{equation}\label{EOmCX}
\Om_{\rho,d}:C_d\,\onto\, X_{\rho,d}.
\end{equation}

\begin{Lemma} \label{LXFree}
As a $\k$-module, $X_{\rho,d}$ is free of finite rank. 
Moreover, if $\F$ is a field of fractions of $\k$ or a quotient field of $\k$ then $X_{\rho,d,\F}\cong \F\otimes_\k X_{\rho,d}$. 
\end{Lemma}
\begin{proof}
As $X_{\rho,d}$ is an idempotent truncation of $\Cent_{\rho,d}$, this follows from Lemmas~\ref{LIdSymNew}(i) and \ref{LIdO}. 
\end{proof}

\begin{Lemma} \label{LProjGen}
If $\k$ is a field with $\cha\k=0$ or  $\cha\k>d$ then $X_{\rho,d}\bar \ggi_{\om_d}$ is a projective generator for $X_{\rho,d}$. 
\end{Lemma}
\begin{proof}
It is enough to show that for every  $L\in\Irr(X_{\rho,d})$,  we have 
$\bar \ggi_{\om_d}L\neq 0$. This comes from inflating $L$ to an irreducible $C_d$-module via $\Om_{\rho,d}$ and applying Theorem~\ref{TLargeChar}. 
\end{proof}

\begin{Corollary} \label{C8.11} 
Let $\k$ have the field of fractions of characteristic $0$. Then the $X_{\rho,d}$-module $X_{\rho,d}\bar \ggi_{\om_d}$ is faithful.
\end{Corollary}
\begin{proof}
Let $\K$ be the field of fractions of $\k$. By Lemma~\ref{LXFree}, the algebra  $X_{\rho,d}$ is $\k$-free, so it is enough to prove the result over $\K$, i.e. to show that the $X_{\rho,d,\K}$-module $X_{\rho,d,\K}\bar \ggi_{\om_d}$ is faithful, which follows from Lemma~\ref{LProjGen}. 
\end{proof}

Recall the notation (\ref{ELaBjj}). 

\begin{Theorem} \label{TDimRoCKTrunc} 
Let $(\la,\bi)\in \Comp^\col(n,d)$. Then the $\k$-module $\bar \ggi^{\la,\bi}X_{\rho,d}\bar f_{\om_d}$ is free of rank $$\binom{d}{\la_1\,\cdots\,\la_n}\,4^{d-|\la,\bi|_{\ell-1}}\,3^{|\la,\bi|_{\ell-1}}.$$ 
\end{Theorem}
\begin{proof}
By Lemma~\ref{LXFree}, we may assume that $\k$ is a field. Recalling the idempotent $1_{\bi_\rho}$ from Lemma~\ref{LMatrix}, we have the idempotents 
$\iota_{\rho,d\de}(1_{\bi_\rho}\otimes \ggi^{\la,\bi})$ and 
$\iota_{\rho,d\de}(1_{\bi_\rho}\otimes \ggi_{\om_d})$ in $H_\theta$. By Lemma~\ref{LIdSymNew}, we have 
$$
\dim \bar \ggi^{\la,\bi}X_{\rho,d}\bar f_{\om_d}=\dim \iota_{\rho,d\de}(1_{\bi_\rho}\otimes \ggi^{\la,\bi})H_\theta \iota_{\rho,d\de}(1_{\bi_\rho}\otimes \ggi_{\om_d}).
$$ 
This dimension equals the expression in the statement by \cite[Main Theorem]{KlDim}. 
\end{proof}

\subsection{Regrading of $X_{\rho,d}$}
Recall from \S\ref{SSRegradingCd} that we have a regrading $\zC_d$ of $C_d$. We now define the corresponding regrading of $X_{\rho,d}$. Recalling (\ref{EShiftsCuspidal})
and the general setting of \S\ref{SSRegr}, we obtain the new graded superalgebra 
\begin{equation}\label{EReGradingX}
\zX_{\rho,d}:=\bigoplus_{(\la,\bi),(\mu,\bj)\in\EC^\col(d)}
\funQ^{t_{\la,\bi}-t_{\mu,\bj}}\Uppi^{\eps_{\la,\bi}-\eps_{\mu,\bj}} \bar \ggi^{\mu,\bj}X_{\rho,d}\bar \ggi^{\la,\bi}.
\end{equation}

Comparing to (\ref{EReGradingC}), the homomorphism (\ref{EOmCX}) yields the surjective graded superalgebra homomorphism 
\begin{equation}\label{EOmCX'}
\Om_{\rho,d}:\zC_d\,\onto\, \zX_{\rho,d}.
\end{equation}
In accordance with (\ref{EBar}), for an element $\zc\in\zC_d$ we  denote $\Om_{\rho,d}(\zc)\in\zX_{\rho,d}$ by $\bar \zc$. In particular, we elements $ \bar \lgathz_{\la,\bi},\,\bar \ggis_{\om_d},\,\bar \ggis^\bj\in\zX_{\rho,d}$.

Upon restriction, $\Om_{\rho,d}$ yields a homomorhism $\ggis_{\om_d}\zC_d\ggis_{\om_d} \onto \bar \ggis_{\om_d}\zX_{\rho,d}\bar \ggis_{\om_d}$. Recalling the homomorphism $\Theta_d$ from (\ref{EThetad}), we have a graded superalgebra homomorphism
\begin{equation}
\label{EXid}
\Xi_{\rho,d}:=\Om_{\rho,d}\circ \Theta_d: W_d(\Zig_\ell)\to \bar \ggis_{\om_d}\zX_{\rho,d}\bar \ggis_{\om_d}.
\end{equation}
By (\ref{EThetaId}), for all $\bj\in J^d$ we have 
\begin{equation}\label{EXiId}
\Xi_{\rho,d}(\ze^\bj)=\bar \ggis^\bj. 
\end{equation}

\begin{Theorem} \label{TXiIso} 
Suppose that $H_{\rho+\de,\k/\m}$ is a symmetric superalgebra for any maximal ideal $\m$ of $\k$. Then 
the homomorphism $\Xi_{\rho,d}$ (over $\k$) is an isomorphism. 
\end{Theorem}
\begin{proof}
This follows from \cite[Theorem 4.4.31]{KlLi} by regrading. To give more details, let $A_\ell$ be a regrading of $\Zig_\ell$ given as follows:
\begin{equation}\label{EAEll}
A_\ell=\bigoplus_{i,j\in J}(\Uppi \funQ^{2})^{i-j}\ze^{[i]}\Zig_\ell \ze^{[j]},
\end{equation} 
see \cite[\S3.2a]{KIS}, and $H_d(A_\ell)$ be the affine zigzag superalgebra of $A_\ell$ with the natural embedding $\iota^{A_\ell}:W_d(A_\ell)\,\into\, H_d(A_\ell)$, see \cite{KlLi},\cite[\S3.2b]{KIS}. 
By \cite[Theorem 4.2.47]{KIS}, there is an explicit isomorphism $F_d:H_d(A_\ell)\iso 
\ggi_{\om_d}C_d\ggi_{\om_d}$.

Moreover, by \cite[Proposition 3.2.28]{KIS}, 
there is an explicit isomorphism of graded superalgebras 
$$
\phi:\bigoplus_{\bi,\bj\in J^n}\funQ^{t_{\om_d,\bi}-t_{\om_d,\bj}}\Uppi^{\eps_{\om_d,\bi}-\eps_{\om_d,\bj}}e^\bj H_d(A_\ell)e^\bi \iso H_d(\Zig_\ell).
$$
Upon restriction, $\phi$ yields the isomorphism $\bar \phi:W_d(A_\ell)\iso W_d(\Zig_\ell)$. Note that as (ungraded non-super) algebras, $\ggi_{\om_d}C_d\ggi_{\om_d}=\ggis_{\om_d}\zC_d\ggis_{\om_d}$,  $\bar \ggi_{\om_d}X_{\rho,d}\bar \ggi_{\om_d}=\bar \ggis_{\om_d}\zX_{\rho,d}\bar \ggis_{\om_d}$. 
So we have a diagram
\begin{equation*}\label{OHCD}
\begin{tikzcd}
W_d(A_\ell)\arrow[rr,"\iota^{A_\ell}" above, hook]
\arrow[d,"\bar\phi" left]
&&H_d(A_\ell)
\arrow[d,"\phi" left]
\arrow[rr,"F_d" above]
&& \ggi_{\om_d}C_d\ggi_{\om_d}
\arrow[rr,"\Om_{\rho,d}" above]
\arrow[d,equals]
&& \bar \ggi_{\om_d}X_{\rho,d}\bar \ggi_{\om_d}
\arrow[d,equals]
\\
W_d(\Zig_\ell)\arrow[rr,"\iota^{\Zig_\ell}" above,hook]
&&
H_d(\Zig_\ell)
\arrow[rr,"\zF_d" above]
&&
\ggis_{\om_d}\zC_d\ggis_{\om_d}
\arrow[rr,"\Om_{\rho,d}" above]
&& \bar \ggis_{\om_d}\zX_{\rho,d}\bar \ggis_{\om_d}
\end{tikzcd}
\end{equation*}
By construction in \cite{KIS}, we have $\zF_d\circ \phi=F_d$, so the diagram commutes. 

Now, if $H_{\rho+\de,\k/\m}$ is a symmetric superalgebra then it is a QF2 algebra, so we conclude from \cite[Theorem 4.4.31]{KlLi} that 
$\Om_{\rho,d}\circ F_d\circ \iota^{A_\ell}$ is an isomorphism over $\k/\m$ for any maximal ideal $\m$ of $\k$. We conclude that 
$\Om_{\rho,d}\circ F_d\circ \iota^{A_\ell}$ is an isomorphism over $\k$. So $\Xi_{\rho,d}=\Om_{\rho,d}\circ \zF_d\circ \iota^{\Zig_\ell}$ is an isomorphism as required. 
\end{proof}

\section{RoCK blocks and generalized Schur algebras}
\label{SRoCKSchur}

We now assume that $\k$ is a discrete valuation ring with the quotient field $\F:=\k/\m$ of odd characteristic $p$ and the field of fractions $\K$ of characteristic $0$. 
We exclude the case where $p$ divides $2\ell+1$ but $p<2\ell+1$ (thus, our main case of interest $p=2\ell+1$ is not excluded).  
We set $\bar\K$ to be the algebraic closure of $\K$. 
These assumptions allow us to apply Lemma~\ref{LSymmConditions} to deduce:

\begin{Lemma} \label{LHEtaSym} 
For any $\eta\in Q_+$, the $\F$-superalgebra $H_{\eta,\F}$ and the $\bar\K$-superalgebra $H_{\eta,\bar\K}$ are symmetric superalgebras. 
\end{Lemma}


Throughout the section, we assume that $\theta$ is RoCK, $\rho=\rho(\theta)$ and $d=d(\theta)$.

\subsection{Symmetricity}

Recall the homomorphism of graded superalgebras $\Xi_{\rho,d}:W_d(\Zig_\ell)\to \bar\ggis_{\om_d}\zX_{\rho,d}\bar\ggis_{\om_d}$ defined in (\ref{EXid}). In view of Lemma~\ref{LHEtaSym}, Theorem~\ref{TXiIso} yields:

\begin{Corollary} \label{Cor1} The homomorphism $\Xi_{\rho,d}$ (over $\k$) is an isomorphism. 
\end{Corollary}

From Lemmas~\ref{LHEtaSym} and \ref{LIdSymNew}, we also immediately get: 

\begin{Corollary} \label{Cor2} 
The $\F$-superalgebra $\Cent_{\rho,d,\F}\cong \F\otimes _\k\Cent_{\rho,d}$ and $\bar\K$-superalgebra $\Cent_{\rho,d,\bar\K}\cong \bar\K\otimes _\k\Cent_{\rho,d}$ are symmetric superalgebras. 
\end{Corollary}

Recall from (\ref{EGGCent}) that the (graded) superalgebra $X_{\rho,d}$ is an idempotent truncation of $\Cent_{\rho,d}$. 
Using Lemmas~\ref{LSymIdTr}, \ref{LIdSymNew}(i) and Corollary~\ref{Cor2}, we now also deduce

\begin{Corollary} \label{Cor3} 
The $\F$-superalgebra $X_{\rho,d,\F}\cong \F\otimes _\k X_{\rho,d}$ and the $\bar\K$-superalgebra $X_{\rho,d,\bar\K}\cong \bar\K\otimes _\k X_{\rho,d}$ are symmetric superalgebras. 
\end{Corollary}

The (graded) superalgebra $\zX_{\rho,d}$ is obtained from $X_{\rho,d}$ by regrading, see (\ref{EReGradingX}).  So, applying Corollary~\ref{Cor3} and Lemma~\ref{LSymRegr}, we get

\begin{Corollary} \label{Cor4} 
The $\F$-superalgebra $\zX_{\rho,d,\F}\cong \F\otimes _\k \zX_{\rho,d}$ and $\bar\K$-superalgebra $\zX_{\rho,d,\bar\K}\cong \bar\K\otimes _\k \zX_{\rho,d}$ are symmetric superalgebras. 
\end{Corollary}

\subsection{Identifying $\bar \ggis^{\la,\bi}\zX_{\rho,d}\bar \ggis_{\om_d}$ and $M_{\la,\bi}$} Let $(\la,\bi)\in \Comp^\col(n,d)$. Recall from (\ref{EThetaAction}) the right graded $W_d(\Zig_\ell)$-supermodule structure on $\ggis^{\la,\bi}\zC_{d}\ggis_{\om_d}$. Similarly, 
$\bar \ggis^{\la,\bi}\zX_{\rho,d}\bar \ggis_{\om_d}$ is naturally a right graded $\bar \ggis_{\om_d}\zX_{\rho,d}\bar \ggis_{\om_d}$-supermodule, and, using the isomorphism $\Xi_{\rho,d}$ of  Corollary~\ref{Cor1}, we have the structure of right graded $W_d(\Zig_\ell)$-supermodule on $\bar \ggis^{\la,\bi}\zX_{\rho,d}\bar \ggis_{\om_d}$ with
\begin{equation}\label{EXiAction}
\zv h=\zv\Xi_{\rho,d}(h)\qquad (\zv\in \bar \ggis^{\la,\bi}\zX_{\rho,d}\bar \ggis_{\om_d},\ h\in W_d(\Zig_\ell)). 
\end{equation}
By definition, the restriction of $\Om_{\rho,d}$ gives a bidegree $(0,\0)$ surjective homomorphism of graded right $W_d(\Zig_\ell)$-supermodules
$
\Om_{\la,\bi}:\ggis^{\la,\bi}\zC_{d}\ggis_{\om_d}\,\onto\,\bar \ggis^{\la,\bi}\zX_{\rho,d}\bar \ggis_{\om_d}.
$

\begin{Theorem} \label{TIdentifyM} 
For any $(\la,\bi)\in\Comp^\col(n,d)$ there exists a bidegree $(0,\0)$ isomorphism of right graded $W_d(\Zig_\ell)$-supermodules 
$$
\upeta_{\la,\bi}:M_{\la,\bi}\iso \bar \ggis^{\la,\bi}\zX_{\rho,d}\bar \ggis_{\om_d},\ m_{\la,\bi}\mapsto \bar \lgathz_{\la,\bi}. 
$$
\end{Theorem}
\begin{proof}
Let the homomorphism $\upzeta_{\la,\bi}:M_{\la,\bi}\to \ggis^{\la,\bi}\zC_{d}\ggis_{\om_d}$ be as in Lemma~\ref{LUptheta}. Then 
$$\upeta_{\la,\bi}:=\Om_{\la,\bi}\circ\upzeta_{\la,\bi}:M_{\la,\bi}\to \bar \ggis^{\la,\bi}\zX_{\rho,d}\bar \ggis_{\om_d}$$ 
is a bidegree $(0,\0)$ homomorphism of right graded $W_d(\Zig_\ell)$-supermodules which maps $m_{\la,\bi}$ onto 
$\bar \lgathz_{\la,\bi}$. By Theorem~\ref{TGath}(iv), $\ggis^{\la,\bi}\zC_{d}\ggis_{\om_d}$ is generated by $\lgathz_{\la,\bi}$ as a right $\ggis_{\om_d}\zC_{d}\ggis_{\om_d}$-module. Hence $\bar \ggis^{\la,\bi}\zX_{\rho,d}\bar \ggis_{\om_d}=\Om_{\la,\bi}(\ggis^{\la,\bi}\zC_{d}\ggis_{\om_d})$ is generated by $\bar \lgathz_{\la,\bi}$ as a right $\bar \ggis_{\om_d}\zX_{\rho,d}\bar \ggis_{\om_d}$-module. Since $\Xi_{\rho,d}$ is an isomorphism by Corollary~\ref{Cor1}, we deduce that $\bar \ggis^{\la,\bi}\zX_{\rho,d}\bar \ggis_{\om_d}$ is generated by $\bar \lgathz_{\la,\bi}$ as a right $W_d(\Zig_\ell)$-module. Thus $\upeta_{\la,\bi}$ is surjective. But by Lemma~\ref{LDimM} and Theorem~\ref{TDimRoCKTrunc}, the $\k$-modules $M_{\la,\bi}$ and $\bar \ggis^{\la,\bi}\zX_{\rho,d}\bar \ggis_{\om_d}$ are free of the same finite rank. So $\upeta_{\la,\bi}$ is an isomorphism. 
\end{proof}

\subsection{The endomorphism algebra $\zE_{\rho,d,n}$}
Let $n\in \N_+$. 
Recall from \S\ref{SSPar} that we consider multicompositions in $\Comp^J(n,d)$ as colored compositions in $\Comp^\col(n\ell,d)$ via the injection $\upgamma^J_{n,d}$ from (\ref{EBeta}). In particular, for any $\bla\in\Comp^J(n,d)$,  we have defined in Remark~\ref{RBLa} the module $M_\bla=M_{\upgamma^J_{n,d}(\bla)}$ with generator $m_\bla$ such that $m_\bla\ze^\bla=m_\bla$ and $m_\bla w=m_\bla$ for all $w\in\Si_\bla$, see (\ref{ESiBla}). 
Similarly, for any $\bla\in\Comp^J(n,d)$, we have the Gelfand-Graev word
\begin{equation}\label{EGGBla}
\ggw^\bla=\ggw^{\upgamma^J_{n,d}(\bla)}
\end{equation}
the 
idempotents $\ggis^\bla=\ggis^{\upgamma^J_{n,d}(\bla)}$  
and $\ggis(\bla)=\ggis(\upgamma^J_{n,d}(\bla))$  
in $\zC_d$, the element $\bar \lgathz_{\bla}=\bar \lgathz_{\upgamma^J_{n,d}(\bla)}\in \bar  \ggis^\bla\zX_{\rho,d}\bar \ggis_{\om_d}$, etc. In particular, from Theorem~\ref{TIdentifyM}, we have the  isomorphism of right graded $W_d(\Zig_\ell)$-supermodules 
\begin{equation}\label{EEtaBla}
\upeta_{\bla}:M_{\bla}\iso\bar  \ggis^{\bla}\zX_{\rho,d}\bar \ggis_{\om_d},\ m_{\bla}\mapsto \bar \lgathz_{\bla}.
\end{equation}

Define the graded left $\zX_{\rho,d}$-supermodule 
$$
\Ga_{\rho,d,n}:=\bigoplus_{\bla\in\Comp^J(n,d)}\zX_{\rho,d}\bar \ggis^\bla
$$
and the graded superalgebra
$$
\zE_{\rho,d,n}:=\End_{\zX_{\rho,d}}(\Ga_{\rho,d,n})^\sop.
$$
Via the isomorphism (\ref{EEndAei}), we identify 
\begin{equation}\label{EIdentifyE}
\zE_{\rho,d,n}=\bigoplus_{\bla,\bmu\in\Comp^J(n,d)}\bar \ggis^\bmu\zX_{\rho,d}\bar \ggis^\bla.
\end{equation}

\begin{Lemma} \label{LEXMor} 
If $n\geq d$ then the graded superalgebras $\zE_{\rho,d,n}$ and $\zX_{\rho,d}$ are graded Morita superequivalent.  
\end{Lemma}
\begin{proof}
In view of (\ref{EGad}), using the assumption $n\geq d$, we see that $\sum_{\bla\in D}\ggis^\bla=\bar \ggis_d=1_{\zX_{\rho,d}}$ for some subset $D\subseteq \Comp^J(n,d)$, so $\Ga_{\rho,d,n}$ is a projective generator for $\zX_{\rho,d}$. It remains to apply  Example~\ref{ExProgenerator}. 
\end{proof}

\begin{Lemma} \label{L8.22} 
The $\F$-algebra $\F\otimes_\k \zE_{\rho,d,n}$ and the $\bar\K$-algebra $\bar\K\otimes_\k \zE_{\rho,d,n}$ are symmetric superalgebras.
\end{Lemma}
\begin{proof}
We have 
$$
\F\otimes_\k \zE_{\rho,d,n}=\F\otimes_\k\End_{\zX_{\rho,d}}(\Ga_{\rho,d,n})^\sop\cong \End_{\F\otimes_\k\zX_{\rho,d}}(\F\otimes_\k\Ga_{\rho,d,n})^\sop.
$$
Now the claim for $\F$ follows from Corollary~\ref{Cor4} and Lemmas~\ref{LSopSymm},\,\ref{LSymIdTr}. The argument for $\bar\K$ is similar.
\end{proof}

\subsection{The isomorphism $\Upphi_{\rho,d,n}$}
Let $\zx\in \bar \ggis^\bmu\zX_{\rho,d}\bar \ggis^\bla$. In view of 
 the identification (\ref{EIdentifyE}), we consider $\zx$ as an element of $\zE_{\rho,d,n}$. Recalling from (\ref{EXiAction}) the right graded $W_d(\Zig_\ell)$-supermodule structure on $\bar \ggis^\bla\zX_{\rho,d}\bar \ggis_{\om_d}$, we have a  homomorphism of right graded $W_d(\Zig_\ell)$-supermodules
$$
\bar \ggis^\bla\zX_{\rho,d}\bar \ggis_{\om_d}\to \bar \ggis^\bmu\zX_{\rho,d}\bar \ggis_{\om_d},\ \zv\mapsto \zx\zv.
$$
Identifying $\bar \ggis^\bla\zX_{\rho,d}\bar \ggis_{\om_d}$ with $M_\bla$ and $\bar \ggis^\bmu\zX_{\rho,d}\bar \ggis_{\om_d}$ with $M_\bmu$ via the isomorphisms (\ref{EEtaBla}), we obtain an element 
$
\Upphi_{\rho,d,n}(\zx)\in\Hom_{W_d(\Zig_\ell)}(M_\bla,M_\bmu).
$
In other words, for $v\in M_\bla$, we have 
\begin{equation}\label{EPhiRhoDN}
\Upphi_{\rho,d,n}(\zx):M_\bla\to M_\bmu,\ v\mapsto \upeta_{\bmu}^{-1}(\zx\upeta_{\bla}(v)). 
\end{equation}
Recall from (\ref{EIdentifyXilaSXimu}) that $\Hom_{W_d(\Zig_\ell)}(M_\bla,M_\bmu)$ is identified with $\xi_\bmu S^{\Zig_\ell}(n,d) \xi_\bla$. So the assignments $\zx\mapsto \Upphi_{\rho,d,n}(\zx)$ for all 
$\bla,\bmu\in \Comp^J(n,d)$ and all $\zx\in \bar \ggis^\bmu\zX_{\rho,d}\bar \ggis^\bla$ extend uniquely to a (bidegree $(0,\0)$)  homomorphism of graded $\k$-superalgebras  
\begin{equation}\label{EPhi}
\Upphi_{\rho,d,n}:\zE_{\rho,d,n}\to S^{\Zig_\ell}(n,d)
\end{equation}
with 
\begin{equation}\label{EPhiId}
\Upphi_{\rho,d,n}(\bar\ggis^{\bla})= \xi_\bla\qquad(\bla\in\Comp^J(n,d)).
\end{equation}

\begin{Lemma} \label{C8.19}
The homomorphism $\Upphi_{\rho,d,n}$ is injective.
\end{Lemma}
\begin{proof}
Otherwise there exist $\bla,\bmu\in \Comp^J(n,d)$ and a non-zero $\zx\in \bar \ggis^\bmu\zX_{\rho,d}\bar \ggis^\bla$ such that $\zx\bar \ggis^\bla\zX_{\rho,d}\bar \ggis_{\om_d}=0$. Equivalently, $\zx\zX_{\rho,d}\bar \ggis_{\om_d}=0$. But this contradicts Corollary~\ref{C8.11}. 
\end{proof}

\begin{Lemma} \label{L8.20}
The image $\Upphi_{\rho,d,n}(\zE_{\rho,d,n})$ contains the degree zero component $S^{\Zig_\ell}(n,d)^0$. 
\end{Lemma}
\begin{proof}
Let $\bla,\bmu\in \Comp^J(n,d)$ and take an arbitrary  degree $0$ element 
$$\xi\in \xi_\bmu S^{\Zig_\ell}(n,d)^0\xi_\bla=\Hom_{W_d(\Zig_\ell)}(M_\bla,M_\bmu)^0.
$$ 
Then we can write $\xi(m_\bla)=m_\bmu h$ for some degree $0$ element $h\in W_d(\Zig_\ell)$. Since $\xi(m_\bla)\ze^\bla=\xi(m_\bla \ze^\bla)=\xi(m_\bla)$ and $\xi(m_\bla)w=\xi(m_\bla w)=\xi(m_\bla)$ for all $w\in\Si_\bla$, we have $m_\bmu h\ze^\bla=m_\bmu h$ and $m_\bmu h w=m_\bmu h$ for all $w\in\Si_\bla$. 
From (\ref{EThetaAction}) and Lemma~\ref{LUptheta}, the degree $0$ element 
$$
\zv:=\upzeta_\bmu(m_\bmu h)= \lgathz_\bmu\Theta_d(h)\in \ggis^{\la,\bi}\zC_{d}\ggis_{\om_d}
$$
satisfies $\zv=\zv\Theta_d(\ze^\bla)$ and $\zv=\zv\Theta_d(w)$ for all $w\in\Si_\bla$. Note that $\Theta_d(\ze^\bla)=\ggis(\bla)$. So we can apply Theorem~\ref{TGath2} to deduce the existstence of a (degree $0$) element $\zx\in \ggis^\bmu\zC_d\ggis^\bla$ such that $\zv=\zx\lgathz_\bla$. 
Thus $\lgathz_\bmu\Theta_d(h)=\zx\lgathz_\bla$. 
Applying $\Om_{\rho,d}$, we get 
$
\bar\lgathz_\bmu\Xi_d(h)=\bar \zx\bar\lgathz_\bla.
$
Therefore 
$$
\Upphi_{\rho,d,n}(\bar\zx)(m_\bla)=\upeta_\bmu^{-1}(\bar\zx\upeta_\bla(m_\bla))=\upeta_\bmu^{-1}(\bar\zx \bar\lgathz_\bla)=\upeta_\bmu^{-1}(\bar\lgathz_\bmu\Xi_d(h))
=\upeta_\bmu^{-1}(\bar\lgathz_\bmu)h=m_\bmu h.
$$
So $\Upphi_{\rho,d,n}(\bar\zx)(m_\bla)=\xi(m_\bla)$, hence 
$\Upphi_{\rho,d,n}(\bar\zx)=\xi$, i.e. $\xi\in \Upphi_{\rho,d,n}(\zE_{\rho,d,n})$. 
\end{proof}

Recall the map $\tti^\bla$ from (\ref{Etti}). 

\begin{Lemma} \label{L8.21} 
For any $\bla\in\Comp^J(n-1,d-1)$ and $\zx\in \Zig_\ell$, we have $\tti^\bla(\zx)\in\Upphi_{\rho,d,n}(\zE_{\rho,d,n})$. 
\end{Lemma}
\begin{proof}
We may assume that $\zx\in\ze^{[j]}\Zig_\ell\ze^{[k]}$. Recalling the notation (\ref{EInsertion}), we have the element $\zx_1=\zx\otimes 1_{\Zig_\ell}^{\otimes (d-1)}\in \Zig_\ell^{\otimes d}\subseteq W_d(\Zig_\ell)$. 
Recall the parabolic subalgebra $\zC_{\om_d}=\zC_1^{\otimes d}\subseteq \ggis_{\om_d}\zC_d\ggis_{\om_d}$ from (\ref{ECParz}). We have the embedding $\zC_1\to  \zC_{\om_d},\ \zc\mapsto \zc_1:=\zc\otimes \zf_1^{\otimes (d-1)}.
$ 

Recall the homomorphism 
$\Theta_d=\zF_d\circ \iota^{\Zig_\ell}: W_d(\Zig_\ell)\to \ggis_{\om_d}\zC_{d}\ggis_{\om_d}$ from (\ref{EThetad}).
From the explicit description of $\zF_d$ in Theorem~\ref{THCIsoZ}, it follows that $\Theta_d(\zx_1)=\zc_1$ for some $\zc\in\ggis^j\zC_1\ggis^k$. 

Using the notation $\bla_{\{i\}}$ from \S\ref{SSSpecialElements}, note using 
(\ref{EGath}) for the second equality that for any $i\in J$ we have  
\begin{align*}
&\ggis^{\bla_{\{i\}}}=\ggis^i\otimes \ggis^\bla\in \zC_1\otimes \zC_{d-1}=\zC_{(1,d-1)}\subseteq \zC_d,
\\
&\lgathz_{\bla_{\{i\}}}=\ggis^i\otimes \lgathz_\bla\in \zC_1\otimes \zC_{d-1}=\zC_{(1,d-1)}\subseteq \zC_d.
\end{align*}
It follows that $\zc_1\lgathz_{\bla_{\{k\}}}=\lgathz_{\bla_{\{j\}}}\zc_1$, and that 
$
\zb:=\ggis^{\bla_{\{j\}}}\zc_1=\zc_1\ggis^{\bla_{\{k\}}}$ belongs to $
\ggis^{\bla_{\{j\}}}\zC_d\ggis^{\bla_{\{k\}}}.
$ 
Moreover, since $\ggis^{\bla_{\{k\}}}\lgathz_{\bla_{\{k\}}}=\lgathz_{\bla_{\{k\}}}$ by Theorem~\ref{TGath}(i), we have 
$$
\zb\lgathz_{\bla_{\{k\}}}=
\zc_1\ggis^{\bla_{\{k\}}}\lgathz_{\bla_{\{k\}}}=\zc_1\lgathz_{\bla_{\{k\}}}=\lgathz_{\bla_{\{j\}}}\zc_1=
\lgathz_{\bla_{\{j\}}}\Theta_d(\zx_1).
$$
Applying the homomorphism $\Om_{\rho,d}$ and recalling (\ref{EXid}), we get 
$$
\bar\zb\bar\lgathz_{\bla_{\{k\}}}=
\lgathz_{\bla_{\{j\}}}\Om_{\rho,d}(\Theta_d(\zx_1))=
\bar\lgathz_{\bla_{\{j\}}}\Xi_{\rho,d}(\zx_1).
$$
Now, recalling (\ref{EPhiRhoDN}), we have 
\begin{align*}
\Upphi_{\rho,d,n}(\bar\zb)(m_{\bla_{\{k\}}})&=\upeta_{\bla_{\{j\}}}^{-1}(\bar\zb\upeta_{\bla_{\{k\}}}(m_{\bla_{\{k\}}}))
=\upeta_{\bla_{\{j\}}}^{-1}(\bar\zb\bar\lgathz_{\bla_{\{k\}}})
=\upeta_{\bla_{\{j\}}}^{-1}(\bar\lgathz_{\bla_{\{j\}}}\Xi_{\rho,d}(\zx_1))
\\
&=m_{\bla_{\{j\}}}\zx_1
=\tti^\bla(m_{\bla_{\{j\}}}). 
\end{align*}
So $\Upphi_{\rho,d,n}(\bar\zb)=\tti^\bla$. 
\end{proof}

\begin{Theorem} \label{TEIsoT} 
Let $n\geq d$. Then the homomorphism $\Upphi_{\rho,d,n}:\zE_{\rho,d,n}\to S^{\Zig_\ell}(n,d)$ has image $T^{\Zig_\ell}(n,d)$ and defines an isomorphism of graded $\k$-superalgebras $\zE_{\rho,d,n}\cong T^{\Zig_\ell}(n,d)$. 
\end{Theorem}
\begin{proof}
By Lemma~\ref{C8.19}, $\Upphi_{\rho,d,n}$ is an injective graded superalgebra homomorphism. By Corollary~\ref{CGen2} and Lemmas~\ref{L8.20},\,\ref{L8.21}, we have that the subalgebra  $T^{\Zig_\ell}(n,d)$ is contained in $\Upphi_{\rho,d,n}(\zE_{\rho,d,n})$. By Lemma~\ref{L8.22}, the $\F$-superalgebra $\F\otimes_\k\Upphi_{\rho,d,n}(\zE_{\rho,d,n})\cong 
\F\otimes_\k\zE_{\rho,d,n}$ is a symmetric superalgebra. Since $\F$ is the only field of the form $\k/\m$ for a maximal ideal $\m\subset \k$, Theorem~\ref{TMaxSymm} implies $\Upphi_{\rho,d,n}(\zE_{\rho,d,n})=T^{\Zig_\ell}(n,d)$. 
\end{proof}

\subsection{Morita equivalences}
Recall that throughout Section~\ref{SRoCKSchur} we are assuming that $H_\theta$ is a RoCK block with $\rho=\rho(\theta)$ and $d=d(\theta)$.

\begin{Theorem} \label{TMainText} 
If $n\geq d$ then the graded $\k$-superalgebras $H_\theta$ and $T^{\Zig_\ell}(n,d)$ are graded Morita superequivalent. 
\end{Theorem}
\begin{proof}
By Lemma~\ref{LIdSymNew}, there exists a bidegree $(0,\0)$ idempotent $\eps\in H_\theta$ such that $\Cent_{\rho,d}\cong\eps H_\theta\eps$. We have further idempotent truncation 
$X_{\rho,d}=\bar\ggi_d\Cent_{\rho,d}\bar \ggi_d$, see (\ref{EGGCent}). Thus $X_{\rho,d}=eH_\theta e$ for the bidegree $(0,\0)$ idempotent $e=\eps \bar\ggi_d=\bar\ggi_d\eps$. 
The regrading $\zX_{\rho,d}$ of $X_{\rho,d}$ defined in (\ref{EReGradingX}) is graded Morita superequivalent to $X_{\rho,d}$, see \S\ref{SSRegr}. Moreover, $\zX_{\rho,d}$ is graded Morita superequivalent to $\zE_{\rho,d,n}$, see Lemma~\ref{LEXMor}. Finally, by Theorem~\ref{TEIsoT}, we have the isomorphism of graded superalgebras $\zE_{\rho,d,n}\cong T^{\Zig_\ell}(n,d)$. So it suffices to prove that $eH_\theta e$ is graded Morita superequivalent to $H_\theta$, or by Lemma~\ref{LMorExtScal} that $e_\F H_{\theta,\F} e_\F$ is graded Morita superequivalent to $H_{\theta,\F}$. 

It is well-known, see for example \cite[Corollary 2.2.15]{KlLi}, that $e_\F H_{\theta,\F} e_\F$ is graded Morita superequivalent to $H_{\theta,\F}$ if $|\Irr(e_\F H_{\theta,\F} e_\F)|= |\Irr(H_{\theta,\F})|<\infty$. By the first paragraph, the algebras $e_\F H_{\theta,\F} e_\F$ and $T^{\Zig_\ell}(n,d)_\F$ are graded Morita superequivalent, so it suffices to prove that 
$|\Irr(H_{\theta,\F})|=|\Irr(T^{\Zig_\ell}(n,d)_\F)|$. But this follows from Lemmas~\ref{LIrrHTheta} and \ref{LIrrT}. 
\end{proof}

We spell out explicitly the graded superequivalence functor 
$\mod{H_\theta}\to\mod{T^{\Zig_\ell}(n,d)}$ coming from Theorem~\ref{TMainText}. Following the proof of the theorem, we set  
$\eps:=\iota_{\rho,d\de}(1_{\bi_\rho}\otimes 1_{d\de})$
as in Lemma~\ref{LIdSymNew}. Then, recalling (\ref{EGad}) and (\ref{EGGIdempotent}), we have the idempotent
$$
e=\eps \bar\ggi_d=\eps\sum_{(\la,\bi)\in\EC^\col(d)}\bar\ggi^{\la,\bi}=\sum_{(\la,\bi)\in\EC^\col(d)}1_{\bi_\rho\ggw^{\la,\bi}}\in H_\theta 
$$
and the functor
$$
\mod{H_\theta}\to\mod{eH_\theta e},\ V\mapsto eV=\bigoplus_{(\la,\bi)\in\EC^\col(d)}1_{\bi_\rho\ggw^{\la,\bi}}V.
$$
Here $1_{\bi_\rho\ggw^{\bla}}$ is the divided power idempotent corresponding to the concatenation of the divided power words $\bi_\rho$ from (\ref{EIRho}) and $\ggw^{\bla}$ from (\ref{EGGBla}). 
By Lemma~\ref{LIdSymNew} and (\ref{EGGCent}), 
$$X_{\rho,d}\cong eH_\theta e\cong\bigoplus_{(\la,\bi),(\mu,\bj)\in\EC^\col(d)}1_{\bi_\rho\ggw^{\mu,\bj}}H_\theta1_{\bi_\rho\ggw^{\la,\bi}}.
$$
Combining with passing from $X_{\rho,d}$ to its regrading 
$$\zX_{\rho,d}\cong \bigoplus_{(\la,\bi),(\mu,\bj)\in\EC^\col(d)}
\funQ^{t_{\la,\bi}-t_{\mu,\bj}}\Uppi^{\eps_{\la,\bi}-\eps_{\mu,\bj}} 
1_{\bi_\rho\ggw^{\mu,\bj}}H_\theta1_{\bi_\rho\ggw^{\la,\bi}}
$$ 
defined in (\ref{EReGradingX}) yields the functor 
$$
\mod{H_\theta}\to\mod{\zX_{\rho,d}},\ V\mapsto \bigoplus_{(\la,\bi)\in\EC^\col(d)}
\funQ^{-t_{\la,\bi}}\Uppi^{-\eps_{\la,\bi}} 
1_{\bi_\rho\ggw^{\la,\bi}}V.
$$
Recalling (\ref{EIdentifyE}), the Gelfand-Graev words  
$\ggw^{\bla}$ from (\ref{EGGBla}) and the isomorphism $\Upphi_{\rho,d,n}$ from Theorem~\ref{TEIsoT}, for $n\geq d$, we have the algebra 
\begin{align*}
T^{\Zig_\ell}(n,d)\cong \zE_{\rho,d,n}&=\bigoplus_{\bla,\bmu\in\Comp^J(n,d)}\bar \ggis^\bmu\zX_{\rho,d}\bar \ggis^\bla
\\
&\cong\bigoplus_{\bla,\bmu\in\Comp^J(n,d)}\funQ^{t_{\bla}-t_{\bmu}}\Uppi^{\eps_{\bla}-\eps_{\bmu}} 
1_{\bi_\rho\ggw^{\bmu}}H_\theta1_{\bi_\rho\ggw^{\bla}},
\end{align*} 
which now gives the functor 
\begin{equation}\label{EFRhoND}
\funF_{\rho,d,n}:\mod{H_\theta}\to\mod{T^{\Zig_\ell}(n,d)},\ V\mapsto \bigoplus_{\bla\in\Comp^J(n,d)}\funQ^{-t_{\bla}}\Uppi^{-\eps_{\bla}}1_{\bi_\rho\ggw^{\bla}} V.
\end{equation}
In view of (\ref{EPhiId}), the direct sum decomposition in the right hand side above is precisely the decomposition of $\funF_{\rho,d,n}(V)$ into the weight spaces $\xi_\bla\funF_{\rho,d,n}(V)$:
\begin{equation}\label{E1Xi}
\funQ^{-t_{\bla}}\Uppi^{-\eps_{\bla}}1_{\bi_\rho\ggw^{\bla}} V
=\xi_\bla\funF_{\rho,d,n}(V)\qquad(\bla\in\Comp^J(n,d)). 
\end{equation}

The proof of Theorem~\ref{TMainText} shows precisely that the functor $\funF_{\rho,d,n}$ is a graded superequivalence. Let ${\mathbb L}=\F$ or $\K$. 
Extending the scalars from $\k$ to ${\mathbb L}$, we get the graded superequivalence 
$$
\funF_{\rho,d,n,{\mathbb L}}:\mod{H_{\theta,{\mathbb L}}}\to\mod{T^{\Zig_\ell}(n,d)_{\mathbb L}},\ V\mapsto \bigoplus_{\bla\in\Comp^J(n,d)}\funQ^{-t_{\bla}}\Uppi^{-\eps_{\bla}}1_{\bi_\rho\ggw^{\bla}} V.
$$
In particular:

\begin{Corollary} \label{CMainField} 
Let ${\mathbb L}=\F$ or $\K$. 
If $n\geq d$ then the graded ${\mathbb L}$-superalgebras $H_{\theta,{\mathbb L}}$ and $T^{\Zig_\ell}(n,d)_{\mathbb L}$ are graded Morita superequivalent. 
\end{Corollary}
\begin{proof}
Follows from Theorem~\ref{TMainText} upon extending scalars from $\k$ to $\F$.
\end{proof}

Recall the set   
$
\Irr(T^{\Zig_\ell}(n,d)_{\mathbb L})=\{L^{\Zig_\ell}_{n,d}(\bla)_{\mathbb L}\mid\bla\in \Comp^J_+(n,d)\}
$
from (\ref{EIrrT}). If $n\geq d$, we have $\Comp^J_+(n,d)=\Par^J(d)$. Recall the classical Schur algebra weight multiplicities $K_{\bla,\bmu,{\mathbb L}}$ from (\ref{EKBlaBmu}). 

\begin{Corollary} \label{CIrrHTheta}
Let $n\geq d$, and for $\bla\in \Par^J(d)$ denote $L_{\rho,d}(\bla)_{\mathbb L}:=\funF_{\rho,d,n,{\mathbb L}}(L^{\Zig_\ell}_{n,d}(\bla)_{\mathbb L})$. Then:
\begin{enumerate}
\item[{\rm (i)}] We have $\Irr(H_{\theta,{\mathbb L}})=\{L_{\rho,d}(\bla)_{\mathbb L}\mid\bla\in \Par^J(d)\}$.
\item[{\rm (ii)}] For any $\bla\in\Par^J(d)$ and $\bmu\in\Comp^J(n,d)$, we have an isomorphism of graded superspaces $1_{\bi_\rho\ggw^{\bmu}}L_{\rho,d}(\bla)_{\mathbb L}\simeq \xi_\bmu L^{\Zig_\ell}_{n,d}(\bla)_{\mathbb L}$. In particular,
$$\dim 1_{\bi_\rho\ggw^{\bmu}}L_{\rho,d}(\bla)_{\mathbb L}=
K_{\bla,\bmu,{\mathbb L}}.$$
Moreover, for any $\bi\in I^\rho_\di$, we have
$$\dim 1_{\bi\ggw^{\bmu}}L_{\rho,d}(\bla)_{\mathbb L}=K_{\bla,\bmu,{\mathbb L}}\cdot\dim 1_\bi H_{\rho,{\mathbb L}} 1_{\bi_\rho}.$$

\item[{\rm (iii)}] The graded $H_{\theta,{\mathbb L}}$-supermodule $L_{\rho,d}(\bla)_{\mathbb L}$ does not depend on the choice of $n\geq d$.
\end{enumerate}
\end{Corollary}
\begin{proof}
(i) follows from the fact that $\funF_{\rho,d,n,{\mathbb L}}$ is an equivalence, (ii) follows from (\ref{E1Xi}), (\ref{EWtMulBMuBLa}), and Lemma~\ref{LBiBiRho}, (iii) follows from (ii) since the formal characters of the modules $L^{\Zig_\ell}_{n,d}(\bla)_{\mathbb L}$ are linearly independent. 
\end{proof}

From Theorem~\ref{TMainText} and Lemma~\ref{LTTrunc}, we immediately deduce. 

\begin{Corollary} \label{PWreathMorita} 
The graded $\K$-superalgebras $H_{\theta,\K}$ and $W_{d,\K}^{\Zig_\ell}$ are graded Morita superequivalent. If $p>d$ then 
the $H_{\theta,\F}$ and $W_{d,\F}^{\Zig_\ell}$ are also graded Morita superequivalent. 
\end{Corollary}

\subsection{Theorem A and variations} We now prove Theorem A from the introduction. 
\label{SSTA}

\begin{Theorem} \label{TAMainBody}
Let $\cha \F=2\ell+1$ and $\Blo^{\rho,d}$ be a RoCK spin block of $\F\tilde\Si_n$. 
\begin{enumerate}
\item[{\rm (i)}] If $|\rho|-h(\rho)+d$ is even, then $\Blo^{\rho,d} \sim_{\sM} T^{\Zig_\ell}(d,d)_\F$ and $\Blo^{\rho,d}_\0 \sim_{\Mor}T^{\Zig_\ell}(d,d)_{\F,\0}$. 
\item[{\rm (ii)}] If $|\rho|-h(\rho)+d$ is odd, then $\Blo^{\rho,d} \sim_{\sM} T^{\Zig_\ell}(d,d)_\F\otimes \cC_1$ and $\Blo^{\rho,d}_\0 \sim_{\Mor}T^{\Zig_\ell}(d,d)_\F$. 
\end{enumerate}
\end{Theorem}
\begin{proof}
Recalling (\ref{EBloRel}), let $\theta=\cont(\rho)+d\de$, so $\Blo^{\rho,d}=\cT_{\theta}$ and  $\cont(\rho)=\rho(\theta)$, $d=d(\theta)$. Write $\cont(\rho)=\sum_{i\in I}r_i\al_i$. 
By Theorem~\ref{T5.3.31}, we have 
$\Blo^{\rho,d}\sim_{\sM} H_{\theta,\F}\otimes \cC_{n-r_0}$. 

Now, using \cite[(22.15)]{Kbook} for the second congruence, we see that 
$$n-r_0=|\rho|+dp-r_0\equiv 
|\rho|+d-r_0 \pmod{2}
\equiv
|\rho|-h(\rho)+d\pmod{2}.$$

By (\ref{ECliffEq}), we have 
$$
\Blo^{\rho,d}\sim_{\sM}
\left\{
\begin{array}{ll}
H_{\theta,\F} &\hbox{if $|\rho|-h(\rho)+d$ is even,}\\
H_{\theta,\F}\otimes \cC_1 &\hbox{if $|\rho|-h(\rho)+d$ is odd.}
\end{array}
\right.
$$
From Corollary~\ref{CMainField}, we now deduce 
$$
\Blo^{\rho,d}\sim_{\sM}
\left\{
\begin{array}{ll}
T^{\Zig_\ell}(d,d)_\F &\hbox{if $|\rho|-h(\rho)+d$ is even,}\\
T^{\Zig_\ell}(d,d)_\F\otimes \cC_1 &\hbox{if $|\rho|-h(\rho)+d$ is odd.}
\end{array}
\right.
$$
The statement for $\Blo^{\rho,d}_\0$ now follows from \cite[Lemmas 2.2.17,\,2.2.20]{KlLi}.
\end{proof}

We have the following $q$-analogue of Theorem~\ref{TAMainBody} for Olshanski's algebra $\cY_\theta(q)$:

\begin{Theorem} \label{TA'}
Let $\mathbb L$ be an algebraically closed field of characteristic different from $2$, the parameter $q\in{\mathbb L}^\times$ be a primitive $(2\ell+1)$st root of unity, and $\theta=\sum_{i\in I}m_i\al_i\in Q_+$ be RoCK. Then
$$
\cY_\theta(q)\sim_{\sM}
\left\{
\begin{array}{ll}
T^{\Zig_\ell}(d,d)_{\mathbb L} &\hbox{if $m_0$ is even,}\\
T^{\Zig_\ell}(d,d)_{\mathbb L}\otimes \cC_1 &\hbox{if $m_0$ is odd.}
\end{array}
\right.
$$
\end{Theorem}
\begin{proof}
This follows from Theorem~\ref{TKKT2} and Corollary~\ref{CMainField}. 
\end{proof}

\subsection{Irreducible $\cT_n$-supermodules in RoCK blocks.}
\label{SSGGFragm}
Throughout this subsection we assume that $\cha \F=2\ell+1$. 
Moreover, recall that throughout Section~\ref{SRoCKSchur} we are assuming that $H_\theta$ is a RoCK block with $\rho=\rho(\theta)$ and $d=d(\theta)$. In particular, $\theta=\cont(\rho)+d\de$ and $\Blo^{\rho,d}=\cT_{\theta}$. 

We have the idempotents $\{e(\bi)\mid \bi\in I^\theta\}$ in $\cT_\theta$ defined in \S\ref{SSKKT}. Recall the bijection $\beta:\Irr(H_\theta)\iso \Irr(\cT_\theta)$ from Theorem~\ref{TWtHWtT}
and the set $\Irr(H_{\theta,\F})=\{L_{\rho,d}(\bla)_\F\mid\bla\in \Par^J(d)\}$ from Corollary~\ref{CIrrHTheta}.
From (\ref{EIrrT}) we have the set   
$
\Irr(T^{\Zig_\ell}(d,d)_\F)=\{L^{\Zig_\ell}_{n,d}(\bla)_\F\mid\bla\in \Comp^J_+(d,d)\}
$. Note that $\Comp^J_+(d,d)=\Par^J(d)$.

If 
$|\rho|-h(\rho)+d$ is even then by Theorem~\ref{TAMainBody} we have $\Blo^{\rho,d} \sim_{\sM} T^{\Zig_\ell}(d,d)_\F$, and so to every $\bla\in\Par^J(d)$, we have the irreducible $\Blo^{\rho,d}$-supermodule $D(\bla)$ of type $\Mtype$ corresponding to $L^{\Zig_\ell}_{n,d}(\bla)_\F$ under the Morita superequivalence. 
If 
$|\rho|-h(\rho)+d$ is odd then by Theorem~\ref{TAMainBody} we have $\Blo^{\rho,d} \sim_{\sM} T^{\Zig_\ell}(d,d)_\F\otimes\cC_1$, and so to every $\bla\in\Par^J(d)$, we have the irreducible $\Blo^{\rho,d}$-supermodule $D(\bla)$ of type $\Qtype$ corresponding to $L^{\Zig_\ell}_{n,d}(\bla)_\F\boxtimes U_1$ under the Morita superequivalence.

\begin{Theorem} \label{T190924} 
We have 
$\Irr(\Blo^{\rho,d})=\{D(\bla)\mid \bla\in\Par^J(d)\}.$ 
Moreover, for each $\bla\in\Par^J(d)$, we have:
\begin{enumerate}
\item[{\rm (i)}] $D(\bla)$ is of type of type $\Mtype$ if $|\rho|-h(\rho)+d$ is even, and of type of type $\Qtype$ if $|\rho|-h(\rho)+d$ is odd;
\item[{\rm (ii)}] $D(\bla)=\beta(L_{\rho,d}(\bla)_\F)$. 
\end{enumerate} 
\end{Theorem}
\begin{proof}
The first statement and part (i) come from Theorem~\ref{TAMainBody}. Part (ii)comes from the definition of the modules $D(\bla)$, the definition of the bijection $\beta$ in \S\ref{SSGGFragm}, and the definition of the modules $L_{\rho,d}(\bla)_\F$ in Corollary~\ref{CIrrHTheta}.  
\end{proof}

Let $\bmu=(\mu^{(0)},\dots,\mu^{(\ell-1)})\in\Comp^J(n,d)$. Recalling (\ref{EGGBla}),  (\ref{EGGW}) and (\ref{EGG!}), we have the Gelfand-Graev word
$$
\ggw^\bmu=\ggw^{\mu^{(0)}_1,0}\cdots \ggw^{\mu^{(0)}_n,0}\ \cdots\ \ggw^{\mu^{(\ell-1)}_1,\ell-1}\cdots \ggw^{\mu^{(\ell-1)}_n,\ell-1}
$$
with 
$$
\ggw^\bmu!=\prod_{j\in J}\prod_{r=1}^n\ggw^{\mu^{(j)}_r,j}!
=\prod_{j\in J}\prod_{r=1}^n((2\mu^{(j)}_r)!)^{\ell-j}(\mu^{(j)}_r!)^{2j+1}.
$$
Recalling (\ref{EIRho}), we now have 
\begin{equation}\label{EIRhoGGW!}
(\bi_\rho\ggw^\bmu)!=\bi_\rho!\,\ggw^\bmu!=a_1!\cdots a_t!\prod_{j\in J}\prod_{r=1}^n((2\mu^{(j)}_r)!)^{\ell-j}(\mu^{(j)}_r!)^{2j+1}.
\end{equation}

Recalling the notation (\ref{EBiBar}), we have the non-divided power word $\bar\bi_\rho\bar\ggw^\bmu$.

\begin{Theorem} \label{TGGFragment} 
Let $\theta=\sum_{i\in I}m_i\al_i$ be of height $n$ with with $\rho=\rho(\theta)$ and $d=d(\theta)$, and $\bi\in I^\rho$. 
For each $\bla\in\Par^J(d)$ and $\bmu\in\Comp^J(n,d)$, we have  
$$\dim e(\bi\bar\ggw^\bmu)D(\bla)= 2^{\lceil(n-m_0)/2\rceil}\,(\dim 1_\bi H_{\rho,\F}1_{\bi_\rho})\,K_{\bla,\bmu,\F}\prod_{j\in J}\prod_{r=1}^n((2\mu^{(j)}_r)!)^{\ell-j}(\mu^{(j)}_r!)^{2j+1}.$$
In particular, 
\begin{align*}
\dim e(\bar\bi_\rho\bar\ggw^\bmu)D(\bla)&
=a_1!\cdots a_t!\,2^{\lceil(n-m_0)/2\rceil}K_{\bla,\bmu,\F}\,\prod_{j\in J}\prod_{r=1}^n((2\mu^{(j)}_r)!)^{\ell-j}(\mu^{(j)}_r!)^{2j+1}.
\end{align*}
\end{Theorem}
\begin{proof}
By Lemma~\ref{L!}, Corollary~\ref{CIrrHTheta}(ii),  and Theorems~\ref{TWtHWtT},\,\ref{T190924}(ii), we have 
\begin{align*}
\dim e(\bar\bi_\rho\bar\ggw^\bmu)D(\bla)&=(\bi_\rho\ggw^\bmu)!\,2^{\lceil(n-m_0)/2\rceil}\dim 1_{\bi_\rho\ggw^\bmu} L_{\rho,d}(\bla)_\F
\\
&=(\bi_\rho\ggw^\bmu)!\,2^{\lceil(n-m_0)/2\rceil}K_{\bla,\bmu,\F}
\\
&=a_1!\cdots a_t!\,2^{\lceil(n-m_0)/2\rceil}K_{\bla,\bmu,\F}\prod_{j\in J}\prod_{r=1}^n((2\mu^{(j)}_r)!)^{\ell-j}(\mu^{(j)}_r!)^{2j+1}.
\end{align*}
The first formula follows from the second one using Lemmas~\ref{LBiBiRho} and \ref{L!}.
\end{proof}

\begin{Remark} 
{\rm 
Note that $\dim 1_\bi H_{\rho,\F}1_{\bi_\rho}$ appearing in the theorem is easy to compute. Indeed, 
$\dim 1_\bi H_{\rho,\F}1_{\bi_\rho}=a_1!\cdots a_t!\dim 1_\bi H_{\rho,\F}1_{\bar\bi_\rho}$ thanks to \cite[Lemma 3.2.4]{KlLi}, and $\dim 1_\bi H_{\rho,\F}1_{\bar\bi_\rho}$ is computed  in terms of standard tableaux using \cite[Theorem 3.1.31]{KlLi}. 
}
\end{Remark}

\subsection{Remarks on $A_\ell$ vs. $\Zig_\ell$}
\label{SSRegrRem}
Throughout this subsection we work over ${\mathbb L}=\F$ or $\K$ (and  drop ${\mathbb L}$ from all the indices). 
Let $H_\theta$ be a RoCK block (of cyclotomic quiver Hecke superalgebras) with $d=d(\theta)$.

Let $A_\ell$ be the regrading of $\Zig_\ell$ defined in (\ref{EAEll}). (Note that, annoyingly, $A_\ell$ was denoted $\Zig_\ell$ in \cite{KlLi}). It is a somewhat non-trivial fact which follows from \cite[Proposition 3.2.28]{KIS} that the wreath superproduct $W_d^{A_\ell}$ is isomorphic to a regrading of the wreath superproduct $W_d^{\Zig_\ell}$. 

One of the main results of \cite{KlLi} is that  
the graded ${\mathbb L}$-superalgebras $H_{\theta}$ and $W_{d}^{A_\ell}$ are graded Morita superequivalent as long as ${\mathbb L}=\K$ or ${\mathbb L}=\F$ and $p>d$. Since a graded superalgebra is graded Morita superequivalent to its regrading by \S\ref{SSRegr}, this result also follows immediately from  Corollary~\ref{PWreathMorita} and the isomorphism mentioned in the previous paragraph. 

It has been conjectured in \cite[Conjecture 3]{KlLi} that in general  $H_{\theta}$ and $T^{A_\ell}(d,d)$ are graded Morita superequivalent. In Corollary~\ref{CMainField}, we prove instead that $H_{\theta}$ and $T^{\Zig_\ell}(d,d)$ are graded Morita superequivalent. In view of this result, Conjecture 3 of \cite{KlLi} is in fact {\em false}, since $T^{A_\ell}(d,d)$ and $T^{\Zig_\ell}(d,d)$ are {\em not} in general graded Morita superequivalent (and not even Morita equivalent). The last statement follows from the fact that the irreducible modules of $T^{A_\ell}(d,d)$ and $T^{\Zig_\ell}(d,d)$ have the same dimensions (being inflations of the irreducible modules over the common quotient algebra appearing in the right hand side of (\ref{EDegZero})), while the dimensions of the algebras $T^{A_\ell}(d,d)$ and $T^{\Zig_\ell}(d,d)$ are in general different, so these algebras in general have different Cartan invariants.

Thus, while for the case where $p>d$ and $p=0$ (`abelian defect case', see \cite{KlLi}) one can get away without regrading, for the general case regrading is crucial.

\end{document}